\documentclass[11pt]{amsart}
\usepackage{amsmath,ifthen, amsfonts, amssymb,
srcltx, amsopn, mathrsfs} 
\usepackage{stmaryrd}

\usepackage[colorlinks,citecolor=blue,linkcolor=blue]{hyperref}
\usepackage{graphicx, overpic}

\usepackage[small]{caption}
\usepackage{tikz-cd}

 \usepackage{enumerate}
 \usepackage{microtype}
 \usepackage[nameinlink]{cleveref}

\newcommand{\showcomments}{yes}
\renewcommand{\showcomments}{no}

\newsavebox{\commentbox}
\newenvironment{com}%
{\ifthenelse{\equal{\showcomments}{yes}}%a
{\footnotemark
        \begin{lrbox}{\commentbox}
        \begin{minipage}[t]{1.25in}\raggedright\sffamily\tiny
        \footnotemark[\arabic{footnote}]}
{\begin{lrbox}{\commentbox}}}%
{\ifthenelse{\equal{\showcomments}{yes}}%
{\end{minipage}\end{lrbox}\marginpar{\usebox{\commentbox}}}
{\end{lrbox}}}

\theoremstyle{plain}
\newtheorem{thm}{Theorem}[section]
\newtheorem{lem}[thm]{Lemma}

\newtheorem{cor}[thm]{Corollary}
\newtheorem{conj}[thm]{Conjecture}
\newtheorem{prop}[thm]{Proposition}

\newtheorem*{namedtheorem}{\theoremname}
\newcommand{\theoremname}{testing}
\newenvironment{named}[1]{\renewcommand{\theoremname}{#1}\begin{namedtheorem}}{\end{namedtheorem}}

\theoremstyle{definition}
\newtheorem{defn}[thm]{Definition}
\newtheorem{rem}[thm]{Remark}

\newtheorem{conv}[thm]{Convention}

\newtheorem{question}[thm]{Question}
\newtheorem{prob}[thm]{Problem}

\numberwithin{equation}{section}

\DeclareMathOperator{\Aut}{Aut}

\DeclareMathOperator{\stab}{Stab}

\DeclareMathOperator{\arccosh}{arccosh}

\DeclareMathOperator{\len}{len}

\newcommand{\dist}{\textup{\textsf{d}}}

\newcommand{\field}[1]{\mathbb{#1}}
\newcommand{\integers}{\ensuremath{\field{Z}}}

\newcommand{\naturals}{\ensuremath{\field{N}}}

\newcommand{\hyperbolic}{\ensuremath{\field{H}}}

\newcommand{\ZZ}{\integers}
\newcommand{\HH}{\hyperbolic}
\renewcommand{\setminus}{\smallsetminus}

\newcommand{\neb}{\mathcal N}
\newcommand{\from}{\colon} % As in $f \from X \to Y$

\DeclareMathOperator{\Isom}{Isom}

\newcommand{\boundary}   {{\ensuremath \partial}}
\newcommand{\bdy}{\boundary}

\newcommand{\systole}[1]{\ensuremath{\vert\!\vert #1 \vert\!\vert}}
\newcommand{\stabletrans}[1]{\ensuremath{\llbracket #1 \rrbracket}}
\newcommand{\nclose}[1]{\ensuremath{\langle\!\langle#1\rangle\!\rangle}}

\newcommand{\diameter}{\operatorname{diam}}
\newcommand{\hull}{\operatorname{hull}}

\begin{document}

\title[Cubulating random quotients of hyperbolic cubulated groups]{Cubulating random quotients \\ of hyperbolic cubulated groups}

\begin{abstract}
We show that low-density random quotients of cubulated hyperbolic groups are again cubulated (and hyperbolic).
Ingredients of the proof include  cubical small-cancellation theory, the exponential growth of conjugacy classes,
and the statement that hyperplane stabilizers grow exponentially more slowly than the ambient cubical group.
\end{abstract}

\author{David Futer}
\address{Dept. of Mathematics \\ Temple University \\ Philadelphia, PA 19147 \\ USA}
\email{dfuter@temple.edu}
\author{Daniel T. Wise}
           \address{Dept. of Math. \& Stats.\\
                    McGill Univ. \\
                    Montreal, QC, Canada H3A 0B9 }
           \email{wise@math.mcgill.ca}
% \subjclass[2010]{20F67, 20F65, 20P05,  20E06}
% \keywords{CAT(0) cube complexes, graphs of groups}
\date{\today}
\thanks{Research supported by NSF and NSERC}

\maketitle

\section{Introduction}

Gromov introduced the \emph{density model} of random groups \cite[Chapter 9]{Gromov93}. Given a free group $F$, let $S_{\ell}(F)$ be the set of length $\ell$ words. For a density $d \in (0,1)$, choose a set $R \subset S_\ell$ by selecting 
$\lfloor | S_\ell |^{d} \rfloor$ elements from $S_\ell$, uniformly and independently. The associated \emph{random group at density $d$} is the quotient $F/ \nclose{R}$. Thus, for a  rank $r$ free group, a density $d$ random quotient has the form $\langle a_1, \ldots, a_r \mid g_1, \ldots, g_k \rangle$, where $k \sim (2r-1)^{d\ell}$ and the $g_i$ are  independently chosen random words.
\begin{com}
{\bf \normalsize COMMENTS\\}
ARE\\
SHOWING!\\
\end{com}

Gromov proved that with overwhelming probability as $\ell \to \infty$, random groups at density $d$ are hyperbolic when $d< \frac{1}{2}$, and trivial or $\integers_2$ when $d > \frac{1}{2}$. 
Ollivier proved that the same phase transition at density $\frac{1}{2}$ occurs in quotients of any torsion-free hyperbolic group \cite{Ollivier:PhaseTransition}. 
While Gromov's intention was perhaps to illustrate the ubiquity of hyperbolic groups, his construction initiated a fertile topic of study. For instance,
\.Zuk showed that random groups at density $d > \frac{1}{3}$ satisfy Kazhdan's property (T), with overwhelming probability \cite{Zuk03}. See  Kotowski and Kotowski \cite{KotowskiSquared} for full details, and Ollivier \cite{Ollivier05invitation} for an excellent survey of related topics.
Very recently, Ashcroft extended this result, proving that random quotients of a hyperbolic group at density $d > \frac{1}{3}$ have property  (T) \cite{Ashcroft:PropertyT}.

The property of being \emph{cocompactly cubulated} -- acting properly and cocompactly on a CAT(0) cube complex -- can be viewed as a strong negation of property (T) \cite{NibloRoller98}. In this direction, a simple computation shows  that with overwhelming probability as $\ell \to \infty$, random groups at density $d<\frac{1}{12}$ satisfy 
 the $C'(\frac16)$ small-cancellation condition, hence are cocompactly cubulated by  \cite{WiseSmallCanCube04}.
Ollivier and Wise  \cite{OllivierWiseDensity} showed that the same conclusion holds at density $d < \frac{1}{6}$. A series of papers by Mackay and Przytycki \cite{MackayPrzytycki}, Montee \cite{Montee:Cubulation}, and Ashcroft \cite{Ashcroft:Cubulation} produced a nontrivial action on a a CAT(0) cube complex at density $d < \frac{1}{4}$.
 See also Odrzyg\'o\'zd\'z  \cite{Odrzygozdz} and Duong \cite{Duong:Thesis} for cubulation results in the square probability model.

The purpose of this text is to extend these cubulation results from quotients of a free group (the fundamental group of a graph) to quotients
 of a cubulated hyperbolic group. 
%     Our result is not definitive, but does succeed in showing that cubulation persists.
Interestingly, our density statement is entangled with the relationship between
the growth of a hyperbolic group $G$ and the growth of hyperplane stabilizers in $G$.

\subsection{Main result}
Our main theorem is the following:

\begin{thm}\label{Thm:main}
Let $G=\pi_1X$, where $X$ is a compact non-positively curved cube complex, and suppose that $G$ is hyperbolic.
Let $b$ be the growth exponent of $G$ with respect to the universal cover $\widetilde X$. 
Let $a$ be the maximal growth exponent of a stabilizer of an essential hyperplane of $\widetilde X$.
Let $k\leq  e^{c \ell}$, where
\[ c \:<  \: \min \left\{ \frac{(b-a)}{20}, \,  \frac{b}{41} \right\} .
\]
Then with overwhelming probability as $\ell \to \infty$, for any set of conjugacy classes
$[g_1], \ldots, [g_k]$ with translation length $|g_i| \leq \ell$,
the group $G/ \nclose{g_1, \ldots, g_k}$ is hyperbolic and is the fundamental group of a compact, non-positively curved cube complex.
\end{thm}

In the above theorem, $|g_i|$ is the translation length of $g_i$ acting on the universal cover $\widetilde X$; see \Cref{Def:TransLength}. The growth exponent of $G$ with respect to $\widetilde X$ is a constant $b > 0$ such that the number of $G$--orbit points in an $\ell$--ball in $\widetilde X$ is approximately $e^{b \ell}$. See \Cref{Def:Growth} and \Cref{Thm:ConjugacyGrowth} for a more precise characterization. The growth exponent of a hyperplane stabilizer is defined analogously.

A hyperplane $\widetilde{U} \subset \widetilde{X}$ is \emph{essential} if $\widetilde{X}$ is not contained in a finite neighborhood of $\widetilde{U}$. A typical inessential hyperplane arises  when $\widetilde{X} \cong \widetilde Y \times \widetilde F$, where $\widetilde F$ is finite and nontrivial. After subdividing, a compact non-positively curved cube complex $X$ always deformation retracts to $X'$, where all hyperplanes of $X'$ are essential. 

The model of randomness employed in \Cref{Thm:main} is that we are sampling uniformly from the set of conjugacy classes whose translation length on $\widetilde X$ is at most $\ell$. This departs from Gromov's density model in two small ways: we are permitting relators whose translation length is less than $\ell$, and we are only counting one relator per conjugacy class. Ultimately, the exponential growth rate of balls in $\widetilde X$ is the same as the growth rate of spheres, and the growth of group elements is nearly the same as the growth of conjugacy classes (\Cref{Thm:ConjugacyGrowth}). Therefore, small variations in the model of randomness tend to have no effect on the conclusions one can reach about random quotients. See Ollivier (\cite[Sections 4--5]{Ollivier:PhaseTransition} and \cite[Section I.2.c]{Ollivier05invitation}) for a detailed and axiomatic discussion of this phenomenon.

The conclusion of  \Cref{Thm:main} --- that $\overline G = G/ \nclose{g_1, \ldots, g_k}$ is hyperbolic and cocompactly cubulated --- has powerful consequences. A theorem of Agol~\cite{Agol:VirtualHaken} implies $\overline G$ is \emph{virtually special}, meaning $\overline G$ virtually embeds into a right-angled Artin group. Then, by a theorem of Haglund and Wise~\cite{HaglundWiseCoxeter}, all quasiconvex subgroups of $\overline G$ are separable. Furthermore, $\overline G$ is linear over $\ZZ$ \cite{HsuWiseGraphProducts, DavisJan00} and virtually surjects the free group $F_2$ \cite{AntolinMinasyan:TitsAlternative}.

\subsection{Density and optimality of constants}
In  \Cref{Thm:main}, the density of the random presentation $G/ \nclose{g_1, \ldots, g_k}$ is $\frac{c}{b}$. Thus all densities in the conclusion of the theorem are at most $\frac{1}{41}$, and might be lower, depending on the value of $a$. 
This is considerably lower than the densities appearing in the theorems surveyed at the start of the Introduction.
The primary reason for needing the low density is that \Cref{Thm:main} is proved by establishing that random quotients of $G$ satisfy the $C'(\frac{1}{20})$ cubical small cancellation condition. (See \Cref{Def:Piece} for the definition and \Cref{Thm:C'20Proper} for the precise statement that $C'(\frac{1}{20})$ plus several mild hypotheses implies cubulation of the quotient.) Indeed, the probabilistic pigeonhole principle \cite[p. 31]{Ollivier05invitation} implies that at any density larger than $\frac{1}{40}$ there will almost surely be pieces that fellow-travel for more than $\frac{1}{20}$ of their length. Thus our density hypotheses are nearly optimal for ensuring  $C'(\frac{1}{20})$ small cancellation.

Strengthening \Cref{Thm:main} beyond density $\frac{1}{40}$ would require one of two improvements. 
First, one could attempt to strengthen \Cref{Thm:C'20Proper} and establish the cubulation of $C'(\alpha)$ quotients for some parameter $\alpha > \frac{1}{20}$. Second, one could move away from small cancellation theory entirely, for instance by employing isoperimetric inequalities for van Kampen diagrams as in the work of Ollivier and Wise~\cite{OllivierWiseDensity}. This is also the approach employed in Ollivier's work on quotients of 
torsion-free hyperbolic groups~\cite[Thm 3]{Ollivier:PhaseTransition}, which applies at all densities up to $\frac{1}{2}$ but only ensures the hyperbolicity of the quotient.

We remark that so long as $G \ncong \integers$,
the main theorem is non-vacuous, meaning $a < b$, because hyperplane stabilizers in $G$ have strictly lower growth exponents than $G$ itself. See  \cite[Thm 1.3]{DahmaniFuterWise} and \Cref{Thm:SubgroupGrowth} below. Thus we may always pick $c > 0$ in \Cref{Thm:main}. This leads to the following corollary:

\begin{cor}\label{Cor:LowDensity}
Let $X$ be a compact non-positively curved cube complex, such that $\pi_1 X$ is hyperbolic and nonelementary.
Then, at a sufficiently low density, generic quotients of $\pi_1 X$ are cocompactly cubulated and hyperbolic.
\end{cor}

For instance, if $G$ is a surface group and the hyperplane stabilizers are cyclic (which occurs in the standard cubulations of $G$), we have $a = 0$, hence random quotients of $G$ are cubulated and hyperbolic at density $\frac{1}{41}$. This is far lower than the density of $\frac{1}{6}$ at which the corresponding conclusion is known for quotients of free groups. This supports our belief that the constants in \Cref{Thm:main} can be improved considerably, especially for surface groups. See \Cref{Prob:SurfaceOptimalDensity} and \Cref{Prob:GeneralOptimalDensity}. 
%    \begin{com} We may wish to say something about van Kampen diagrams in \Cref{Prob:SurfaceOptimalDensity}. \end{com}

\subsection{Actions on other metric spaces}
The growth of a group $G$ is highly sensitive to the choice of metric space on which $G$ acts. 
For instance, hyperbolic manifold groups in dimension $n \geq 3$ admit canonical geometric actions on a hyperbolic space $\Upsilon = \HH^n$, but any cubulated group admits infinitely many distinct actions on non-isometric cube complexes. We may wish to study quotients of $G$ by some number of relators that are sampled with respect to length in $\Upsilon$ rather than a cube complex $\widetilde X$. Although lengths in $\widetilde X$ and $\Upsilon$ can be compared via an (equivariant) quasi-isometry, measuring lengths in the two spaces can lead to rather distinct samples of short words.

Nonetheless, there is an analogue of \Cref{Thm:main} for sampling words in $G$ with respect to an action on $\Upsilon$. To formulate the next theorem, we introduce the following non-standard quantification of quasiisometries. A \emph{$\lambda$--quasiisometry} is a coarsely surjective function $f \from  \widetilde X \to \Upsilon$ such that there exist positive constants $\lambda_1, \lambda_2, \epsilon$ with $\lambda_1 \lambda_2 = \lambda$, where
every pair of points $x,y \in X$ satisfies
\begin{equation}\label{Eqn:Lambda12}
\frac{1}{\lambda_1} \dist_{\widetilde X}(x,y) - \epsilon \: \leq \: \dist_{\Upsilon} (f(x), f(y)) \: \leq \: \lambda_2 \dist_{\widetilde X}(x,y) + \epsilon.
\end{equation}

The product $\lambda_1 \lambda_2 = \lambda$ remains unchanged if the metric on one of the spaces $\widetilde X$ or $\Upsilon$ is rescaled by a multiplicative constant. More generally, $\lambda \geq 1$ has the following meaning. If $G$ acts properly and cocompactly on both $\widetilde X$ and $\Upsilon$, then these actions induce pseudo-metrics $d_1, d_2$ on $G$ itself. We call these pseudo-metrics \emph{roughly similar} if they are related by a $G$--equivariant $1$--quasiisometry.
The space of rough similarity classes of pseudo-metrics on $G$ is itself an interesting metric space $\mathscr D(G)$, studied topologically 
since the work of Furman~\cite{Furman:CoarseGeometricPerspective}, and metrically since the work of Reyes~\cite[Def 1.2]{Reyes:SpaceOfMetrics}.
In Reyes's natural metric on $\mathscr D(G)$, the distance between $[d_1]$ and $[d_2]$ is precisely $\log \lambda$ for the optimal constant $\lambda = \lambda_1 \lambda_2$ in a $G$--equivariant quasiisometry $\widetilde X \to \Upsilon$. See \cite{Reyes:SpaceOfMetrics,CantrellReyes:MarkedLengthSpectrum,BrodyReyes:HyperbolicCubulations}.

We are interested in quotients of $G$, where the conjugacy classes of relators are drawn uniformly from among all elements of $\Upsilon$--length less than $\ell$. At sufficiently low density, depending on $\lambda$, these quotients are again hyperbolic and cubulated.

\begin{thm}\label{Thm:OtherSample}
Let $G=\pi_1X$, where $X$ is a compact non-positively curved cube complex, and suppose that $G$ is hyperbolic. Suppose that $G$ also acts properly and cocompactly on a geodesic metric space $\Upsilon$, where every non-trivial element of $G$ stabilizes a geodesic axis. Suppose that there is a $G$--equivariant $\lambda$--quasiisometry   $\widetilde X \to \Upsilon$.

Let $b$ be the growth exponent of $G$ with respect to $\Upsilon$,  
and let $a$ be the maximal growth exponent in $\Upsilon$ of a stabilizer of an essential hyperplane of $\widetilde X$.
Let $k\leq  e^{c \ell}$, where

\[ c<\min \left\{ \frac{(b-a)}{20 \lambda}, \,  \frac{b}{40 \lambda +1} \right\} .
\]
Then with overwhelming probability as $\ell \to \infty$, for any set of conjugacy classes
$[g_1], \ldots, [g_k]$ with translation length $|g_i|_{{}_\Upsilon}  \leq \ell$,
the group $\overline{G} = G/ \nclose{g_1, \ldots, g_k}$ is hyperbolic and is the fundamental group of a compact, non-positively curved cube complex.
\end{thm}

As with \Cref{Thm:main}, this result is non-vacuous, because hyperplane stabilizers in $G$ have strictly lower exponential growth rates (with respect to $\Upsilon$) than $G$ itself.  
By \Cref{Thm:SubgroupGrowth} below, we have $b >a$, hence we may choose $c >0$.

As a special case, suppose that a group $G$ preserves a tiling of $\HH^2$ by regular right-angled pentagons. Let $\widetilde X$ be the square complex dual to the pentagonal tiling. Then every hyperplane is a line, hence hyperplane stabilizers are cyclic and have exponential growth rate $a = 0$. By a theorem of Huber \cite{Huber}, generalized by Margulis \cite{Margulis:Growth}, the growth exponent of $G$ with respect to $\HH^2$ is $b=1$. Furthermore, in \Cref{Prop:PentagonalQI}, we show that the optimal multiplicative constant in a $\lambda$--quasiisometry from $\widetilde X$ to $\HH^2$ is $\lambda \approx 1.5627$. Consequently,
 \Cref{Thm:OtherSample} has the following corollary.
 
 \begin{cor}\label{Cor:H2Action}
 Let $S$ be a hyperbolic surface tiled by regular right-angled pentagons, and let $G = \pi_1 S$. For a number $\ell \gg 0$, let $k \leq e^{ \ell / 63.51}$, and let
 $[g_1], \ldots, [g_k]$ be conjugacy classes in $G$, chosen uniformly at random from among those of $\widetilde S$--length at most $\ell$. Then with overwhelming probability as $\ell \to \infty$, the group $\overline{G} = G/ \nclose{g_1, \ldots, g_k}$ is hyperbolic and is the fundamental group of a compact, non-positively curved cube complex.
 \end{cor}

In particular, \Cref{Cor:H2Action} says that quotients of surface groups, with words sampled from a pentagonal hyperbolic metric, are cubulated and hyperbolic at density $\frac{1}{64}$. This is somewhat worse than sampling with respect to a cubical metric, where we obtain the same conclusion at density $\frac{1}{41}$; see the discussion after \Cref{Cor:LowDensity}. The multiplicative gap between these densities is essentially the constant $\lambda$.

Studying the effect of quasiisometry constants on density statements about quotients of $G$ has led us to conjecture that there exist cubulations of a hyperbolic manifold group with quasiisometry constant $\lambda$ arbitrarily close to $1$. Equivalently, the points of $\mathscr D(G)$ corresponding to cocompact hyperbolic structures are limit points of metrics coming from cocompact cubulations.
See \Cref{Conj:OptimalCubulateManifold} and the subsequent discussion. This conjecture has been recently proved by Brody and Reyes~\cite{BrodyReyes:HyperbolicCubulations}.

\subsection{Overview} 
\Cref{sec:growth} reviews several results about the growth of a hyperbolic group, including the existence of a growth exponent and the statement that infinite-index quasiconvex subgroups have a lower growth exponent than the ambient group.
\Cref{sec:cubical presentation background} reviews the definitions of cubical small-cancellation theory and proves \Cref{Thm:C'20Proper}, a non-probabilistic cubulation criterion for $C'(\frac{1}{20})$ small-cancellation quotients of a cubulated group. This criterion is of independent interest, and has already been used in the work of Jankiewicz and Wise~\cite{JankiewiczWise}.

The probabilistic arguments supporting the proof of \Cref{Thm:main} are contained in \Cref{Sec:Pieces}. In that section, we control the sizes of pieces in a generic cubical presentation and show that below a certain density, a cubical presentation is $C'(\frac{1}{20})$, hence the quotient group is hyperbolic and cubulated. All of the geometric arguments in that section are coarse in nature, and it becomes natural to work with a certain generalization of pieces called \emph{loose pieces}. The study of loose pieces also permits a translation between a $G$--action on a cube complex $\widetilde X$ and a $G$--action on a more general metric space $\Upsilon$. We undertake this translation in \Cref{Sec:OtherMetric}, where we prove \Cref{Thm:OtherSample}. Finally, in \Cref{Sec:Pentagonal}, we find the optimal multiplicative constants in a quasiisometry between a pentagonal tiling of $\HH^2$ and the dual cube complex, proving \Cref{Prop:PentagonalQI} and \Cref{Cor:H2Action}.

 In \Cref{Sec:Problems}, we collect some problems and questions motivated by these results.

\subsection{Acklowledgements}
We thank Eduardo Reyes and Sam Taylor for several enlightening conversations. We thank the referee for numerous minor corrections, and for suggesting a potential improvement to our results, as sketched in \Cref{Rem:ProbabilisticQITranslation} and \Cref{Ques:UpsilonTypicalPieces}.

\section{Growth}\label{sec:growth}

This section collects several definitions and results about the growth of hyperbolic groups. None of the results recalled here are original.

\begin{defn}[Translation lengths]\label{Def:TransLength}
Let $G$ be a group acting properly and cocompactly on a metric space $\Upsilon$, and let $g \in G$ be an infinite-order element. The \emph{translation length} of $g$ is defined to be $|g|_{{}_\Upsilon} = \inf \{ \dist(x, gx) : x \in \Upsilon \}$. The \emph{stable translation length} of $g$ is
\begin{equation*}%\label{Eqn:SystoleTrans}
\stabletrans{g}_{{}_\Upsilon}  = \lim_{n \to \infty} \frac{\dist(x, g^n x)}{n}, 
\end{equation*}
for an arbitrary $x \in \Upsilon$. It is a standard property of isometries of metric spaces that the limit exists and is independent of $x$ \cite[page 230]{BridsonHaefliger}. 
\end{defn}

Observe that each of $|g|_{{}_\Upsilon}$ and $\stabletrans{g}_{{}_\Upsilon}$ only depends on the conjugacy class $[g]$.
By triangle inequalities, $\stabletrans{g}_{{}_\Upsilon}   \leq  |g|_{{}_\Upsilon}$ for every $g$. In addition, when $\Upsilon$ is hyperbolic, there is a constant $C$ such that the following holds for every infinite-order $g \in G$:
\[
\stabletrans{g}_{{}_\Upsilon}  \: \leq \: |g|_{{}_\Upsilon} \: \leq \: \stabletrans{g}_{{}_\Upsilon} + C.
\]
Finally, if $g$ stabilizes a geodesic axis in $\Upsilon$ (as will typically be the case in our applications), we have $\stabletrans{g}_{{}_\Upsilon}   =  |g|_{{}_\Upsilon}$.

\begin{defn}[Growth]\label{Def:Growth}
Let $G$ be a finitely generated group acting properly and cocompactly on a metric space $\Upsilon$.
Fix a basepoint $x \in \Upsilon$ and a subset $H \subset G$.
The \emph{growth function} of $H$ with respect to $\Upsilon$ is the function $f_{H,\Upsilon}:\naturals \rightarrow \naturals$
defined by
\begin{equation*}%\label{Eqn:fH}
f_{H,\Upsilon}(n) = \# \left\{ h\in H: \dist_\Upsilon(x, \, hx)\leq n \right\} .
\end{equation*}

Since $G$ is a quotient of a finite-rank free group, and the action on $\Upsilon$ is proper, the growth function $f_{H,\Upsilon}$ is no larger than exponential. Thus it makes sense to consider the logarithm of $f$. The \emph{growth exponent} of $H$ with respect to $\Upsilon$ is
\begin{equation*}
\xi_H (\Upsilon) = \lim_{n \to \infty} \frac{ \log f_{H,\Upsilon}(n) }{n} ,
\end{equation*}
whenever the limit exists. We emphasize that the limit
depends a great deal on $\Upsilon$. However, triangle inequalities in $\Upsilon$ imply that $\xi_H(\Upsilon)$ is independent of the basepoint.  
\end{defn}

Many results in the literature are expressed in terms of the \emph{growth rate}  $\lambda_H(\Upsilon) = \lim \sqrt[n]{f_{H,\Upsilon}(n)} = e^{ \xi_H(\Upsilon) } $ instead of the growth exponent $\xi_H(\Upsilon)$. However, the two notions carry the same information.

The following result was proved by Dahmani and the authors~\cite[Thm 1.1]{DahmaniFuterWise}, and independently Matsuzaki, Yabuki, and Jaerisch \cite[Cor 2.8]{MYJ}.  See also \cite[Thm 1.3]{DahmaniFuterWise} for a statement that does not assume $G$ is hyperbolic, but does assume $\Upsilon$ is a CAT(0) cube complex and $H$ is a hyperplane stabilizer. 

\begin{thm}\label{Thm:SubgroupGrowth}
Let $G$ be a non-elementary hyperbolic group acting properly and cocompactly on a metric space $\Upsilon$.
Let $H$ be a quasiconvex subgroup of infinite index. Then  the growth exponents $\xi_H$ and $\xi_G$ exist, and
\[
\xi_H(\Upsilon) < \xi_G (\Upsilon).
\]
\end{thm}

The following theorem combines two results of Coornaert and Knieper \cite{Coornaert93, CoornaertKnieper2002}. Recall that a non-trivial element $g \in G$ is \emph{primitive} if $g \neq h^n$ for any $n > 1$. Primitive conjugacy classes are defined similarly.

\begin{thm}\label{Thm:ConjugacyGrowth}
Let $G$ be a nonelementary group acting properly and cocompactly on a $\delta$--hyperbolic metric space $\Upsilon$. Then, there exist positive constants $A,B, b, n_0$, where $b = \xi_G(\Upsilon)$, such that the following hold for $n \geq n_0$.
\begin{enumerate}[\:\: $(1)$]
\item\label{Itm:TotalGrowth} The total number $f_{G,\Upsilon}(n)$ of elements that translate a basepoint $x \in \Upsilon$ by distance at most $n$ satisfies
\[
A e^{bn} \leq f_{G,\Upsilon}(n) \leq B e^{bn} .
\]
\item\label{Itm:ConjugacyGrowth}
The number $p_n$ of primitive conjugacy classes of translation length
 at most $n$ satisfies
%\begin{equation}\label{Eqn:ConjGrowth}
\[
A \frac{e^{bn}}{n} \leq p_n \leq B e^{bn} .
\]
% \end{equation}
\end{enumerate}
\end{thm}

\begin{proof}
Conclusion \eqref{Itm:TotalGrowth} is due to Coornaert \cite[Thm 7.12]{Coornaert93}. Conclusion \eqref{Itm:ConjugacyGrowth} is due to Coornaert and Knieper \cite[Thm 1.1]{CoornaertKnieper2002}.
\end{proof}

\begin{rem}\label{Rem:GroupoidGrowth}
One consequence of \Cref{Thm:ConjugacyGrowth}.\eqref{Itm:TotalGrowth} is that when $\Upsilon$ is a cell complex, we get the same upper and lower bounds (with a modified upper constant $B$) on the number of vertices in a metric ball in $\Upsilon$, where the vertices being counted are not required to be in the $G$--orbit of the basepoint $x$. This holds because of the cocompactness of the $G$--action.
\end{rem}

\begin{rem}\label{Rem:NonPrimitive}
The upper bound on the number $p_n$ of primitive classes in \Cref{Thm:ConjugacyGrowth}.\eqref{Itm:ConjugacyGrowth} implies that the number of
\emph{non-primitive} conjugacy classes of translation length at most $n$ is bounded above by
\[
\sum_{j = 1}^{n/2} B e^{bj} < n B e^{bn/2} \ll p_n.
\]
Combining this fact with the lower bound of \Cref{Thm:ConjugacyGrowth}.\eqref{Itm:ConjugacyGrowth}  implies that the proportion of non-primitive conjugacy classes is at most
\[
\frac{ n^2 B}{A} e^{-bn/2} .
\]
Hence, for large $n$, the non-primitive conjugacy classes form a vanishingly small proportion of all conjugacy classes up to length $n$.
\end{rem}

\section{Cubical presentations and small-cancellation theory}\label{sec:cubical presentation background}

This section reviews some definitions and results about cubical small-cancellation theory. Our primary reference is Wise~\cite[Chapters 3--5]{WiseIsraelHierarchy}. We also prove \Cref{Thm:C'20Proper}, a properness criterion that follows from \cite[Thm 5.44]{WiseIsraelHierarchy} but is easier to apply.

\subsection{Cubical presentations and pieces}

\begin{defn}[Cubical presentation]\label{Def:CubicalPres}
A \emph{cubical presentation} $X^* = \langle X | Y_1, \ldots, Y_m \rangle$ consists of non-positively curved cube complexes $X$ and $Y_i$, and a set of local isometries $Y_i \looparrowright X$.  The cubical presentation $X^*$ corresponds to a topological space, also denoted $X^*$, consisting of $X$ with a cone on each $Y_i$. Accordingly, we call each $Y_i$ a \emph{cone} of $X^*$.
\end{defn}

See \Cref{Fig:CubicalPres} for an example. In this paper, it will always be the case that $\pi_1 Y_i \cong \integers$. However, this assumption is not present elsewhere in the literature.

For a hyperplane $\widetilde{U}$ of $\widetilde{X}$, the \emph{carrier} $N(\widetilde{U})$ is the union of all closed cubes intersecting $\widetilde{U}$.

The \emph{systole} $\systole{X}$ is the infimal length of an essential combinatorial closed path in $X$.
In terms of \Cref{Def:TransLength}, $\systole{X}$ is the smallest translation length of a non-trivial element of $\pi_1 X$ acting on $\widetilde X$.

\begin{defn}[Pieces]\label{Def:Piece}
Let $X^* = \langle X | Y_1, \ldots, Y_m \rangle$ be a cubical presentation. 
A \emph{cone-piece} of $X^*$ between $Y_i$ and $Y_j$ is a component of $\widetilde{Y}_i \cap \widetilde{Y}_j$, for some choice of lifts of $\widetilde Y_i$ and $\widetilde Y_j$ to $\widetilde X$, excluding the case where $\widetilde Y_i = \widetilde Y_j$.

A \emph{wall-piece} of $X^*$ in $Y_i$ is a non-empty intersection of $\widetilde{Y}_i \cap N(\widetilde{U})$, where $\widetilde{U}$ is a hyperplane that is disjoint from  $\widetilde{Y}_i$.

Given a constant $\alpha > 0$, we say that $X^*$ satisfies the $C'(\alpha)$ \emph{small-cancellation condition} if for every cone-piece or wall-piece $P$ involving $Y_i$, we have $\diameter (P) < \alpha \systole{ Y_i }$. 
\end{defn}

When pieces are small, the topological space $X^*$ satisfies a number of pleasant properties. 

\begin{lem}[\hbox{\cite[Thm 3.32 and Thm 4.1]{WiseIsraelHierarchy}}]\label{Lem:1/12 lift}
If $X^*$ is $C'(\frac{1}{12})$, then every cone $Y_i \looparrowright X^*$ lifts to an embedding
  $Y_i \hookrightarrow \widetilde{X^*}$.
\end{lem}

In a generalization of ordinary small-cancellation theory, small pieces guarantee the persistence of hypebolicity in $X^*$:

\begin{lem}[\hbox{\cite[Lem 3.70 and Thm 4.7]{WiseIsraelHierarchy}}]\label{lem:1/14 hyperbolic}
If $\pi_1 X$ is hyperbolic and $X^*$ is compact and $C'(\frac{1}{14})$, then $\pi_1X^*$ is hyperbolic.
\end{lem}

The main result of this section is the following theorem, which  guarantees that $\pi_1 X^*$ acts properly on a CAT(0) cube complex.

\begin{thm}\label{Thm:C'20Proper}
Let $X^* = \langle X \mid Y_1, \ldots, Y_k \rangle$ be a $C'(\frac{1}{20})$ cubical presentation. Suppose that $X$ is compact, and every $Y_i$ is compact and deformation retracts to a closed--geodesic. 
In addition, suppose that for every hyperplane  $U \subset Y_i$,  the carrier $N(U)$ is embedded,  the complement $Y_i \setminus U$ is contractible, and  $\diameter (N(U)) < \frac{1}{20} \systole{Y_i}$.
 
Then  
$\pi_1 X^*$ acts properly and cocompactly on a CAT(0) cube complex dual to a wallspace structure on $\widetilde{X^*}$ satisfying the $B(8)$ condition.

Moreover, if $\overline g \in \pi_1 X^* \setminus \{1\}$ stabilizes a cell of the dual cube complex, then $\overline g$ is the image of an element $g \in \pi_1 X$
such that a conjugate of some $\pi_1 Y_i$ lies in $\langle g \rangle$. In particular, if each $\pi_1 Y_i$ is maximal cyclic, then  $\pi_1X^*$ acts freely on $\widetilde{X^*}$.
\end{thm}

\Cref{Thm:C'20Proper} will be proved with the aid of \cite[Thm 5.44 and Cor 5.45]{WiseIsraelHierarchy}. Setting up the proof using those results requires a number of auxiliary definitions, beginning with the $B(8)$ condition.

\subsection{Wallspace small-cancellation conditions}
We will construct a wallspace structure for $\widetilde{X^*}$. To do so, we will define a wallspace on each $Y_i$, whose walls are equivalence classes of hyperplanes in $Y_i$. This will generate a wallspace structure
for $\widetilde{X^*}$ whose walls are equivalence classes of hyperplanes generated by the equivalence relation
 fostered by the wallspace structures on the lifts $Y_i\hookrightarrow \widetilde{X^*}$.

\begin{figure}
\includegraphics[width=5in]{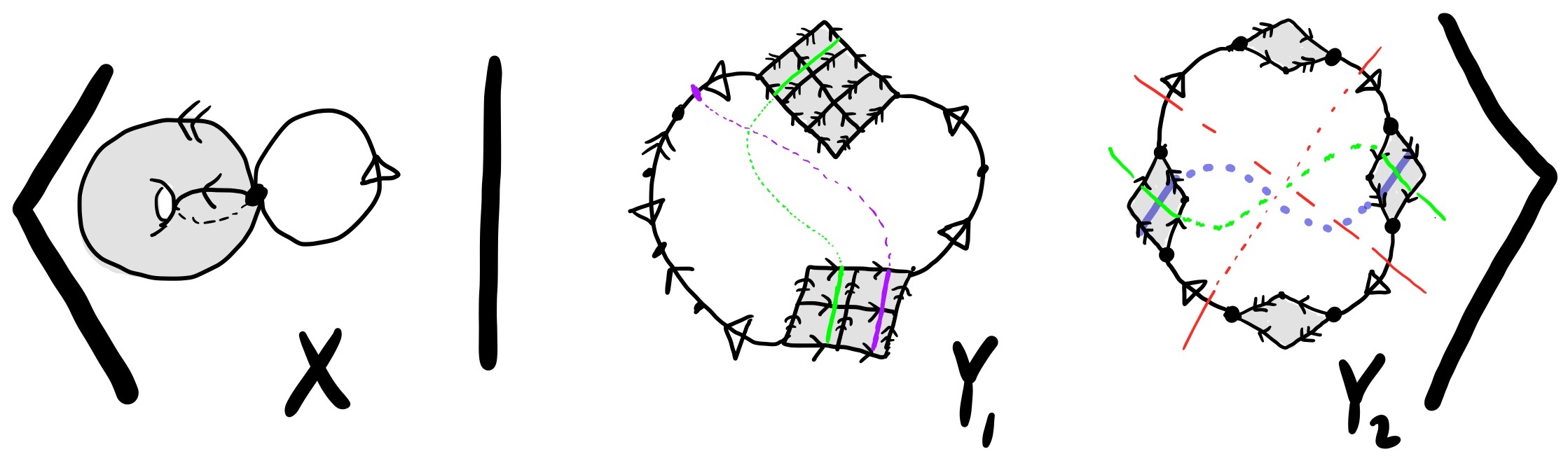}
\caption{A cubical presentation illustrating conditions \eqref{B6:wallspace} and \eqref{B6:aut} of \Cref{Def:B6}. In this example, as in \Cref{Thm:C'20Proper}, we have $\pi_1 Y_i \cong \ZZ$ for each $Y_i$, and the wallspace structure on $Y_i$ has two diametrically opposed hyperplanes in each wall. Unlike the setting of \Cref{Thm:C'20Proper}, $\pi_1 X$ is not hyperbolic in this example.}
\label{Fig:CubicalPres}
\end{figure}

\begin{defn}[$B(6)$ and $B(8)$ conditions]\label{Def:B6}
Suppose that $X^*= \langle X \mid Y_1, \ldots, Y_k \rangle$ is a $C'(\frac{1}{12})$ cubical presentation. A \emph{$B(6)$ structure on $X^*$} consists of a wallspace structure on each $Y_i$, which is preserved by its automorphisms, and satisfies a certain small-cancellation condition (see \Cref{Fig:CubicalPres} for an example). More precisely:

\begin{enumerate}[\:\: ($1$)]
\item \label{B6:wallspace} The collection of hyperplanes of each $Y_i$ is partitioned into classes satisfying the following conditions.
No two hyperplanes in the same class cross or osculate; in particular, the carrier of each hyperplane $u_j$ embeds. 
The union $U = \cup u_j$ of the hyperplanes in a class is called a \emph{wall}.
Furthermore, for each wall $U$, there are \emph{halfspaces} $\overleftarrow{U},\overrightarrow{U}$ such that $Y_i = \overleftarrow{U} \cup \overrightarrow{U}$ and $U = \overleftarrow{U} \cap \overrightarrow{U}$. 
\item \label{B6:aut} 
$\Aut(Y_i\rightarrow X)$ preserves the above \emph{wallspace structure} on $Y_i$.
\item \label{B6:homotopy} If $P$ is a path in $Y_i$ that is the concatenation of at most  7 piece-paths
 and $P$ starts and ends on the carrier
$N(U)$ of a wall $U$, then $P$ is path-homotopic into $N(U)$. 
\end{enumerate}

The \emph{$B(8)$ condition} is defined by replacing \eqref{B6:homotopy} with the stronger condition:

\begin{enumerate}[\:\: ($1'$)]
\setcounter{enumi}{2}
\item\label{Itm:B8} If $P$ is a path that is the concatenation of at most  8 piece-paths
 and $P$ starts and ends on the carrier
$N(U)$ of a wall then $P$ is path-homotopic into $N(U)$.
\end{enumerate}
\end{defn}

In the above, 
a \emph{piece-path} in $Y$ is a path in a piece of $Y$.
Meanwhile, $\Aut(Y_i\rightarrow X)$ is the group of automorphisms $\phi:Y_i\to Y_i$ such that  
\begin{tikzcd}[row sep = tiny, column sep = small]
Y_i \arrow[rr] \arrow[rd] & &Y_i \arrow[dl] \\
& X 
\end{tikzcd}
commutes.

\subsection{The properness criterion}

We now collect some terminology that will be needed to state and apply \Cref{Thm:cubulating X star}.

In our usage, geodesics are globally distance realizing.
By contrast, a \emph{closed--geodesic} $w \rightarrow Y$ in a non-positively curved cube complex is a combinatorial immersion
of a circle  whose universal cover $\widetilde w$  lifts to a combinatorial geodesic $\widetilde w\rightarrow \widetilde Y$
 in the universal cover of $Y$. 
%     \begin{com} Cautionary non-example: a M\"obius band of width $1$ does not deformation retract to a closed--geodesic.\end{com} 
 We emphasize that (the image of) a closed--geodesic is not a geodesic in $Y$, because it is not distance realizing.

Let $U$ be a hyperplane and $v$ a $0$--cube. We say that $U$ is \emph{$m$--proximate} to $v$ if there is a path $P=P_1\cdots P_m$
such that each $P_i$ is either a single edge or a path in a piece,
and $v$ is the initial vertex of $P_1$ and $U$ is dual to an edge in $P_m$. 
A wall is \emph{$m$--proximate} to $v$ if has a hyperplane that is $m$--proximate to $v$.

A hyperplane $u$ of a cone $Y$ of $X^*$ is \emph{piecefully convex} if the following holds:
For any path $\xi\rho \rightarrow Y$ with endpoints on $N(u)$,
if $\xi$ is a geodesic and $\rho$ is either trivial or lies in a piece of $Y$ containing an edge dual to $u$,
then $\xi\rho$ is path-homotopic in $Y$ to a path $\mu\rightarrow N(u)$.

We will verify pieceful convexity via the following criterion.

\begin{rem}[\hbox{\cite[Rem~5.43]{WiseIsraelHierarchy}}]\label{rem:pieceful convexity}
Let $M$ be the maximal diameter of any piece of $Y$ in $X^*$.
Then a hyperplane $u$ of  $Y$ is piecefully convex provided its carrier $N=N(u)$ satisfies:
$\dist_{\widetilde Y}(g\widetilde N,\widetilde N)>M$ 
for any translate $g\widetilde N\neq \widetilde N\subset \widetilde Y$.
\end{rem}

The following is a simplified restatement of \cite[Thm~5.44]{WiseIsraelHierarchy}, which also incorporates
 \cite[Lem~3.70 and Cor~5.45]{WiseIsraelHierarchy}.

\begin{thm}\label{Thm:cubulating X star}
Suppose that $X^* = \langle X \mid \{Y_i\} \rangle$ satisfies the following hypotheses:
\begin{enumerate}[\:\: $(1)$]
\item\label{Itm:B6} $X^*$ is $C'(\frac{1}{14})$ and satisfies the $B(6)$ condition.
%% $C'(1/14)$ is included to guarantee short innerpaths; see \cite[Lem~3.70]{WiseIsraelHierarchy}. We do not want to talk about short innerpaths directly.
\item\label{Itm:Pieceful} Each hyperplane of each cone $Y_i$ is piecefully convex.
\item\label{Itm:CutWall} For each $Y_i$, each infinite order element of $\Aut(Y_i)$ is cut by a wall of $Y_i$.
\item \label{Itm:Separation} Let $\kappa\rightarrow Y_i$ be a geodesic with endpoints $p,q$.
Let $u_1$ and $u_1'$ be distinct hyperplanes in the same wall of $Y_i$.
Suppose $\kappa$ traverses a $1$--cell dual to $u_1$, 
and either $u_1'$ is $1$--proximate to $q$ or $\kappa$ traverses a $1$--cell dual to $u_1'$.
  Then there is a  wall $U_2$ in $Y_i$ that separates $p,q$  but is not $2$--proximate to $p$ or $q$.
  \end{enumerate}
Then $\pi_1X^*$ acts with torsion stabilizers on the dual cube complex of the $B(6)$ structure.
\end{thm}
We refer to \Cref{fig:KappaAndFriends} for a depiction of the notation in Hypothesis~\eqref{Itm:Separation}. We will now use \Cref{Thm:cubulating X star} to prove \Cref{Thm:C'20Proper}.

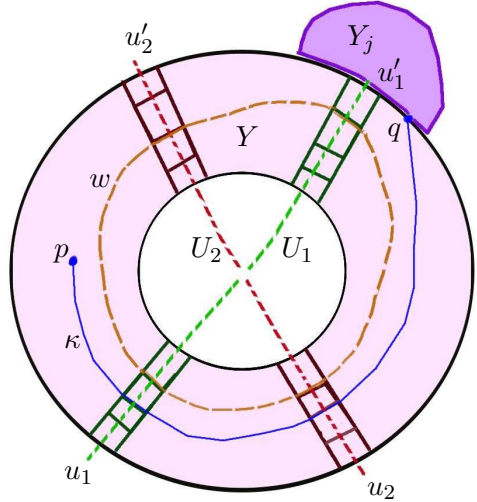
\begin{figure}
\begin{overpic}[width=2.5in]{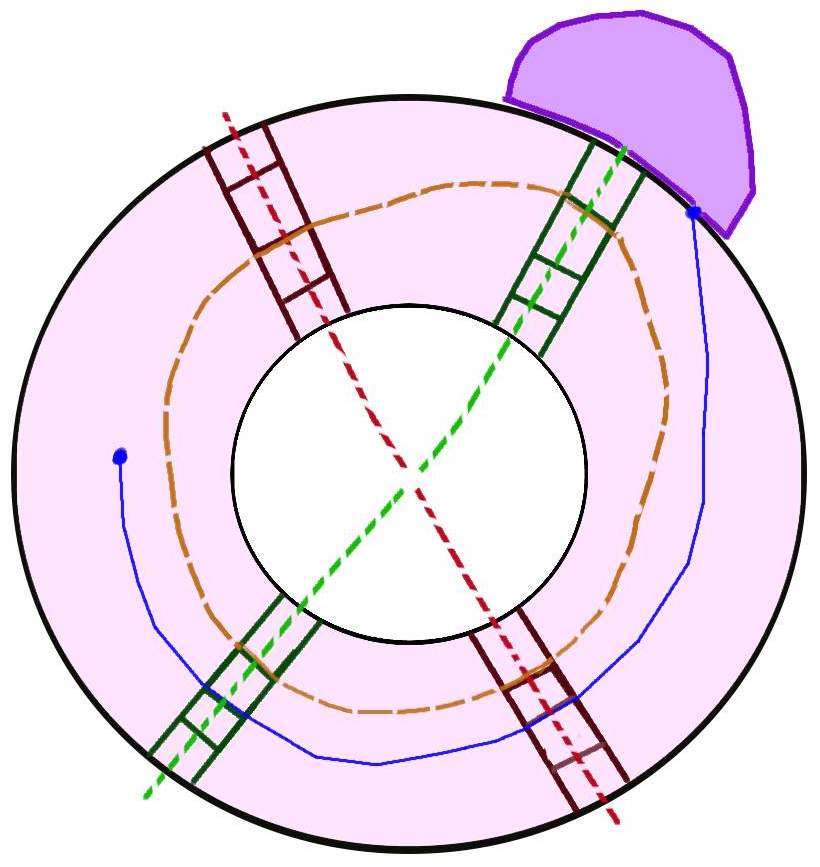}
\put(10,48){$p$}
\put(76,72){$q$}
\put(12,30){$\kappa$}
\put(17,62){$w$}
\put(46,70){$Y$}
\put(68,90){$Y_j$}
\put(72,1){$u_2$}
\put(24,90){$u_2'$}
\put(37,48){$U_2$}
\put(12,4){$u_1$}
\put(74,83){$u_1'$}
\put(55,48){$U_1$}
\end{overpic}
\caption{The scenario that arises in hypothesis~\eqref{Itm:Separation} of \Cref{Thm:cubulating X star}. For the geodesic $\kappa$, we have to verify that neither $p$ nor $q$ is $2$--proximate to $U_2 = u_2 \cup u_2'$.}
\label{fig:KappaAndFriends}
\end{figure}

\begin{proof}[Proof of \Cref{Thm:C'20Proper}]
Observe that \Cref{Def:Piece} of the $C'(\alpha)$ condition involves a strict inequality. Since $X^*$ is a finite $C'(\frac{1}{20})$ presentation, we may choose a constant $\alpha < \frac{1}{20}$ such that $X^*$ is $C'(\alpha)$ and such that $\diameter{N(u)} < \alpha \systole{Y_i}$ for every hyperplane of every $Y_i$.
Fix such an $\alpha$ for the remainder of the proof.

Observe that cubical subdivision preserves all the hypotheses of the theorem, and the $C'(\alpha)$ condition in particular.
We subdivide $X$ a number of times, while retaining the original metric. That is: every original edge of $X$ continues to have length $1$, while every edge of the $m^{\mathrm{th}}$ subdivision has length $2^{-m}$. The reason for the iterated subdivision is that the length of an edge contributes additive error to several calculations below, and a large value of $m$ will make this additive error negligible.
For instance, since $\alpha < \frac{1}{20}$ and $\systole{Y_i} \geq 1$, a large value of $m$ ensures that
\[
20 \alpha \systole{Y_i} < \systole{Y_i} - 2^{-m} \quad \text{for every }  i.
\]
Other inequalities with additive constants will follow similarly.

After the first subdivision of $X$, the hyperplanes of $X$ and of $Y_i$ become $2$--sided.
By hypothesis, there is a closed--geodesic $w_i$ that generates $\pi_1 Y_i$.
Moreover, we may assume that $\Aut(Y_i \to X)$ stabilizes $w_i$, for the following reason.
Since $Y_i$ is compact, $\widetilde{Y_i}$ is a quasi-line, hence $\stab(\widetilde{Y_i})$ is virtually cyclic.
Let $g$ be the generator of the maximal cyclic subgroup in $\stab(\widetilde{Y_i})$.
 Since we have subdivided $X$ at least once, $g$ stabilizes a geodesic axis $\widetilde{v_i} \subset \widetilde{Y_i}$, by \cite{HaglundSemiSimple}. Let $v_i = \langle g \rangle \backslash \widetilde{v_i}$. Now, we may choose $w_i  =  \langle g^n \rangle \backslash \widetilde{v_i}$, where $g^n$ generates $\pi_1(Y_i)$.

After the first subdivision,
the closed--geodesic $w_i \subset Y_i$ has an even number of edges.
The hypothesis that every hyperplane of $Y_i$ has contractible complement implies that every hyperplane intersects some edge of $w_i$.
 Our wallspace structure on each $Y_i$
will be defined by declaring each wall  to consist of the pair of hyperplanes corresponding to
a pair of antipodal edges of $w_i$. (See \Cref{Fig:CubicalPres} for a partial illustration of such a pairing, where the closed--geodesics $w_i$ are presumed to run along the inner boundary of each $Y_i$.)

For the rest of the proof, we focus on one relator $Y = Y_i$ and its closed--geodesic $w = w_i$. Before verifying the hypotheses of \Cref{Thm:cubulating X star}, we check an inequality that will be very useful in the sequel.

Let $u,u' \subset Y$ be a pair of hyperplanes in the same wall of $Y$. This means that $u,u'$ are dual to edges $\dot u, \dot u' \subset w$. Recall that $u$ has an embedded carrier such that $\diameter(N(u)) < \alpha \systole{Y}$, and similarly for $N(u')$.
Now, let $x \in N(u)$ and $x' \in N(u')$ be arbitrary. Then
\begin{align}
\dist(x,x') 
& \geq  \dist(\dot u, \dot u') - \dist(x, \dot u) - \dist(\dot u', x') \nonumber \\
& >  \left( \frac{1}{2} \systole{Y} - 2^{-m} \right) - 2 \alpha \systole{Y} \label{Eqn:FarHyperplanes} \\
& >  8 \alpha \systole{Y}, \label{Eqn:FarHyperplanesAlpha}
\end{align}
because $\alpha < \frac{1}{20}$ and $m$ is chosen so that $2^{-m}$ is tiny.

Next, we check the hypotheses of \Cref{Thm:cubulating X star}. In lieu of  the $B(6)$ condition of \Cref{Def:B6}, we will in fact verify the stronger $B(8)$ condition.

To check Condition~\eqref{B6:wallspace} of \Cref{Def:B6}, let $u,u'$ be distinct hyperplanes in the same wall of $Y$. Then each of $N(u)$ and $N(u')$ is embedded by hypothesis.
Furthemore, \Cref{Eqn:FarHyperplanesAlpha} implies $N(u) \cap N(u') = \emptyset$. Thus $u \cup u'$ partition $Y$ into two halfspaces.
Condition \eqref{B6:aut} of \Cref{Def:B6} holds because our (revised) choice of $w$ ensures that  $\Aut(Y \rightarrow X)$ stabilizes $w$, hence preserves the wallspace structure. 

For Condition~$(\ref{Itm:B8}')$, consider $P=P_1\cdots P_{8}$, a concatenation of $8$ piece-paths that starts on $N(u)$ and ends on either $N(u')$ or on $N(u)$.
    We may assume that each $P_i$ is a geodesic. If $P$ ends on $N(u')$, then the $C'(\alpha)$ hypothesis implies that $|P_i| < \alpha \systole{Y}$, hence $|P| < 8 \alpha \systole{Y}$, contradicting \Cref{Eqn:FarHyperplanesAlpha}. 

Now, consider the case where $P$ starts and ends on $N(u)$. Since we have already checked that $P$
 is too short to reach $N(u')$,
it follows that $P$ lifts to $\widetilde{Y} - (\pi_1 Y) \widetilde u'$,
%\begin{com} I think the meaning of this is clear...  \end{com}
  each of whose components is convex.
Thus $P$ is path-homotopic into $N(u)$. We have verified Condition~$(\ref{Itm:B8}')$ of \Cref{Def:B6}. Hence $X^*$ satisfies the $B(8)$ condition, and  hypothesis \ref{Thm:cubulating X star}.\eqref{Itm:B6}.

Hypothesis \ref{Thm:cubulating X star}.\eqref{Itm:Pieceful}, namely pieceful convexity of the hyperplanes of $Y$, follows from \Cref{rem:pieceful convexity}. Indeed, let $x \in N(\widetilde u)$ and $x' \in gN(\widetilde u)$ be points on the carriers of two distinct lifts of $u$ to $\widetilde Y$. Then, by the same calculation as in \Cref{Eqn:FarHyperplanes,Eqn:FarHyperplanesAlpha}, we have $\dist(x,x') > 8 \alpha \systole{Y}$, which exceeds the diameter of a piece.

Hypothesis~\ref{Thm:cubulating X star}.\eqref{Itm:CutWall} holds vacuously, since each $Y$ is compact.

We now verify Hypothesis~\ref{Thm:cubulating X star}.\eqref{Itm:Separation}.
Let $\kappa \to Y$ be a geodesic from $p$ to $q$, as in that condition. 
We begin by deriving upper and lower bounds on $|\kappa| = \dist(p,q)$. Let $\dot p, \dot q$ be vertices of $w$ that are closest to $p,q$, respectively. Since a hyperplane carrier containing $p$ cuts $w$ and has diameter less than $\alpha \systole{Y}$, we have $\dist(p, \dot p) \leq \alpha \systole{Y}$, and similarly $\dist(q, \dot q) \leq \alpha \systole{Y}$. Since $|w| = \systole{Y}$, we have $\dist(\dot p, \dot q) \leq \frac{1}{2} \systole{Y}$. Thus
\begin{align}
|\kappa| = \dist(p,q) 
& \leq \dist(p, \dot p) + \dist(\dot p, \dot q) + \dist(\dot q, q) \nonumber \\
&<  \frac{1}{2} \systole{Y} + 2 \alpha \systole{Y}. \label{Eqn:KappaUpper}
\end{align}
For the lower bound on $|\kappa|$, recall that there is a wall $U_1 = u_1 \cup u_1'$  such that $\kappa$ traverses a $1$--cell dual to $u_1$, and either $u_1'$ is $1$--proximate to $q$ or $\kappa$ traverses a $1$--cell dual to $u_1'$. See \Cref{fig:KappaAndFriends}. If $\kappa$ intersects $u_1'$, then \Cref{Eqn:FarHyperplanes} implies 
\[
|\kappa| \geq \dist(N(u_1) , N(u'_1)) >  \left( \frac{1}{2} \systole{Y} - 2^{-m} \right) - 2\alpha \systole{Y}.
\]
Otherwise, if $q$ is $1$--proximate to $u_1'$, then $\dist(q, u_1') < \alpha \systole{Y}$, hence we obtain the weaker conclusion 
\begin{align}
|\kappa| 
&>  \dist(N(u_1) , N(u'_1)) - \alpha \systole{Y} \nonumber \\
& >  \left( \frac{1}{2} \systole{Y} - 2^{-m} \right) - 3\alpha \systole{Y}. \label{Eqn:KappaLower}
\end{align}

Now, let $y$ be the midpoint of $\kappa$, and let $u_2$ be a hyperplane such that $y \in N(u_2)$.
Let $u_2'$ be the other hyperplane in the same wall $U_2$. We claim that neither $u_2$ nor $u_2'$ is $2$--proximate to $p$ or $q$. To see this, we will check that $\dist(\{p, q\}, u_2) > 2 \alpha \systole{Y}$, and similarly for $u_2'$. For any $x \in N(u_2)$, we have

\begin{align*}
\dist(\{p, q\}, x) & \geq \dist(\{p, q\}, y) - \dist(x,y) \\
& > \frac{1}{2} |\kappa| - \alpha \systole{Y} \\
& > \frac{1}{2} \left(   \frac{1}{2} \systole{Y} - 2^{-m}  - 3\alpha \systole{Y}    \right)  - \alpha \systole{Y}    \\
& = \frac{1}{4} \systole{Y} - \frac{5}{2} \alpha \systole{Y} -  2^{-(m+1)} \\
& > 2 \alpha \systole{Y}.
\end{align*}
Here, the first inequality is the triangle inequality, the second inequality is the diameter bound on $N(u_2)$, the third inequality is \Cref{Eqn:KappaLower}, and the final inequality follows
because $\alpha < \frac{1}{20}$ and $2^{-m}$ is tiny. Similarly, for any $x' \in N(u_2')$, we have

\begin{align*}
\dist(\{p, q\}, x') & \geq  \dist(x', y) - \dist(\{p, q\}, y)\\
& >  \left( \frac{1}{2} \systole{Y} - 2^{-m}  - 2 \alpha \systole{Y} \right) - \frac{1}{2} |\kappa| \\
& >  \left( \frac{1}{2} \systole{Y} - 2^{-m}  - 2 \alpha \systole{Y} \right) - \frac{1}{2} \left(  \frac{1}{2} \systole{Y} + 2 \alpha \systole{Y} \right) \\
& = \frac{1}{4} \systole{Y} - 3 \alpha \systole{Y} - 2^{-m} \\
& > 2 \alpha \systole{Y}.
\end{align*}
Here, the first inequality is the triangle inequality, the second inequality is \Cref{Eqn:FarHyperplanes}, the third inequality is \Cref{Eqn:KappaUpper}, and the final inequality holds
because $\alpha < \frac{1}{20}$ and $2^{-m}$ is tiny.
Thus neither $p$ nor $q$ can be $2$--proximate to $u_2$ or $u_2'$.
We have thus verified that \Cref{Thm:cubulating X star} applies, hence $\pi_1 X^*$ acts with torsion stabilizers
on the CAT(0) cube complex dual to the wallspace structure we have constructed. Since $X$ is compact, this action is cocompact.

Finally, we check the conclusion about cell stabilizers in $\pi_1 X^*$. 
We have already shown that cell stabilizers must be torsion.
%    Since $\pi_1 X^*$ acts properly on the dual cube complex, any cell stabilizer must be finite. 
By \cite[Thm 4.2 and Rem 4.3]{WiseIsraelHierarchy}, we know that when $\alpha < \frac{1}{20}$,
any torsion element in  $\pi_1X^*$ lies in $\Aut(Y_i \rightarrow X)$ for some lift $Y_i \hookrightarrow \widetilde{X^*}$. 
Recall that $\widetilde Y_i$ is quasi-isometric to a line.
If $\sigma \in \Aut(Y_i \rightarrow X)$ is non-trivial, let $\overline{Y_i} = \langle \sigma \rangle \backslash Y_i$. Then $\overline{Y_i} \rightarrow X$ is a local isometry, with $\pi_1(\overline{Y_i})$ an infinite cyclic subgroup properly containing $\pi_1 Y_i$.
In particular, if $\pi_1 Y_i$ is maximal cyclic, then we obtain a contradiction, hence $\pi_1 X^*$ acts freely.
\end{proof}

\section{Controlling pieces in generic quotients}
\label{Sec:Pieces}

This section contains the proof of the main theorem, \Cref{Thm:main}. In the proof, we need to control the sizes of pieces in a generic cubical presentation. Wall-pieces are controlled in \Cref{Prop:WallPieceControl} and cone-pieces are controlled in a sequence of lemmas, culminating in \Cref{Prop:ConePieceSystole}. 

Most of the work in this section occurs in the language of \emph{loose pieces}, which represent a coarsening of the overlaps defined in \Cref{Def:Piece}. We define loose pieces in \Cref{Def:LoosePiece}. In \Cref{Lem:LoosenUp}, we show that every wall-piece or cone-piece in the sense of \Cref{Def:Piece} gives rise to a corresponding loose piece.

Via loose pieces, the proofs of Propositions~\ref{Prop:WallPieceControl} and \ref{Prop:ConePieceSystole} readily generalize to the context of non-cubical hyperbolic metric spaces in \Cref{Sec:OtherMetric}.

\subsection{Loose pieces and convex hulls of quasi-axes}\label{Sec:Loose}

Let $\Upsilon$ be a metric space. For any subset $s \subset \Upsilon$ and $J > 0$, the notation $\neb_J(s)$ denotes the closed $J$--neighborhood of $s$.

\begin{defn}[Quasi-axes]\label{Def:QuasiAxis}
Let $\widetilde w, \widetilde w'$ be bi-infinite geodesics in a metric space  $\Upsilon$. We declare them to be equivalent, and write $\widetilde w \approx \widetilde w'$, whenever $\widetilde w \subset \neb_r(\widetilde w')$ for some $r \geq 0$. In the main case of interest, when $\Upsilon$ is $\delta$--hyperbolic,  we can use a uniform value $r = 2\delta$, and deduce that $\widetilde w \approx \widetilde w'$ if and only if $\bdy \widetilde w = \bdy \widetilde w'$.

Let $g$ be a hyperbolic isometry of $\Upsilon$.
A \emph{quasi-axis} for $g$ is a bi-infinite geodesic $\widetilde w \subset \Upsilon$ such that $g \widetilde w \approx \widetilde w$.
 When $\Upsilon$ is $\delta$--hyperbolic,
any geodesic connecting the fixed points of $g$ on $\bdy \Upsilon$
is a quasi-axis for $g$.
\end{defn}

\begin{lem}\label{Lem:HullUniformThickness}
Let $\widetilde{X}$ be a $\delta$--hyperbolic, finite dimensional CAT(0) cube complex. Then there is a uniform constant $K = K(\widetilde{X})$ with the following property. For every bi-infinite geodesic $\widetilde{w}$, 
the convex hull of $ \widetilde w $ satisfies
\[ 
   \hull \left(  \widetilde w  \right) \subset \neb_K(\widetilde{w} ) . 
%    \hull \left( \bigcup_{\widetilde w' \approx \widetilde w}  \widetilde w'  \right) \subset \neb_K(\widetilde{w} ) . 
\]
\end{lem}

\begin{proof}
Consider a vertex $p \in \widetilde{X}$ that lies far from $\widetilde w$. Let $s$ be a shortest CAT(0) (non-combinatorial) geodesic segment from $p$ to $\widetilde w$, oriented outward from $p$. Sageev and Wise have observed  \cite[Remark~3.2]{SageevWiseCores} that the first cube met by $s$ contains a midcube of  a hyperplane $H$ that intersects $s$ at an angle bounded away from $0$. The angle bound depends only on the dimension of the cube, hence can be taken uniformly over $\widetilde X$.

We claim that when $s$ is sufficiently long, $H \cap \widetilde w = \emptyset$. In other words, $H$ separates $p$ from $\widetilde w$, hence $p \notin \hull(\widetilde w)$. The claim holds because an intersection $H \cap \widetilde w$ would determine a CAT(0) geodesic triangle with a long base along $s$ and two large angles along $s$. Such a triangle cannot be $\delta$--thin.
%    
%    The same argument, using a geodesic quadrilateral with a $2\delta$--short side between $\widetilde w$ and $\widetilde w'$, implies that $H \cap \widetilde w' = \emptyset$ for every $\widetilde w' \approx \widetilde w$. Thus a point $p$ sufficiently far from $\widetilde w$ must lie outside $\hull(\cup \widetilde w')$.
\end{proof}

The following construction relates cubical presentations to quasi-axes.   
Given $G = \pi_1 X$ and  $g \in G - \{1\}$, choose any 
 quasi-axis $\widetilde w$ for $g$. Then, construct  a convex subcomplex $\widetilde{Y} = \hull (  \widetilde{w} ) \subset \widetilde{X}$,
whose quotient $Y =  \langle g \rangle \backslash \widetilde Y$ admits a local isometry into $X$.
By \Cref{Lem:HullUniformThickness},  $\widetilde Y$ lies in a uniform neighborhood of $\widetilde w$, hence its quotient $Y$ is compact. 
    We say that $Y$ is \emph{a quasi-circle defined by $g$}.

If  $g$ stabilizes an axis $\widetilde w$, as will be the case following \Cref{Conv:NiceX}, then the quasi-circle $Y$ deformation retracts to the closed--geodesic $w = \langle g \rangle \backslash \widetilde w$.

Now, we adopt the following convention.

\begin{conv}\label{Conv:NiceX}
From now until the start of the proof of \Cref{Thm:main}, the symbol $X$ denotes a compact non-positively curved cube complex, such that $\widetilde X$ is $\delta$--hyperbolic for some $\delta$. We assume that $G = \pi_1 X$ is nonelementary.

We assume that $X$ has been subdivided at least once, so that every $g \neq 1$ stabilizes a geodesic axis in $\widetilde X$ \cite{HaglundSemiSimple}.

Finally, we assume that every hyperplane of $X$ is essential. This assumption is harmless, because every non-positively curved cube complex $X'$ contains a convex essential core, whose hyperplanes are essential. See, for instance, Caprace and Sageev \cite[Proposition 3.5]{CapraceSageev2011}.
\end{conv}

Our arguments in this section will employ the following, looser analogue of \Cref{Def:Piece}.

\begin{defn}[Loose pieces]\label{Def:LoosePiece}
Let $X$ be as in \Cref{Conv:NiceX}.
Fix a constant $J \in \naturals$. We define $J$--loose pieces first in $\widetilde X$, then in $X$.

Consider bi-infinite geodesics $\widetilde w, \widetilde w'$ that do not share an endpoint in $\bdy \widetilde X$. Then a \emph{$J$--loose cone-piece between $\widetilde w$ and $\widetilde w'$} is the maximal geodesic segment $s \subset \widetilde w$ whose endpoints are contained in $\neb_{J} (\widetilde w')$. 
The \emph{companion of $s$} is the maximal segment $s'  \subset \widetilde w'$ whose endpoints are at distance $J$ from the corresponding endpoints of $s$. 

Now, let $w, w'$ be closed--geodesics in $X$. If $w \neq w'$, a \emph{$J$--loose cone-piece} between $w$ and $w'$ is the projection of a $J$--loose cone-piece between preimages $\widetilde w$ and $\widetilde w'$. A \emph{$J$--loose cone-piece between $w$ and itself}
arises in the above scenario where  $\widetilde w' = h\widetilde w$ for some $h \notin \stab_G(\widetilde w)$. 

In a similar manner, a \emph{$J$--loose wall-piece in $\widetilde w$} is a maximal geodesic segment  $s \subset \widetilde{w}$ whose endpoints are contained in $\neb_J (N(\widetilde{U}))$ for a hyperplane $U$.
A \emph{$J$--loose wall-piece in a closed--geodesic $w$} is the projection of a $J$--loose wall-piece in $\widetilde w$.
\end{defn}

The next lemma shows that there is a uniform value $J = J(\widetilde X)$ such that every piece in a cubical presentation $X^*$ must correspond to a $J$--loose piece in $X$.

\begin{lem}\label{Lem:LoosenUp}
Consider a cubical presentation $X^* = \langle X \mid Y_1, \ldots, Y_m \rangle$. Suppose that $\widetilde X$ is hyperbolic and  that every $Y_i$ is a quasi-circle. 
Let  $\widetilde w_i$ be a bi-infinite geodesic in $\widetilde Y_i$. 
Then there is a constant $J = J(\widetilde{X})$ such that the following hold for every $d \geq 0$:
\begin{enumerate}[\:\: $(1)$]
\item Every diameter $ \geq d$ cone-piece of $X^*$ between $Y_i$ and $Y_j$ determines a $J$--loose cone-piece between $\widetilde w_i$ and $\widetilde w_j$, of diameter  $\geq d$. 
\item Every diameter $ \geq d$ wall-piece of $X^*$ in $Y_i$ determines a $J$--loose wall-piece in $\widetilde w_i$, also of diameter  $\geq d$.
\item\label{Itm:HyperplaneInY} Every diameter $\geq d$ intersection between $\widetilde Y_i$ and a hyperplane carrier $N(U)$ determines a $J$--loose wall-piece  in $\widetilde w_i$, of diameter  $\geq d$.
\end{enumerate}
Furthermore, we may take $J \leq 3K$, where $K$ is the constant of  \Cref{Lem:HullUniformThickness}.
\end{lem}

\begin{proof}
First, consider cone-pieces.
Let $p,q \in \widetilde Y_i \cap \widetilde Y_j$ be points such that $\dist(p,q) \geq d$. 
By \Cref{Lem:HullUniformThickness}, there are points $p_i, q_i \in \widetilde w_i$ 
that are $K$--close to $p$ and $q$, respectively. Applying \Cref{Lem:HullUniformThickness} to the quasi-circle $Y_j$ shows that $p_i, q_i$ are $2K$--close to $\widetilde w_j$.
Since $\widetilde w_i$ is a geodesic, we have $\dist(p_i, q_i) \geq d - 2K$. Extending the segment $[p_i, q_i]$ 
by $K$ in each direction, we obtain a geodesic segment $[p_i', q_i'] \subset \widetilde w_i$ whose length is at least $d$ and whose endpoints lie in $\neb_{3K}(\widetilde w_j)$. Thus $[p_i, q_i]$ is contained in a $3K$--loose cone-piece between $\widetilde w_i$ and $\widetilde w_j$.

The proof for wall-pieces and intersections with hyperplanes is identical. In this case, we may take $J = 2K$.
\end{proof}

\begin{rem}\label{Rem:MoreGeneralHyperplane}
Besides the looseness constant $J$, there is one important respect in which \Cref{Def:LoosePiece} is more general than \Cref{Def:Piece}. Whereas the definition of a wall-piece in $Y_i$ requires the hyperplane $U$ to be disjoint from $Y_i$, the definition of a $J$--loose wall-piece in $\widetilde w_i$ has no analogous restriction. Consequently, the $J$--loose wall-pieces in \Cref{Lem:LoosenUp}.\eqref{Itm:HyperplaneInY} need not come from wall-pieces. When we apply \Cref{Thm:C'20Proper}, we will care about \emph{all} intersections between $Y_i$ and hyperplane carriers in $X$, and all such intersections will be controlled using $J$--loose wall-pieces.
\end{rem}

\subsection{Wall-pieces}\label{Sec:WallPiece}

Our next goal is to give a criterion on the growth of the number of relators ensuring that
with overwhelming probability, the diameter of wall-pieces in a cubical presentation is bounded by $\alpha \systole{Y_i}$. We formulate the criterion in terms of loose pieces, for greater applicability in \Cref{Thm:C'20Proper} and in the generalized setting of \Cref{Sec:OtherMetric}.

Recall that by \Cref{Conv:NiceX}, we are assuming that $G$ is nonelementary and the hyperplanes of $X$ are essential. By \Cref{Thm:ConjugacyGrowth}, the growth function of $G$ acting on $\widetilde X$ is bounded by constants times $e^{b\ell}$, where $b > 0$ is the growth exponent.
Let $a$ the largest growth exponent among the carriers of the finitely many
orbits of hyperplanes of $\widetilde X$. 
Then $a$ is also the largest growth exponent of any of the $J$--neighborhoods of the hyperplanes, for an arbitrary $J$.
Recall that $a < b$ by \Cref{Thm:SubgroupGrowth}.

For each $\ell > 0$, let $\mathcal{G}(\ell)$ denote the set of nontrivial conjugacy classes in $\pi_1 X$ of length at most $\ell$. For the duration of this section, we will be sampling randomly from $\mathcal G(\ell)$, for large $\ell$.

\begin{prop}\label{Prop:WallPieceControl}
Let $X$ be as in \Cref{Conv:NiceX}.
 For each non-trivial conjugacy class  $\pi_1 X$, choose a 
representative closed--geodesic. 
Let $a < b$ be growth exponents as above, let $\alpha \in (0,1)$, and fix a positive number $c < \alpha(b-a)$. Then, for every $J \in \naturals$, there is a constant $C=C(J,X)$, with the following property.

For $k \leq e^{c \ell}$, choose conjugacy classes $[g_1], \ldots, [g_k] \in \mathcal{G}(\ell)$ uniformly at random. Let $w_i$ be the closed--geodesic representing $[g_i]$.
Then, for all $\ell \gg 0$,  the probability that some $w_i$ contains a $J$--loose wall-piece of diameter at least $\alpha \systole{w_i}$ is bounded above by
 \[
C \ell^2 e^{ (c+(a -b) \alpha) \ell  } .
 \]
 Consequently, with overwhelming probability as $\ell \to \infty$, all $J$--loose wall-pieces in $w_i$ have diameter less than $\alpha |g_i|$.
\end{prop}

\begin{proof}
Let $\Pi(\alpha,\ell)$
 denote the proportion of representative closed--geodesics $w$ of length at most $\ell$ such that
 $w$ contains a $J$--loose wall piece of diameter at least $\alpha \systole{w}$. That is,  $\Pi(\alpha,\ell)$ denotes the proportion of all $w$ of length at most $\ell$ such that 
$ \widetilde w \cap  \neb_J (N(\widetilde{U})) $ contains the endpoints of a segment of length at least $\alpha \systole{w}$
for some hyperplane $U$. Our goal is to estimate $\Pi(\alpha, \ell)$.

For each $n \leq \ell$, let $\mathcal{U}(n,\alpha)$ be the set of (conjugacy representatives of) closed--geodesics $w$ of length  exactly $ n$
that contain a $J$--loose wall-piece $s$ with $|s| \geq \alpha n$.
Then
\[
\Pi(\alpha, \ell) = \frac{ \sum_{n=1}^\ell | \mathcal{U}(n,\alpha) |}{| \mathcal{G}(\ell)  |} .
\]

Each element  $w \in \mathcal{U}(n,\alpha)$ can be cyclically permuted to be of the form $s \cdot y$, where $s$ lies in 
$\neb_J (N(\widetilde{U}))$ for some hyperplane $\widetilde U$, and
 $|s| \geq \alpha n$.
By \Cref{Thm:ConjugacyGrowth}.\eqref{Itm:TotalGrowth} and \Cref{Rem:GroupoidGrowth}, the number of such paths $s$, up to homotopy in $\widetilde X$, is bounded above by $B'e^{a|s|}$, for some constant $B'$ depending on $J$. Similarly, the number of paths $y$ is bounded above by $B''e^{b|y|}$, for some constant $B''$. 

Let $B = B'B''$. Since there are at most $n$  distinct cyclic permutations of $w$,
the number of elements in $\mathcal{U}(n,\alpha)$ is
\begin{equation}\label{Eqn:WallPieceCyclicPerm}
| \mathcal{U}(n,\alpha) |
\leq n \cdot B e^{a |s| } e^{b |y|}
= B n \,  e^{a |s|  + b (|w|-|s|) }
\leq  B n \,
e^{(a \alpha + b(1-\alpha)) n} .
\end{equation}

Summing over all lengths up to $\ell$, we obtain:
\begin{align*}
\sum_{n=1}^\ell | \mathcal{U}(n,\alpha) |
& \leq \sum_{n=1}^\ell Bn  e^{(a \alpha + b(1-\alpha))n}  \\
& \leq B \ell \,  \sum_{n=1}^\ell \left( e^{a \alpha + b(1-\alpha)} \right)^n  \\
& < B \ell \,  \sum_{m=0}^\infty \left( e^{a \alpha + b(1-\alpha)} \right)^{\ell - m}  \\
& =  B  \ell \,  \frac{ e^{(a \alpha + b(1-\alpha)) \ell} }{1 - e^{-(a \alpha + b(1-\alpha))} }   \\
& =   e^{b \ell} \cdot  \ell e^{(a-b)\alpha \ell } \cdot  \frac{ B}{1 - e^{-(a \alpha + b(1-\alpha))} }   \\
& \leq   e^{b \ell} \cdot  \ell e^{(a-b)\alpha \ell } \cdot  \frac{ B}{1 - e^{-a  } }  .
\end{align*}

Meanwhile, \Cref{Thm:ConjugacyGrowth}.\eqref{Itm:ConjugacyGrowth} implies that $|\mathcal{G}(\ell) | \geq \frac{A}{\ell} e^{b \ell}$. Thus
\begin{equation}\label{eq:bigwallpiece propertion}
\Pi(\alpha, \ell) \leq \frac{ \sum | \mathcal{U}(n,\alpha) |}{| \mathcal{G}(\ell)  |}
\leq
\ell^2 e^{(a-b)\alpha \ell } \cdot \frac{B  } { A (1 - e^{-a } ) }.
\end{equation}

Among all choices $(w_1,w_2,\ldots, w_k) \in \mathcal{G}(\ell)^k$, the proportion of $k$--tuples where some $w_i \in \mathcal{U}(n,\alpha)$ is bounded above by
$k \Pi(\alpha,\ell)$.
Since $k  \leq  e^{c \ell}$,
Equation~\eqref{eq:bigwallpiece propertion}
says that the total probability of a large piece is bounded by
\[
k \Pi(\alpha,\ell) \leq \ell^2 e^{c \ell +(a-b)\alpha \ell } \cdot \frac{B  }  { A (1 - e^{-a } ) } 
% old denominator: { A (1 - e^{-(a \alpha + b(1-\alpha))} ) } .
\]
 Setting $C =  \frac{B  } { A (1 - e^{-a} ) }$ completes the proof of the estimate.
In particular, since $c+(a-b)\alpha < 0$ by hypothesis, we have $k \Pi(\alpha,\ell) \to 0$ as $\ell \to \infty$.
\end{proof}

\subsection{Cone-pieces between two closed--geodesics}\label{Sec:ConePiece}
Our next goal is to control cone-pieces whose diameter is bounded below by a constant $d$. The gist of the following sequence of lemmas is that the probability of a $J$--loose cone-piece of diameter at least $d$ declines exponentially with $d$. See \Cref{Prop:FewConePieces}.

We will need to employ separate arguments for pieces between closed--geodesics $w$ and $w'$ chosen separately, and pieces between a closed--geodesic $w$ and itself. In the case of a piece between $w$ and itself, we will further need to consider overlaps, defined as follows.

\begin{figure}
\begin{overpic}[width=5in]{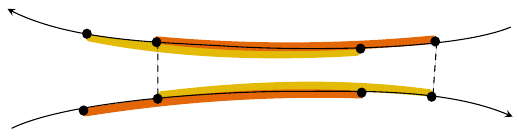}
\put(5,4){$\widetilde w$}
\put(1,18){$\widetilde w' = h \widetilde w$}
\put(14,1){$h^{-1}q'$}
\put(29,3.5){$p$}
\put(67,4){$h^{-1}p'$}
\put(82,4){$q$}
\put(15,21.5){$hq$}
\put(29,20){$p'$}
\put(68,19){$hp$}
\put(82,20){$q'$}
\put(27.5,11){$J$}
\put(84.5,12){$J$}
\put(53,10){$s$}
\put(53,18){$s'$}
\end{overpic}
\caption{The terminology and setup of \Cref{Def:OverlapOrientation}. The $J$--loose piece $s = [p,q] $ is shown in yellow, and the companion $s' = [p',q']$ in orange. In this example, $h$ reverses orientation on the piece $s$, and $s$ has nontrivial $h$--overlap $[p, h^{-1}p']$.}
\label{Fig:Overlap}
\end{figure}
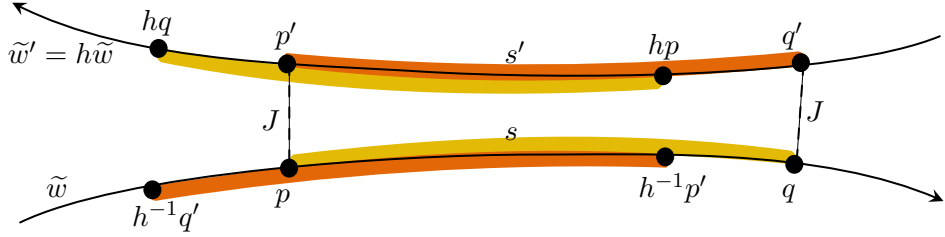

\begin{defn}[Orientation and overlaps]\label{Def:OverlapOrientation}
Let $w \to X$ be a closed--geodesic. Choose distinct preimages $\widetilde w$ and $h \widetilde w$ of $w$ and a constant $J \in \naturals$. Suppose $s = [p,q] \subset \widetilde w$ is a $J$--loose piece between $\widetilde w$ and $h \widetilde w$, such that $|s| > 2J$. Let  $s' = [p',q'] \subset h\widetilde w$ be the companion of $s$, so that
$\dist(p,p')=J =\dist(q,q')$.

Suppose that $\widetilde w$ is oriented from $p$ to $q$, and transfer this orientation to $h \widetilde w$ via $h$. 
Since $|s| = \dist(p,q) > 2J$, the endpoints $p',q'$ of the companion $s'$ cannot coincide.
 We say that $h$ \emph{preserves orientation on $s$} if $p'$ comes before $q'$ in the orientation on $h \widetilde w$, and \emph{reverses orientation on $s$} otherwise. See \Cref{Fig:Overlap}.

The \emph{$h$--overlap in $s$} is the interval $u = s \cap h^{-1} s'$. 
\end{defn}

The main steps of the proof of \Cref{Prop:FewConePieces} can be organized as follows:
\begin{itemize}
\item In \Cref{Lem:ReversingPiecesSmall}, we show that any orientation-reversing piece between $w$ and itself has small overlap, universally bounded by a constant $R$. This is an unconditional (non-probabilistic) statement.
\item In \Cref{Lem:PieceNonGeneric}, we show that pieces between separately sampled closed--geodesics $w$ and $w'$ are exponentially rare. Similarly, pieces between $w$ and itself with small overlap (of size $\leq R$) are exponentially rare.
\item In \Cref{Lem:SelfPieceBigOverlap}, we show that pieces between $w$ and itself with big overlap (of size $\geq R$) are also exponentially rare. This allows us to conclude that each type of piece of diameter $\geq d$ is exponentially rare.
\end{itemize}

\smallskip

We will need the following standard fact about thin quadrilaterals.

\begin{lem}\label{Lem:ThinQuad}
Suppose that $\widetilde X$ is $\delta$--hyperbolic. Let $Q$ be a geodesic quadrilateral with corners $p,q,q',p'$, such that $\dist(p,p') \leq J$ and $\dist(q,q')\leq J$. Then every point $x \in [p,q]$ lies within distance $\mathring J = J + 2\delta$ of some point of $[p',q']$.
\end{lem}

\begin{proof}
This is standard. Since triangles in $\widetilde X$ are $\delta$--thin, then quadrilaterals are $2\delta$--thin. Thus $x \in [p,q]$ is $2\delta$--close to some  point $x' \in [p,p'] \cup [p',q'] \cup [q',q]$. But every point of $[p,p'] \cup [q',q]$ is $J$--close to $[p',q']$.
\end{proof}

\begin{lem}\label{Lem:ReversingPiecesSmall}
Suppose $X$ is compact
 and $\widetilde X$ is $\delta$--hyperbolic. 
 For every $J \in \naturals$, there is a constant $R = R(J, \delta) \geq 10J$ such that the following holds.
Suppose that $s \subset  \widetilde w$ is a $J$--loose cone-piece between $\widetilde w$ and $h \widetilde w$, where $h$ reverses orientation on $s$. Then the $h$--overlap $u \subset s$ has diameter $|u| \leq R$. 
\end{lem}

\begin{proof}
In the notation of \Cref{Def:OverlapOrientation}, the overlap is  $u = [p,q] \cap h^{-1}[p', q']$. Since $h$ is orientation-reversing, the endpoints of $u$ are $x \in \{p, h^{-1}q'\}$ and $y \in \{q, h^{-1}p'\}$. We 
begin by showing there is a constant $L = L(J, \delta)$ such that
\begin{equation}\label{Eqn:EndpointsClose}
\dist(x, hy) \leq L \quad \text{and}  \quad \dist(y, hx) \leq L.
\end{equation}

To prove \eqref{Eqn:EndpointsClose}, we consider two cases. For the first case, suppose that $\dist(p,h^{-1}q') \leq 2J$ and $\dist(q,h^{-1}p') \leq 2J$. Since $x \in \{p, h^{-1}q'\}$ and $hy \in \{p',hq\}$,
 triangle inequalities imply that
\[
\dist(x,hy) \leq \dist(x,p) + \dist(p,p') + \dist(p', hy) \leq 2J + J + 2J = 5J.
\]
By an identical calculation, $\dist(y,hx) \leq 5J$. 

For the second case, suppose that $\dist(p,h^{-1}q') > 2J$ or $\dist(q,h^{-1}p') > 2J$. Since we have $|\dist(p,q) - \dist(p',q')| \leq 2J$, it follows that one of $h^{-1}(q'), h^{-1}(p')$ lies in $[p,q]$. Assume without loss of generality that $h^{-1}(p') \in [p,q]$. If we furthermore assume
that $\dist(q,h^{-1}p') > 2J$,
then the endpoints of $u$ are $x=p$ and $y=h^{-1}p'$, as depicted in \Cref{Fig:Overlap}. It follows immediately that
\[
\dist(x,hy) = \dist(p,p') = J.
\]
Furthermore, as in \Cref{Lem:ThinQuad},
%    since $\widetilde X$ is $\delta$--hyperbolic, 
every point of $[p,q]$ is $2\delta$--close to some point of $[p,p'] \cup [p',q'] \cup [q',q]$. Since $\dist(y, q) > J$, it follows that
$y  \in [p,q]$ is $2\delta$--close to some point $y' \in [p',q']$. Furthermore, since $\dist(x,y) = \dist(hy, hx) = \dist(p',hx)$, 
triangle inequalities imply that $\dist(hx,y') \leq J + 2 \delta$. Thus
\[
\dist(y,hx) \leq \dist(y,y') + \dist(y',hx) \leq 2\delta + (J+2\delta) = J + 4 \delta.
\]

It remains to consider the possibility that $h^{-1}(p') \in [p,q]$ and $\dist(q,h^{-1}p') \leq 2J$. Since we are in the second case, triangle inequalities imply that $2J < \dist(p,h^{-1}q') \leq 4J$. Now, we argue exactly as in the first case, and conclude
\[
\dist(x,hy)  \leq 4J + J + 2J = 7J \quad  \text{and} \quad \dist(y,hx)  \leq 4J + J + 2J = 7J.
\]
Setting $L = \max(7J, J+4\delta)$ completes the proof of \Cref{Eqn:EndpointsClose}.

Now, we complete the proof of the lemma. Applying \eqref{Eqn:EndpointsClose} twice shows that 
$\dist(x,h^2 x) \leq 2L$ and $\dist(h^{-1}x,hx) \leq 2L$.
By \Cref{Conv:NiceX},
$h$ has an invariant geodesic axis $\widetilde v$ in 
$\widetilde X$, hence
 $\stabletrans{h} \geq 1$.
  Thus
the $\langle h^2 \rangle$--orbit $\{ h^{2n} (x) \mid n \in \integers \}$ lies along a quasigeodesic whose quality depends only on $J$ and $\delta$. Similarly, 
$\{ h^{2n} (hx) \mid n \in \integers \}$ lies along a  uniform quasigeodesic. 

The above quasigeodesics must uniformly fellow-travel $\widetilde v$. As a consequence, there is a uniform 
radius $r$, depending only on $J$ and $\delta$,  such that  $\langle h \rangle (x)  \subset  \neb_r(\widetilde v)$.

Since $x, hx \in \neb_r(\widetilde v)$, considering closest-point projections to $\widetilde v$ gives  
\[
\dist(x, hx) \leq r + \stabletrans{h} + r \leq 2 r + L.
\]
Meanwhile \Cref{Eqn:EndpointsClose} gives $\dist(hx, y) \leq L$. Thus $R = 2r + 2L$ is a uniform bound on $|u| = \dist(x,y)$. Since $L \geq 7J$ by definition, we have $R \geq 14J$.
\end{proof}

Now, we argue that closed--geodesics containing a piece of diameter $\geq d$ happen with a probability that decays exponentially with $d$. 

\begin{lem}\label{Lem:PieceNonGeneric}
Suppose that every non-trivial element of $\pi_1 X$ stabilizes a geodesic in a $\delta$--hyperbolic space $\widetilde X$. For each non-trivial conjugacy class in  $\pi_1 X$, choose
 a closed--geodesic in $X$ that represents it. 

Fix constants $J, R \in \naturals$. Then, for all $d > 2J$ and all $\ell$ sufficiently large, the following holds:
\begin{enumerate}[\:\: $(1)$]
\item \label{Itm:SeparateSample}
Among all pairs of conjugacy classes $[g], [g']$ of length at most $\ell$, the proportion whose representative closed--geodesics have a $J$--loose cone piece of diameter $\geq d$ is less than 
 $M \ell^2 \, e^{-b d}$, for a constant $M = M(X, J)$.
 
 \item\label{Itm:SelfPieceSmallOverlap}
 Among all conjugacy classes $[g]$ of length at most $\ell$, the proportion whose representative closed--geodesic has a $J$--loose cone piece with itself, with diameter $\geq d$ and overlap of length $\leq R$, is 
  less than $M_0 \ell^2  e^{-b d}$,  for a constant $M_0 = M_0(X, J, R)$.
  \end{enumerate}
\end{lem}

\begin{proof}
Before beginning the probabilistic portion of the proof, we make a reduction, introduce notation, and name some constants.
Let $D$ be the diameter of $X$. Let $\mathring J = J + 2\delta$, and let $V$ be the maximal number of vertices in a ball of radius  $(D+J+\mathring J)$ in $\widetilde X$. Let $W$ be the maximal number of vertices in a ball of radius $n_0$, where $n_0$ is the threshold constant of \Cref{Thm:ConjugacyGrowth}. Fix a basepoint $x_0 \in \widetilde X$.

Suppose that $[g], [g']$ are conjugacy classes represented by closed--geodesics $w, w'$ that have a $J$--loose cone-piece. (For now, we do not place any constraint on the lengths of $w, w'$, or the piece.)
By \Cref{Def:LoosePiece}, the piece in $w$ is the image of a geodesic segment $s \subset \widetilde w  \subset \widetilde X$, which is a $J$--loose cone piece between $\widetilde w$ and $\widetilde w'$. By choosing the preimage $s$ appropriately, we ensure the additional property that the initial point $p \in s$ lies within radius $D$ of the basepoint $x_0 \in \widetilde X$. Let $s' \subset \widetilde w'$ be the companion of $s$. Then the initial point of $s'$ is a point $p'$ 
that lies $J$--close to $p \in S$, hence within distance $D+J$ of the basepoint $x_0$.  Furthermore, by \Cref{Lem:ThinQuad}, every point of $s'$ lies within distance $\mathring J = J + 2\delta$ of some point of $s$.

Since conjugacy classes in $\pi_1 X$ correspond to free homotopy classes in $X$, every geodesic segment in $ \widetilde w'$ of length $|w'|$ is a fundamental domain for $w'$. It will be convenient to choose a fundamental domain $v' \subset \widetilde w'$ that begins at $p'$, such that the initial segment of $v'$ coincides with the initial segment of $s'$. Since the orientation of $s'$ (determined by the condition that the initial point $p' \in s'$ is $J$--close to $p \in s$) may or may not coincide with the orientation of $w'$, we conclude that the segment $v'$ chosen as above is a fundamental domain for $(w')^{\pm 1}$.
\smallskip

Now, we proceed to the proof of \eqref{Itm:SeparateSample}. Fix a conjugacy class $[g]$ and its representative  closed--geodesic $w$. We will bound the number of conjugacy classes $[g']$, with representative closed--geodesic $w'$, such that $w$ has a $J$--loose cone-piece with $w'$, of diameter at least $d$.

As in the above notation, let $s \subset \widetilde w$ be a preimage of the piece in $w$ that begins within distance $D$ of the basepoint $x_0$. As in the opening paragraph of the proof, let $s' \subset \widetilde w'$ be the corresponding sub-segment of $\widetilde w'$, so that the initial point $p \in s$ is $J$--close to the initial or terminal point $p' \in s'$. % Let $n = |s'|$.

As a warm-up case, suppose that  $|s'| \geq |w'| - n_0$, where $n_0$ is the threshold constant in \Cref{Thm:ConjugacyGrowth}.
As above, $(w')^{\pm 1}$ has a fundamental domain $v' \subset \widetilde w'$ whose initial segment coincides with the initial segment of $s'$. More precisely, we have a fundamental domain $v' = [p', r']$ and a point $q' \in [p', r']$ such that $s' \cap v' = [p',q']$ and $\dist(q',r') \leq n_0$.
Then, by construction, we have $\dist(p,p') = J$. Furthermore, by \Cref{Lem:ThinQuad}, we have $\dist(q',q) \leq \mathring J$ for some point $q \in s$. Observe that $|\dist(p,q) - \dist(p',q')| \leq J + \mathring J$.

By the definition of $V$, and the choice of $s$ so that its start lies close to the basepoint, there are at most $V$ choices for where $p'$ can lie. Then, once $n = \dist(p,q)$ is chosen, the point $q \in s$ is determined, hence there are at most $V$ choices for $q' \in s'$. Since $\dist(q',r') \leq n_0$, it follows that once $q'$ is chosen, there are at most $W$ choices for $r'$.
Consequently, for every $n$, there are at most $V^2W$ choices for the segment $v'$,
hence at most $2 V^2 W$ choices for $[g']$, where the factor of $2$ accounts for the orientation of $g'$. Recalling that $n \leq \ell + J + \mathring J$, the total number of possibilities for $[g']$ is at most
\[
2 V^2W \cdot (\ell + J + \mathring J).
\]
In particular, the bound grows linearly with $\ell$ in the (non-generic) warm-up case.

\smallskip

Having finished the warm-up case, assume that $|s'| < |w'| - n_0$. Then, as above, $(w')^{\pm 1}$ has a fundamental domain $v' \subset \widetilde w'$ whose initial segment coincides with that of $s'$. Since $|w'| > |s| + n_0$, we have $v' = s' \cdot y'$, where $|y'| > n_0$. 

We first bound the number of possibilities for $s'$. 
Since the endpoints of $s'$ are $J$--close to those of $s$, for every choice of $s$ there are at most $2 V^2$ choices for where  $s'$ begins and ends. (The factor of $2$ comes from the choice of direction of fellow-traveling.) Turning attention to $y'$, observe that the endpoint of $s'$ is the starting point of $y'$. Since $|y'| \geq n_0$,  \Cref{Thm:ConjugacyGrowth}.\eqref{Itm:TotalGrowth} and \Cref{Rem:GroupoidGrowth} imply that there are at most $B  e^{b|y'|}$ choices of where $y'$ ends. Thus, for every choice of $s$, the number of choices for  $(s'y')^{\pm 1}$ up to path-homotopy is at most 
\begin{equation}\label{Eqn:InitialCountS}
2\cdot V^2 \cdot B e^{b|y'|} \leq 2 V^2 B e^{b(|w| - |s| + 2 J)} \leq 2 V^2 B e^{b(\ell - |s| + 2 J)}.
\end{equation}

Recall that the companion $s'$ is determined by the $J$--loose cone-piece $s \subset \widetilde w$. There are $|w| \leq \ell$ possible choices of where  the projection of $s$ begins in the closed--geodesic $w$. There are also choices for the length $|s|$, constrained by the inequalities $d \leq |s| < |w'| \leq \ell$. Summing over these possible choices,  we 
conclude that for every conjugacy class $[g]$, the number of possibilities for $[g']$ such that the $J$--loose cone-piece $s$ has diameter at least $d$, is bounded above by
\begin{equation}\label{Eqn:ConePieceInitialCount}
\sum_{|s| = d}^\ell \ell \cdot 2 V^2 B e^{b(\ell - |s| + 2 J)}
<   \sum_{|s| = d}^\infty \ell \cdot 2 V^2 B e^{b(\ell - |s| + 2 J)}
= \frac{ 2 V^2 B \ell \, e^{b(\ell-d + 2J) )}}{ 1 - e^{-b} }.
\end{equation}
By increasing the constant $B$ if needed, we may ensure that the above exponential bound subsumes the linear bound in the non-generic warm-up case.

Next, recall that $\mathcal{G}(\ell)$ denotes the set of nontrivial conjugacy classes of translation length at most $\ell$. By \Cref{Thm:ConjugacyGrowth}.\eqref{Itm:ConjugacyGrowth}, the number of such conjugacy classes satisfies
\[
| \mathcal{G}(\ell) | \geq \frac{A}{\ell} e^{ b \ell} .
\]
Dividing the previous two equations produces a bound on the fraction of conjugacy classes 
$[g']$ whose representative closed--geodesic $w'$ has a $J$--loose cone-piece of diameter at least $d$ with $w$.
This fraction is at most
\[
\frac{ 2 V^2 B \ell e^{b(\ell - d + 2J )}}{ 1 - e^{-b} }  \cdot \frac{\ell} {A} e^{-b \ell}  
\: =  \: e^{-bd} \cdot  \ell ^2 \cdot \frac{ 2 V^2 B e^{2b J}}{ A (1 - e^{-b} )}  
\: = \: \: e^{-bd} \cdot  \ell ^2 \,  M.
\]
Here, $M = \frac{ 2 V^2 B e^{2b J}}{ A (1 - e^{-b} )} $ is a constant that depends only on $X$ and $J$.

\smallskip

The proof of Conclusion~\eqref{Itm:SelfPieceSmallOverlap} is very similar. Suppose that $s \subset \widetilde w$ is a $J$--loose cone piece between $\widetilde w$ and $h \widetilde w$, of diameter $|s| \geq d > 2J$. Since $d$ is large, \Cref{Def:OverlapOrientation} applies. Let $s' \subset h\widetilde w$ be the companion of $s$, and let $u = s \cap h^{-1}(s')$ be the overlap. By hypothesis, we have $|u| \leq R$. Let $(v')^{\pm 1} = s' \setminus h(u) = s' \setminus h(s)$ be the portion of $s'$ that is not in the image of the overlap. Then $s$ and $v'$ project to disjoint portions of the closed--geodesic $w \subset X$, and the sign $\pm 1$ is chosen so that $s$ and $v'$ are oriented consistently along $w$. Thus there is a fundamental domain for $w$ of the form $y_1 v' y_2 s$. 

Observe that $|y_1 v' y_2| \leq \ell - |s|$, and the subsegment $v'$ of length $|v'| \geq |s'| - R \geq |s| - 2J - R$ is entirely determined by $s$ and a bounded amount of extra data.
Thus, as in \Cref{Eqn:InitialCountS}, we compute that for every choice of $s$, the number of choices for $y_1 v' y_2$ up to path-homotopy is at most
\[
2 \ell \, V^2 B e^{b(\ell - 2|s| + 2J +R)}.
\]
Compared to \eqref{Eqn:InitialCountS}, the extra factor of $\ell$ comes from the choice of where in the fundamental domain the subword $v'$ occurs, and the extra constant $R$ in the exponent comes because we have removed the overlap from $s'$. Since there are at most $Be^{b|s|}$ choices for $s$, we conclude that the number of possibilities for $w$ is bounded above by
\[
2 \ell \, V^2 B^2 e^{b(\ell - |s| + 2J +R)} \leq 2 \ell \, V^2 B^2 e^{b(\ell - d + 2J +R)}.
\]

The remainder of the proof of \eqref{Itm:SeparateSample}, comparing the above upper bound to the lower bound  $\mathcal{G}(\ell) \geq (A/\ell) e^{b\ell}$, goes through verbatim. We conclude that the fraction of conjugacy classes $[g]$ whose representative closed--geodesic has a $J$--loose cone-piece with itself, with diameter $\geq d$ and overlap of length $\leq R$, is bounded by
\[
2 \ell \, V^2 B^2 e^{b(\ell - d + 2J +R)} \cdot \frac{\ell}{A} e^{-b \ell} = e^{-bd} \cdot \ell^2 \cdot \frac{ 2 V^2 B^2 e^{b(2J +R)}}{A}.
\]
Setting $M_0 = 2 V^2 B^2 e^{b(2J +R)}/A$ completes the proof.
\end{proof}

\begin{lem}\label{Lem:InfiniteWord}
Suppose $X$ is compact and $\widetilde X$ is $\delta$--hyperbolic.
Choose a non-trivial element $h \in \pi_1 X$.
 Suppose that $yx$ is a  path from $p \in \widetilde X$ to $hp \in \widetilde X$. Suppose that each of $x$ and $y$ is a geodesic, and that 
 $|x| \leq J$ for some $J \geq 0$.
Define $L$ to be the maximal length of a terminal segment of $y$ that is a  $2\delta$ fellow-traveler
 with an initial segment of $hy$.

Let $(yx)^\infty$ denote the bi-infinite path $\cdots h^{-1} (yx) h^0(y x) h^1(yx) h^2 (yx) \cdots$.
Then  $(yx)^\infty$ is an $\eta$--quasigeodesic, with $\eta$ depending only on $J$, $\delta$, and $L$.
\end{lem}

\begin{proof}
If $|y|$ is small, there are only finitely many possibilities for $yx$. The bi-infinite path $(yx)^\infty$ lies at bounded distance from some quasi-axis of $h$. So, take the worst case scenario for the quasigeodesic constant $\eta$.

If $|y|$ is large, we have large geodesic subpaths of $(yx)h(y)$. By hypothesis, we have a bound $L$ on the length of backtracking in $(yx)h(y)$.
Now,  apply a ``local to global'' principle for quasi-geodesics. See e.g.\ Cannon~\cite[Thm 4]{cannon:cocompact-hyp-groups}.
\end{proof}
%	Cannon's paper is technically about hyperbolic manifolds, but everyone knows the ideas extend
%	to the delta-hyperbolic setting. One could also cite Calegari \cite[Lemma 2.2.13]{calegari:ergodic-groups}.

We can now use Lemmas~\ref{Lem:ReversingPiecesSmall}, \ref{Lem:PieceNonGeneric}, and \ref{Lem:InfiniteWord} to show that large $J$--loose pieces between a closed--geodesic and itself are exponentially rare.

\begin{lem}\label{Lem:SelfPieceBigOverlap}
Suppose $X$ is compact and $\widetilde X$ is hyperbolic.
Fix a constant $J \in \naturals$. Then there is a constant $M = M(X, J)$ such that for all sufficiently large $\ell$ and for all $d \in [2J+1, \, \ell/3]$, the following holds.
 Among all conjugacy classes $[g]$ of length at most $\ell$, the proportion whose representative closed--geodesic has a $J$--loose cone piece with itself, with diameter $\geq d$, is 
  less than $M \ell^2  e^{-b d}$.
\end{lem}

\begin{proof}
Consider a closed--geodesic $w$ representing the conjugacy class $[g]$. As a warm-up, we dismiss the (non-generic) possibility where $|w| \leq 2 d$. Since $d \leq \ell/3$, it follows that $|w| \leq 2 \ell / 3$. By \Cref{Thm:ConjugacyGrowth}.\eqref{Itm:ConjugacyGrowth}, the total number of conjugacy classes of length at most $2 \ell/3$ is at most $B e^{2 \ell/3}$, whereas the total number of length at most $\ell$ is at least $A e^{b \ell}/\ell$. Thus, the probability that $|w|$ is at most $2/3$ of the allowed length is 
\[
\frac{\mathcal{G}(2\ell/3)}{\mathcal{G}(\ell)} \leq \frac{B e^{2 b\ell/3}}{A e^{b \ell}/\ell} = \frac{B\ell }{A} e^{-b\ell/3} \leq \frac{B}{A} \ell \,  e^{-bd} .
\]
Thus the conclusion of the lemma holds for $M_1 = B/A$. Note that this warm-up case did not use any hypotheses about pieces between $w$ and itself.

Now, suppose that a closed--geodesic $w$ representing the conjugacy class $[g]$ has a $J$--loose cone-piece with itself. Following \Cref{Def:LoosePiece}, let $s \subset \widetilde w$ be a preimage of the piece in $w$, where the endpoints of $s$ lie at distance $J$ from $h \widetilde w$. Following \Cref{Def:OverlapOrientation}, let $s' \subset h \widetilde w$ be the companion of $s$, and let $u = s \cap h^{-1}s'$ be the $h$--overlap in $s$.

Let $R = R(J, \delta) \geq 10J$ be the constant of \Cref{Lem:ReversingPiecesSmall}. By \Cref{Lem:PieceNonGeneric}.\eqref{Itm:SelfPieceSmallOverlap}, pieces with an overlap of diameter $|u| \leq R$ occur with probability at most $M_2 \ell^2  e^{-b d}$, where $M_2 = M_2 (X, J) = M_0(X, J, R(J,\delta))$ in the notation of \Cref{Lem:PieceNonGeneric}. Thus, we may suppose that $|u| > R$. By \Cref{Lem:ReversingPiecesSmall}, $h$ preserves the orientation on $s$. We now consider two cases.

\smallskip

\textbf{Case 1}: $|s|  \leq \frac{1}{2} |w|$. Without loss of generality, suppose that the overlap $u \subset s$ occurs at the beginning of $u$. Then there is a sub-geodesic $y_1 y_2 y_3 \subset \widetilde w$, where $s = y_2 y_3$ and $s' = h(y_1 y_2)$, so that $y_2 = s \cap h^{-1} s'$ is the $h$--overlap. 
Then $|y_2| \geq R \geq 10J$, where the second inequality comes from \Cref{Lem:ReversingPiecesSmall}. Consequently,
\[
| y_1| = |y_1 y_2| - |y_2| \leq |y_1 y_2| - 10J \leq (|y_2 y_3| + 2J) - 10J \leq  \tfrac{1}{2} |w| - 8J,
\]
and we can conclude that $|y_1 y_2 y_3| < |w|$. Thus we may choose a fundamental domain for $w$ of the form $y_1 y_2 y_3 z$.

Let $x_1$ denote a geodesic from the endpoint of $y_1$ to the start of $h(y_1)$, and let $x_3$ denote a geodesic from the endpoint of $y_3$ to the start of $h(y_3)$. Then $|x_1|, |x_3| \leq J$.
Since $y_1 y_2$ is a geodesic
and $h(y_1)$ fellow-travels with $y_2$, we see that in the path $y_1 x_1 h(y_1)$ there is a uniform upper bound
 on the amount of $(2\delta)$--fellow-traveling between the terminal subpath of $y_1$ and the initial subpath of $h(y_1)$.
Thus, in the notation of \Cref{Lem:InfiniteWord}, we  conclude that  $(x_1 y_1)^\infty$ is an $\eta$--quasigeodesic  for some $\eta  = \eta( \delta, J) >0$. 
Likewise,
 $(y_3 x_3)^\infty$ is an $\eta$--quasigeodesic.
These quasigeodesics fellow-travel, since they are periodically joined by translates of $y_2$. See \Cref{Fig:ConePieceBigOverlap}.

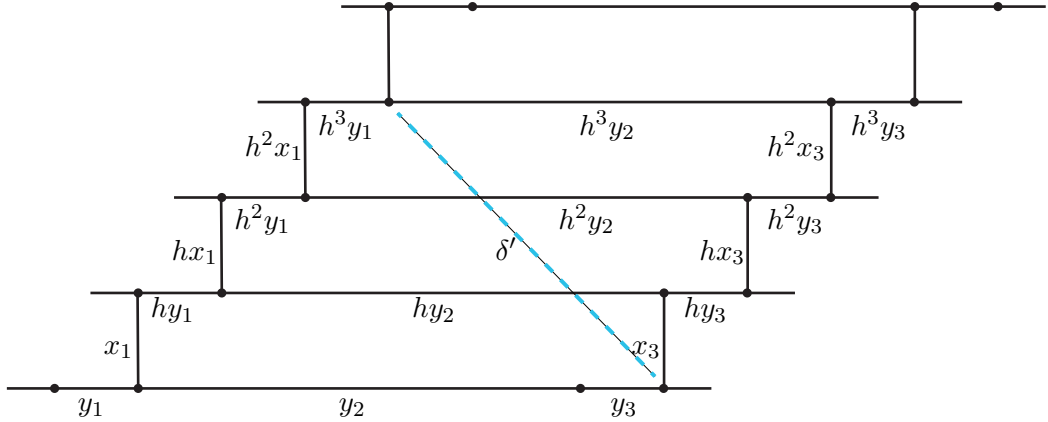
\begin{figure}

\begin{overpic}{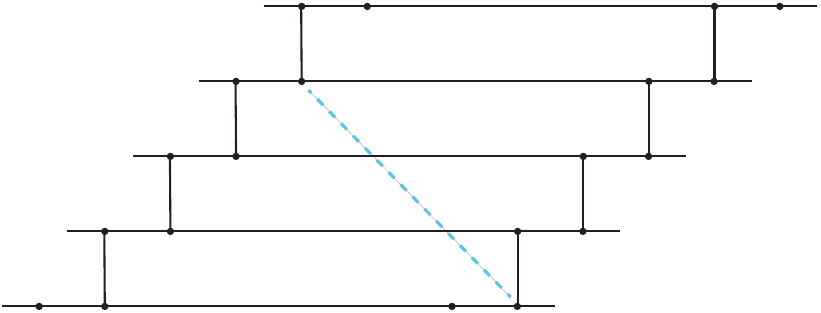}
\put(7,-1.5){$y_1$}
\put(32,-1.5){$y_2$}
\put(58,-1.5){$y_3$}
\put(9.5,4){$x_1$}
\put(60,4){$x_3$}
\put(14,7.5){$hy_1$}
\put(39,7.5){$hy_2$}
\put(65,7.5){$hy_3$}
\put(16,13){$h x_1$}
\put(66.5,13){$h x_3$}
\put(22, 16){$h^2 y_1$}
\put(53, 16){$h^2 y_2$}
\put(73, 16){$h^2 y_3$}
\put(23, 23){$h^2 x_1$}
\put(73, 23){$h^2 x_3$}
\put(30, 25){$h^3 y_1$}
\put(55, 25){$h^3 y_2$}
\put(81, 25){$h^3 y_3$}
\put(47,13){$\delta'$}

\end{overpic}

\caption{Bottom: $y_1 y_2$ is a geodesic, and $h y_1$ fellow-travels $y_2$. Hence there is a bounded amount of fellow-travelling between $y_1$ and $h y_1$. This implies $(y_1 x_1)^\infty$ is a quasigeodesic, and similarly for $(y_3 x_3)^\infty$. The two quasigeodesics must $\delta'$--fellow-travel, hence there is 
a geodesic of length $\delta'$ from the endpoint of $y_3$ to some point on $(y_1 x_1)^\infty$.}
\label{Fig:ConePieceBigOverlap}
\end{figure}

Let $\delta' = \delta'(\delta, J)$ be a uniform constant such that $\eta(\delta,J)$--quasigeodesic quadrilaterals in $\widetilde X$ must be $\delta'$--thin. Then it follows that $(x_1 y_1)^\infty$ and $(y_3 x_3)^\infty$ are $\delta'$--fellow-travelers. Since the endpoint of $y_3$ is $\delta'$--close to $(x_1 y_1)^\infty$,
we have an $\eta$--quasigeodesic triangle with sides consisting of a sub-path of $(x_1 y_1)^\infty$, a copy of $(y_2 y_3)$, and a geodesic segment of length at most $\delta'$ (shown dashed in  \Cref{Fig:ConePieceBigOverlap}). Since the dashed segment is $\delta'$-short, all of $y_2 y_3$ is $2 \delta'$-close to $(x_1 y_1)^\infty$.  Since  $|x_1|  \leq J$, all of $y_2 y_3$ is determined by $y_1$ and a universally bounded amount of data.
Since $| y_2 y_3 | \geq d$, the same argument as in  \Cref{Lem:PieceNonGeneric}.\eqref{Itm:SelfPieceSmallOverlap}, shows that the probability of a $J$--loose cone-piece with this configuration is at most $M_3 \ell^2  e^{-b d}$, for a constant $M_3 = M_3 (X, J)$.
 
 \smallskip
 
\textbf{Case 2}: $|s| > \frac{1}{2} |w|$. By the warm-up argument at the beginning of the proof, we may assume $d \leq \frac{1}{2} |w| < |s|$. As in Case 1, we may assume without loss of generality that the overlap $u$ occurs at the beginning of $s$. 
Let $\mathring s$ be the terminal sub-geodesic of $s$, of length  $| \mathring s| = d$. By \Cref{Lem:ThinQuad}, there is a sub-geodesic $\mathring s' \subset s'$ whose endpoints are $\mathring J$--close to those of $\mathring s$, for $\mathring J = J + 2\delta$. Let $\mathring u = \mathring s \cap h^{-1} (\mathring s')$. We think of $\mathring s$ as a sub-piece and $\mathring u$ as a sub-overlap.

If $|\mathring u| \leq R = R(J, \delta)$, observe that $\mathring s$ and $\mathring v' = \mathring s' \setminus h(\mathring u)$ project to disjoint portions of the closed--geodesic $w$, and that $\mathring v'$ is determined by $\mathring s$ and a universally bounded amount of extra data. Thus, by \Cref{Lem:PieceNonGeneric}.\eqref{Itm:SelfPieceSmallOverlap}, the probability of a $J$--loose cone-piece containing this configuration is at most $M_4 \ell^2  e^{-b d}$, where $M_4 = M_4 (X, J) = M_0(X, \mathring J, R(J,\delta))$ in the notation of \Cref{Lem:PieceNonGeneric}.

If $|\mathring u| \geq R = R(J, \delta)$, we employ the argument of Case 1, with $\mathring J$ in place of $J$, to show that the probability of a $J$--loose cone-piece with this configuration is at most $M_5 \ell^2  e^{-b d}$, for a constant $M_5 = M_5 (X, J)$.

\smallskip

To complete the proof, recall the constants $M_1, \ldots, M_5$ defined above, so that the probability of a $J$--loose cone-piece of each type is bounded by $M_j \ell^2  e^{-b d}$ for the appropriate $M_j$. Setting $M = \max_{j} M_j$ completes the proof.
\end{proof}

\begin{prop}\label{Prop:FewConePieces}
Suppose $X$ is compact and $\widetilde X$ is hyperbolic.
Fix constants $J \in \naturals$ and $\phi \in (0,1)$ and $C > 2/b(1-\phi)$. Then, for all sufficiently large $\ell$ and for all $d \in [C \log \ell, \, \ell/3]$, the following holds.

Suppose that conjugacy classes $[g],[g']$ are chosen  at random from among those of length at most $\ell$. Then there is an $e^{-bd\phi}$ upper bound on the probability that there is a $J$--loose piece of diameter at least $d$ between the representative closed--geodesics.

Similarly, for a randomly chosen conjugacy class $[g]$ of length at most $\ell$, there is an $e^{-bd\phi}$ upper bound on the probability that there is a $J$--loose piece of diameter at least $d$ between the representative closed--geodesic and itself.
\end{prop}

\begin{proof}
Let  $M = M(X,J)$ be the larger of the two constants in \Cref{Lem:PieceNonGeneric}.\eqref{Itm:SeparateSample} and \Cref{Lem:SelfPieceBigOverlap}. Combining the two lemmas, we see that a $J$--loose cone-piece of diameter at least $d$ in the closed--geodesic representing $[g]$ occurs with probability at most $M \ell^2 \,  e^{-bd}$.

Now, let $C > \frac{2}{b(1-\phi)}$ be as in the statement of the lemma. Then,  choosing $\ell$ large ensures that the additive difference $\big( C \log \ell - \frac{ 2\log \ell}{b(1-\phi)}  \big)$ is as large as we like. In particular, for $\ell \gg 0$ and  $d \geq C \log \ell$,
we have
\[
d \geq C \log \ell > \frac{2 \log \ell }{b (1-\phi) } + \frac{ \log M }{b (1-\phi) } = \frac{\log(M \ell^2)}{b (1-\phi) }.
\]
After exponentiating and rearranging terms, we obtain
\[
M \ell^2  < e^{bd(1-\phi)}, \qquad \text{hence} \qquad M \ell^2 e^{-bd} <  e^{-b \phi d }. \qedhere
\]
\end{proof}

\subsection{Controlling pieces among many conjugacy classes}\label{Sec:MainProof}

We can now apply \Cref{Prop:FewConePieces} to bound the probability of a large $J$--loose cone-piece between a pair of conjugacy classes that are sampled from a large collection, where ``large'' is defined as a fraction of the systole. See \Cref{Prop:ConePieceSystole}. We can then combine this result with \Cref{Prop:WallPieceControl} to prove the main theorem.

\begin{lem}\label{Lem:ConePieceDiameter}
Fix constants $q \in \big( 0,\frac{1}{6} \big]$
 and $J \in \naturals$. Let $k \leq e^{c  \ell}$, where $c< q b$ and $b$ is the growth exponent of $\widetilde X$.
Select conjugacy classes $[g_1], \ldots, [g_k]$ at random from among those of length at most $\ell$, with each $[g_i]$ represented by a closed--geodesic $w_i$.

Then, with overwhelming probability as $\ell \to \infty$, all $J$--loose cone-pieces among the $w_i$
have diameter strictly less than $2 q \ell$.
\end{lem}

\begin{proof}
Since $c< q b$, we may choose a constant  $\phi \in (0,1)$  such that $c < q b \phi^2$.
We will consider pieces of diameter at least $d\geq 2 q \ell$. Since $d$ is bounded below by a linear function of $\ell$,  the logarithmic hypothesis on $d$ in
\Cref{Prop:FewConePieces} is satisfied  for $\ell \gg 0$.

The above choices imply that there are $k^2 \leq e^{2c  \ell} < e^{2 q b  \phi^2 \ell}$ pairs of indices $(i,j)$. When $\ell$ is sufficiently large,  \Cref{Prop:FewConePieces} says that  for every pair $(i,j)$, the probability of a $J$--loose cone-piece of diameter at least $d$ between $w_i$ and $w_j$ is less than $e^{-bd\phi}$. 
(This includes the case $w_i = w_j$.)
Thus the total probability that some pair has a $J$--loose cone-piece of diameter $ \geq d$ is less than
\[
k^2 e^{-bd\phi}  < e^{2 q b  \phi^2 \ell} e^{-bd\phi}
= e^{ (2 q \phi \ell -d  ) \, b \phi    }.
\]
Recall that $d \geq 2 q  \ell$ and $\phi \in (0, 1)$.
Then, as $\ell \to \infty$, the exponent in the above probability estimate is bounded as follows:
\[
(2 q \phi \ell -d  ) \, b \phi   \leq  ( 2 q  \phi \ell -  2 q  \ell  ) \, b \phi   = (\phi - 1) \cdot 2q \phi \cdot b \ell \longrightarrow - \infty.
\]
We conclude that with overwhelming probability as $\ell \to \infty$, there are no $J$--loose pieces of diameter  $d \geq 2 q  \ell$.
\end{proof}

\begin{lem}\label{Lem:Systole}
Select conjugacy classes $[g_1], \ldots, [g_k]$ at random from among those of length at most $\ell$, where
$ k \leq e^{c\ell}$ for some constant $c < q b$ where $0<q<1$. Then, with overwhelming probability as $\ell \to \infty$, we have:
 $\min\{ | g_i | : 1 \leq i \leq k \} \geq (1-q)\ell$.
\end{lem}
\begin{proof}
Let $w_i$ be a closed--geodesic representing $[g_i]$.
By \Cref{Thm:ConjugacyGrowth}.\eqref{Itm:ConjugacyGrowth}, the conditional probability that $|w_i|\leq (1-q) \ell$ given that
$|w_i| \leq \ell$ is
\[
\frac{\mathcal{G}((1-q)\ell)}{\mathcal{G}(\ell)} \leq 
\frac{Be^{b (1-q) \ell}}{A e^{b\ell}/ \ell} = \frac{B}{A}\ell e^{-bq\ell}.
\]
Thus the conditional probability that   $|w_i|<(1- q) \ell$ for some $1\leq i\leq k$
is bounded above by $( k\frac{B\ell}{A} )e^{-b q\ell}$. This upper bound approaches zero
exponentially quickly when $k\leq e^{c\ell}$ and $c< q b$.
\end{proof}

\begin{prop}\label{Prop:ConePieceSystole}
Fix $J \in \naturals$ and $\alpha \in \big( 0, \frac{2}{5} \big]$.
Suppose conjugacy classes $[g_1], \ldots, [g_k]$ are chosen at random from among all those of length at most $\ell$. Assume that $k \leq e^{c\ell}$ for some constant $c < b\alpha/(\alpha+2)$. Then, with overwhelming probability as $\ell \to \infty$,
all $J$--loose cone-pieces among the closed--geodesics representing  $[g_1], \ldots, [g_k]$ 
have diameter strictly less than $ \alpha |g_i |$ for every $i$.
\end{prop}

\begin{proof}
We will use Lemmas~\ref{Lem:ConePieceDiameter} and~\ref{Lem:Systole} with $q = \frac{\alpha}{\alpha+2}$. Observe that the hypothesis $\alpha \leq \frac{2}{5}$ implies $q \leq \frac{1}{6}$, as required for \Cref{Lem:ConePieceDiameter}.

With this value of $q$, \Cref{Lem:ConePieceDiameter} says that with overwhelming probability as $\ell \to \infty$, all cone-pieces have diameter strictly less than $ \frac{2 \alpha}{\alpha+2} \ell $.
Meanwhile, \Cref{Lem:Systole} says that with overwhelming probability as $\ell \to \infty$, $\min_{i=1}^k \{|g_i | \}   \geq (1-q) \ell = \frac{2}{\alpha+2} \ell$.

Putting together the last two results, we conclude that the diameter of every cone-piece in $w_i$ is less than $\alpha |w_i | $.
\end{proof}

In the same spirit as \Cref{Lem:Systole}, we have

\begin{lem}\label{Lem:Primitive}
Suppose $k \leq e^{c \ell}$ for some constant $c < b/2$. Then, with overwhelming probability as $\ell \to \infty$, a set of $k$ randomly chosen conjugacy classes has the property that each one is primitive.
\end{lem}

\begin{proof}
By \Cref{Rem:NonPrimitive}, the probability that $[g_i]$ is non-primitive is bounded above by
\[
\frac{\ell^2 B}{A} e^{-b\ell/2} .
\]
Thus the probability that a $k$--tuple of conjugacy classes contains a non-primitive class is at most
\[
k \frac{ \ell^2 B}{A} e^{-b \ell/2} \leq \frac{ \ell^2 B}{A} e^{(c-b/2)\ell},
\]
which goes to $0$ as $\ell \to \infty$ because $(c-b/2) < 0$.
\end{proof}

%    \subsection{Proof of the main theorem}\label{Sec:MainProof}

We can now restate and prove \Cref{Thm:main}.

\begin{named}{Theorem~\ref{Thm:main}}
Let $G=\pi_1X$, where $X$ is a compact non-positively curved cube complex, and suppose that $G$ is hyperbolic.
Let $b$ be the growth exponent of $G$ with respect to $\widetilde X$,  
and let $a$ be the maximal growth exponent of a stabilizer of an essential hyperplane of $\widetilde X$.
Let $k\leq  e^{c \ell}$, where
\[ c<\min \left\{ \frac{(b-a)}{20}, \,  \frac{b}{41} \right\} .
\]
Then with overwhelming probability as $\ell \to \infty$, for any set of conjugacy classes
$[g_1], \ldots, [g_k]$ with each $|g_i| \leq \ell$,
the group $\overline{G} = G/ \nclose{g_1, \ldots, g_k}$ is hyperbolic and is the fundamental group of a compact, non-positively curved cube complex.
\end{named}

\begin{proof}
We may assume that $G$ is non-elementary; otherwise, $a=b=0$ and the theorem holds vacuously.
We begin by replacing $X$ with a closely related cube complex that satisfies \Cref{Conv:NiceX}.
First,  we perform a cubical subdivision of $X$, while retaining the original metric.
Then every conjugacy class  $[g] \subset G$ can be assigned a closed--geodesic representative $w \rightarrow X$. 
Second, if some hyperplane of $X$ is inessential, we replace $X$ by its essential core, as in \cite[Proposition 3.5]{CapraceSageev2011}, which has the same growth exponents $a$ and $b$.

For every relator $[g_i]$ in the statement of the theorem, let $w_i \to X$ be the chosen closed--geodesic. Let $\widetilde w_i \subset \widetilde X$ be a geodesic axis that covers $w_i$, stabilized by $g_i \in [g_i]$. For each $\widetilde w_i$, let 
$\widetilde Y_i = \hull(\widetilde w_i) \subset \widetilde X$,
whose quotient $Y_i =  \langle g_i \rangle \backslash \widetilde Y_i$ admits a local isometry into $X$.
By \Cref{Lem:HullUniformThickness},  $\widetilde Y_i$ lies in a uniform neighborhood of $\widetilde w_i$, hence its quotient $Y_i$ is compact and a quasi-circle. By construction, every hyperplane of $\widetilde Y_i = \hull(\widetilde w_i)$ cuts $\widetilde w_i$. 
This gives us a cubical presentation  $X^* = \langle X \mid Y_1,\ldots, Y_k\rangle$ such that $\pi_1 X^* = \overline{G}$.

Define $J = J(\widetilde X)$ as in \Cref{Lem:LoosenUp}. Then every diameter $ \geq d$ cone-piece of $X^*$ between $Y_i$ and $Y_j$ corresponds to a $J$--loose cone-piece between $w_i$ and $w_j$, of diameter  $\geq d$. Similarly, every diameter $ \geq d$ wall-piece of $X^*$ in $w_i$ corresponds to a $J$--loose wall-piece, also of diameter  $\geq d$.

The hypotheses
 on $c$ allow us to choose a constant $\alpha < \frac{1}{20}$, such that $c < \alpha (b-a)$ and $c < \frac{b \alpha}{\alpha + 2}$.
 Indeed, $c < \frac{1}{20} (b-a)$ and
 \(
 c < \frac{b/20}{1/20 + 2} = \frac{b}{20} \cdot \frac{20}{41} .
 \)
 Then Propositions~\ref{Prop:WallPieceControl} and \ref{Prop:ConePieceSystole} ensure that with overwhelming probability as $\ell \to \infty$,  the $J$--loose (wall or cone) pieces in every $w_i$ have diameter strictly less than $\alpha |w_i| = \alpha \systole{Y_i}$. Thus, by \Cref{Lem:LoosenUp},
  the wall-pieces and cone-pieces
 in a relator $Y_i$ of $X^*$ also
have diameter strictly less than $\alpha \systole{Y_i}$. Thus $X^*$ satisfies the $C'(\alpha)$ small-cancellation condition. 
Since $\alpha < \frac{1}{14}$, \Cref{lem:1/14 hyperbolic} ensures that $\pi_1X^*$ is hyperbolic. 

Next, we check the hypotheses of \Cref{Thm:C'20Proper}. We have verified that with overwhelming probability, $X^*$ is $C'(\frac{1}{20})$ and that every $Y_i$ is a compact quasi-circle that deformation retracts to a closed--geodesic $w_i$.
By \Cref{Lem:LoosenUp}.\eqref{Itm:HyperplaneInY}, every hyperplane  $U \subset Y_i$ has a carrier $N(U)$ of diameter strictly less than $\alpha \systole{Y_i}$, which implies that $N(U)$ is embedded. Since $U$ must cut the closed--geodesic $w_i$, we also conclude that $Y_i \setminus U$ is contractible. Finally,  \Cref{Lem:Primitive} implies that with overwhelming probability, every $w_i$ is  primitive. 
Therefore,  \Cref{Thm:C'20Proper} ensures that $\pi_1 X^*$ acts freely and cocompactly on the CAT(0) cube complex dual to the wallspace on $\widetilde{X^*}$. 
\end{proof}

\section{Generalization to other metric spaces}\label{Sec:OtherMetric}

In this section, we prove \Cref{Thm:OtherSample}, which generalizes \Cref{Thm:main} to the setting of groups acting on other, non-cubical metric spaces. The idea is to prove a probabilistic statement for quotients $\overline{G} = G/ \nclose{g_1, \ldots, g_k}$ where the relators $g_i$ are sampled from all short conjugacy classes in a $G$--action on some metric space $\Upsilon$. As discussed in the introduction, there are many situations (for instance, hyperbolic manifolds) where the most natural geometry associated to a group $G$ is carried by some metric space other than a cube complex. The results of this section enable us to draw conclusions using the growth of $\Upsilon$ rather than the growth of a cube complex.

\subsection{Cube-free definitions and results}

\begin{conv}\label{Conv:NiceUpsilon}
The following assumptions and terminology hold throughout this section. Let $G$ be a nonelementary, torsion-free group acting properly and cocompactly on a $\delta$--hyperbolic geodesic metric space $\Upsilon$. 
%    \begin{com} Used to say: ``We assume that $\Upsilon$ is a cell complex.'' But we only really need the orbit of a basepoint. \end{com}
Let $H_1, \ldots, H_m$ be a collection of infinite index quasiconvex subgroups of $G$, which will remain fixed for the rest of this section. (In the main case of interest, the $H_i$ are hyperplane stabilizers of some action of $G$ on 
a CAT(0) cube complex $\widetilde X$.) Fix a basepoint $\upsilon \in \Upsilon$.

We assume that every nontrivial element $g \in G$ stabilizes a geodesic axis $\widetilde \gamma \subset \widetilde \Upsilon$. 
For every conjugacy class $[g]$, we choose a representative closed--geodesic $\gamma \subset G \backslash \Upsilon$. Then every $g \in [g]$ stabilizes some preimage $\widetilde \gamma$ of $\gamma$.
\end{conv}

\begin{defn}[Loose pieces in $\Upsilon$]\label{Def:LoosePieceOtherMetric}
Fix a constant $J > 0$. Consider bi-infinite geodesics $\widetilde \gamma, \widetilde \gamma' \subset \Upsilon$ that do not share an endpoint in $\bdy  \Upsilon$. Observe that $\widetilde \gamma \cap \neb_{J} (\widetilde \gamma')$ is a closed set because $\neb_{J} (\widetilde \gamma')$ is a closed neighborhood, and is bounded because the geodesics do not share an endpoint. Thus $\widetilde \gamma \cap \neb_{J} (\widetilde \gamma')$ is compact.
A \emph{$J$--loose cone-piece between $\widetilde \gamma$ and $\widetilde \gamma'$} is the maximal geodesic segment $s \subset \widetilde \gamma$ whose endpoints lie in $\neb_{J} (\widetilde \gamma')$. 
The \emph{companion of $s$} is the maximal segment $s'  \subset \widetilde \gamma'$ whose endpoints are at distance $J$ from the corresponding endpoints of $s$. 

Now, let $\gamma, \gamma'$ be closed--geodesics in $G \backslash \Upsilon$. Then a \emph{$J$--loose cone-piece} between $\gamma$ and $\gamma'$ is the projection of a $J$--loose cone-piece between arbitrary preimages $\widetilde \gamma$ and $\widetilde \gamma'$, excluding the case where $\widetilde \gamma = \widetilde \gamma'$.

A \emph{$J$--loose wall-piece in $\widetilde \gamma$} is a maximal geodesic segment  $s \subset \widetilde{\gamma}$ whose endpoints are contained in $\neb_J (g H_i p)$ for one of the chosen subgroups $H_i$ and for some $g \in G$. A \emph{$J$--loose wall-piece in a closed--geodesic $\gamma$} is the projection of a $J$--loose wall-piece in $\widetilde \gamma$.
\end{defn}

\begin{defn}[Cube-free $C'(\alpha)$ presentations]
\label{Def:AuxiliarySmallCancel}

Let $\Upsilon$, $\delta$, $\upsilon$, $G$, and $H_1, \ldots, H_m$ be as in \Cref{Conv:NiceUpsilon}.
Let $\kappa_j$ be the quasiconvexity constant of the orbit $H_j \upsilon$, and let $\kappa = \max_j \{ \kappa_j \} + 2 \delta$. 

Let $g_1, \ldots, g_k$ be infinite-order elements of $G$. For each $g_i$, let $\gamma_i \to G \backslash \Upsilon$ be the chosen closed--geodesic representing $[g_i]$. Since each  $g_i$ stabilizes an axis $\widetilde \gamma_i \subset \Upsilon$, we have $|\gamma_i| = |g_i|_{{}_\Upsilon} = \stabletrans{g_i}_{{}_\Upsilon}$ in \Cref{Def:TransLength}.

The presentation $\langle G : H_1, \ldots, H_m \mid g_1, \ldots, g_k \rangle$ is called \emph{$C'(\alpha)$ with respect to $(\Upsilon, \upsilon, \gamma_1, \ldots, \gamma_k)$} if
\begin{enumerate}[\:\: $(1)$]
\item\label{Itm:GeneralConePiece} For every $g_i, g_j$, the diameter of any $2\delta$--loose cone-piece between $\gamma_i$ and $\gamma_j$ is less than  $ \alpha \stabletrans{g_i}_{{}_\Upsilon}$.

\item\label{Itm:GeneralWallPiece} For every $g_i$, any $\kappa$--loose wall-piece in $\gamma_i$ has diameter 
less than  $ \alpha \stabletrans{g_i}_{{}_\Upsilon}$.
\end{enumerate}
\end{defn}

The results of \Cref{Sec:Pieces} have the following generalization to this context. 

\begin{prop}\label{Prop:GeneralizedConePiece}
Let $G$, $\Upsilon$, and $\{ H_i \}$ be as in \Cref{Def:AuxiliarySmallCancel}.
Let $b$ be the growth exponent of $G$ acting on $\Upsilon$.
Consider conjugacy classes $[g_1], \ldots, [g_k]$ in $G$, chosen uniformly from all conjugacy classes of $\Upsilon$--length at most $\ell$. Suppose $ k \leq e^{c\ell}$ for some constant $c < b\alpha/(\alpha+2)$.
Then, with overwhelming probability as $\ell \to \infty$, property \eqref{Itm:GeneralConePiece} of \Cref{Def:AuxiliarySmallCancel} holds.
\end{prop}

\begin{prop}\label{Prop:GeneralizedWallPiece}
Let $G$, $\Upsilon$, and $\{ H_i \}$ be as in \Cref{Def:AuxiliarySmallCancel}. Let $b$ be the growth exponent of $G$, and let $a$ be an upper bound on the growth exponents of the $H_j$. Note that all lengths and growths are measured with respect to the action on $\Upsilon$.

Consider conjugacy classes $[g_1], \ldots, [g_k]$ in $G$, chosen uniformly from all conjugacy classes of $\Upsilon$--length at most $\ell$. Suppose $k \leq e^{c \ell}$ for some $c < \alpha(b-a)$. Then, with overwhelming probability as $\ell \to \infty$, property \eqref{Itm:GeneralWallPiece}  of \Cref{Def:AuxiliarySmallCancel} holds.
\end{prop}

\begin{proof}[Sketch]
The proofs of 
Propositions~\ref{Prop:GeneralizedConePiece} and~\ref{Prop:GeneralizedWallPiece} 
are nearly identical to those of Propositions~\ref{Prop:ConePieceSystole} and~\ref{Prop:WallPieceControl}, respectively. 
There are two differences in the argument.
The primary difference is that $\Upsilon$ is substituted for $\widetilde{X}$, and hyperplane stabilizers are replaced by general quasiconvex subgroups $H_i$. The requirement that every infinite-order element of $\pi_1 X$ stabilizes a geodesic axis (compare \Cref{Conv:NiceX}), which was heavily used in the proofs of Propositions~\ref{Prop:ConePieceSystole} and~\ref{Prop:WallPieceControl}, is mirrored in our setting by the same requirement in $\Upsilon$.
Note that in both Propositions~\ref{Prop:WallPieceControl} and~\ref{Prop:ConePieceSystole}, we obtain genericity statements via the counts of \Cref{Thm:ConjugacyGrowth}, which apply perfectly well to the $G$--action on $\Upsilon$.

The second difference is more subtle. In the arguments of \Cref{Sec:Pieces}, a loose piece $s \subset \widetilde w$ is always a combinatorial geodesic segment whose endpoints are at vertices of $\widetilde X$. Thus, when we factor a fundamental domain of $w$ as $s \cdot y$, the parts $s$ and $y$ are both combinatorial geodesics, and the number of possibilities for $y$ can be estimated via \Cref{Thm:ConjugacyGrowth} and \Cref{Rem:GroupoidGrowth}. 
%    While the metric space $\Upsilon$ is assumed to be a cell complex, 
Meanwhile, in $\Upsilon$,
the set of endpoints of loose pieces $s \subset \widetilde \gamma$ might be locally infinite. To enable counting arguments, we make the following adjustment: when we factor a fundamental domain for $\gamma$ as $s \cdot y$, we perturb both $s$ and $y$ so that they begin and end at points in the $G$--orbit of the basepoint $\upsilon$. This adjustment perturbs lengths by a bounded additive error, which becomes absorbed into the multiplicative constants of calculations such as 
\eqref{Eqn:WallPieceCyclicPerm} and \eqref{Eqn:InitialCountS}. Thus perturbing $\widetilde \gamma$ to pass through the orbit $G\upsilon$ does not affect the probabilistic conclusions.
\end{proof}

Combining Propositions~\ref{Prop:GeneralizedConePiece} and~\ref{Prop:GeneralizedWallPiece} gives:

\begin{cor}\label{Cor:SmallCancelUpsilon}
Let $G$, $\Upsilon$, and $\{ H_i \}$  be as in \Cref{Def:AuxiliarySmallCancel}. Let $b$ be the growth exponent of $G$ with respect to $\Upsilon$, and let $a$ be an upper bound on the growth exponents of the $H_i$.
As above, choose a basepoint $\upsilon \in \Upsilon$, and a representative closed--geodesic in $G \backslash \Upsilon$ for every conjugacy class in $G$.

Let $k \leq e^{c \ell}$, where $c < \min \{  \alpha(b-a),  b\alpha/(\alpha+2)   \}$, and
consider conjugacy classes $[g_1], \ldots, [g_k]$ in $G$, chosen uniformly from all conjugacy classes of $\Upsilon$--length at most $\ell$. Then, with overwhelming probability as $\ell \to \infty$, the presentation  $\langle G : H_1, \ldots, H_m \mid g_1, \ldots, g_k \rangle$ is $C'(\alpha)$ with respect to $(\Upsilon, \upsilon, \gamma_1, \ldots, \gamma_k)$.
\end{cor}

\subsection{Translation back to cube complexes}

In \Cref{Sec:Pieces}, we used \Cref{Lem:LoosenUp} to show that every piece in a cubical presentation is associated with a corresponding $J$--loose piece, in the sense of \Cref{Def:LoosePiece}. The following statement is an analogue of \Cref{Lem:LoosenUp} that allows us to compare pieces in a cubical presentation $\langle X  \mid Y_1, \ldots, Y_k \rangle$ to loose pieces in a $G$--action on $\Upsilon$.

\begin{prop}\label{Prop:QITranslation}
Let $G = \pi_1 X$, where $X$ is a compact non-positively curved cube complex whose immersed hyperplanes are essential, and let $H_1, \ldots, H_m$ be their fundamental groups.
Suppose that $G$ acts properly and cocompactly on a $\delta$--hyperbolic geodesic metric space $\Upsilon$ that admits a $G$--equivariant $\lambda$--quasiisometry from $\widetilde{X}$.

Let $[g_1], \ldots, [g_k]$ be conjugacy classes in $G$. For each $i$, let $\widetilde{Y_i} = \hull(\widetilde w_i)$, 
where $\widetilde w_i$ is an axis for $g_i$ in $\widetilde{X}$. Finally, let $Y_i  = \langle g_i \rangle \backslash \widetilde{Y_i}$.

Suppose that, for some $\overline{\lambda} > \lambda$, the presentation $\langle G : H_1, \ldots, H_m \mid g_1, \ldots, g_k \rangle$ is $C'(\alpha/ \overline{\lambda})$ with respect to $(\Upsilon, p, \gamma_1, \ldots, \gamma_k)$.
Then there exists a constant $M = M(\alpha, \delta, \kappa, \lambda, \overline{\lambda},  X)$ such that whenever $\stabletrans{g_i}_{{}_\Upsilon} \geq M$ for all $i$, the following holds.
The cubical presentation $\langle X  \mid Y_1, \ldots, Y_k \rangle$ is $C'( \alpha)$. Moreover, every hyperplane  $U \subset Y_i$ has a carrier $N(U)$ of diameter strictly less than $\alpha \systole{Y_i}$.
\end{prop}

\begin{proof}
We begin by recalling the above definition of a $\lambda$--quasiisometry. By \Cref{Eqn:Lambda12}, there is a ($G$--equivariant) function $f \from \widetilde X \to \Upsilon$, along with positive constants $\lambda_1, \lambda_2, \epsilon$ such that $\lambda_1 \lambda_2 = \lambda$ and
\begin{equation}\label{Eqn:LambdaRestate}
\frac{1}{\lambda_1} \dist_{\widetilde X}(x,y) - \epsilon \: \leq \: \dist_{\Upsilon} (f(x), f(y)) \: \leq \: \lambda_2 \dist_{\widetilde X}(x,y) + \epsilon.
\end{equation}
We can now relate the systole of $Y_i$ in $X$ to the (stable) translation length of $g_i$ in $\Upsilon$.
By \Cref{Def:Piece} and \Cref{Def:TransLength}, we have 
\[
\systole{Y_i} =  |g_i|_{{}_X} \geq \stabletrans{g_i}_{{}_X} = \lim_{n \to \infty} \frac{\dist_{\widetilde X}(q, g_i^n q)}{n},
\]
for an arbitrary point $q \in \widetilde X$. Letting $q' = f(q) \in \Upsilon$ yields
\begin{equation}\label{Eqn:SystoleComparison}
\lambda_2 \systole{Y_i} \geq \lambda_2 \! \lim_{n \to \infty} \frac{\dist_{\widetilde X}(q, g_i^n q)}{n} 
=  \lim_{n \to \infty} \frac{\lambda_2 \dist_{\widetilde X}(q, g_i^n q) + \epsilon}{n} \
\geq  \lim_{n \to \infty} \frac{\dist_\Upsilon(q', g_i^n q')}{n}    = \stabletrans{g_i}_{{}_\Upsilon} .
\end{equation}

There is a constant $\eta = \eta(\delta, \lambda_1, \lambda_2, \epsilon)$ with the following properties. First, for every convex set $S \subset \widetilde X$, the image  $f(S) \subset \Upsilon$ is $\eta$--quasiconvex. Second, for every bi-infinite geodesic $\widetilde w \to \widetilde X$, the quasigeodesic image $f(\widetilde w)$ is an $\eta$--fellow-traveler with every geodesic that has the same endpoints in $\bdy \Upsilon$.

\smallskip

Suppose that there is a diameter $d$ cone-piece in $X^*$ between $Y_i$ and $Y_j$, where $d$ is very large (the precise criterion will be described below). Following \Cref{Def:Piece}, this piece is a component of  $\widetilde Y_i \cap \widetilde Y_j$ for appropriate preimages $\widetilde Y_i$ and $\widetilde Y_j$. Let $x,y \in \widetilde Y_i \cap \widetilde Y_j$ be points such that $\dist_{\widetilde X}(x,y) = d$. By \Cref{Lem:HullUniformThickness}, there are points $x_i, y_i \in  \widetilde w_i$ and $x_j, y_j \in \widetilde w_j$ that are $K$--close to $x$ and $y$, respectively. 

The following construction in $\Upsilon$ is illustrated in  \Cref{Fig:InfinitelyManyPrimes}.
Let $\widetilde \gamma_i, \widetilde \gamma_j$ be the representative axes for $g_i, g_j$, respectively, in $\Upsilon$.
Let $\widetilde \varpi_i = f(\widetilde w_i)$ be the image of $\widetilde w_i$ in $\Upsilon$, and observe that $\varpi_i$ lies in the $\eta$--neighborhood of the axis $\widetilde \gamma_i$. Similarly, let $\widetilde \varpi_j =  f(\widetilde w_j)$, and observe that $\widetilde \varpi_j$ lies in the $\eta$--neighborhood of the axis $\widetilde \gamma_j$.
Label points $x_i' =f(x_i), y_i' = f(y_i) \in \widetilde \varpi_i$ and $x_j' = f(x_j), y_j' = f(y_j) \in \widetilde \varpi_j$.
Finally, let $x_i'' , y_i'' \in \widetilde \gamma_i$ be points within $\eta$ of $x_i', y_i' $, and similarly define $x_j'' , y_j'' \in \widetilde \gamma_j$.

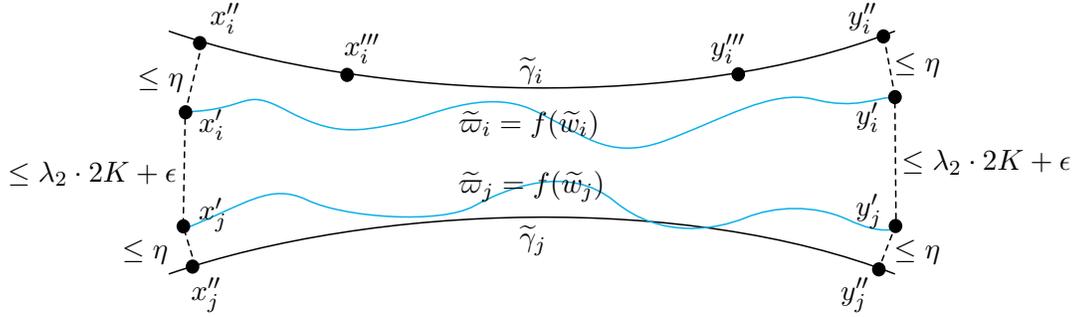
\begin{figure}
\begin{overpic}[width=4in]{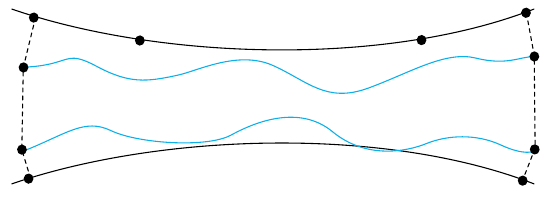}
\put(-4,4.5){$\leq \eta$}
\put(-19,14.5){$\leq \lambda_2 \cdot 2K + \epsilon$}
\put(-2,27){$\leq \eta$}
\put(97,4.5){$\leq \eta$}
\put(98,16){$\leq \lambda_2 \cdot 2K + \epsilon$}
\put(97,29){$\leq \eta$}
\put(40,21){$\widetilde \varpi_i = f(\widetilde w_i)$}
\put(40,13){$\widetilde \varpi_j = f(\widetilde w_j)$}
\put(48,6){$\widetilde \gamma_j$}
\put(48,28){$\widetilde \gamma_i$}
\put(6,9.5){$x_j'$}
\put(6,21){$x_i'$}
\put(92,10){$y_j'$}
\put(92,22){$y_i'$}
\put(5,-1){$x_j''$}
\put(90,-1){$y_j''$}
\put(7.5,35){$x_i''$}
\put(91,35){$y_i''$}
\put(25,31){$x_i'''$}
\put(73,31){$y_i'''$}
\end{overpic}
\caption{The figure lives in $\Upsilon$. The blue quasi-geodesics $\widetilde \varpi_i$ and $\widetilde \varpi_j$  are  the images of axes in $\widetilde X$.  The black geodesics are axes $\widetilde \gamma_i$ and $\widetilde \gamma_j$, respectively. Since $ \widetilde \varpi_i$ must $\eta$--fellow-travel with $ \gamma_i$, there is a point $x_i''$ that is $\eta$--close to $x_i'$, and similarly for the others. The $2\delta$--loose cone-piece between $\widetilde \gamma_i$ and $\widetilde \gamma_j$ will contain the segment $[x_i''', y_i''']$.}
\label{Fig:InfinitelyManyPrimes}
\end{figure}

Observe that the second inequality in \Cref{Eqn:LambdaRestate} gives
\[
\dist_\Upsilon(x_i'', x_j'') \leq \dist_\Upsilon(x_i', x_j') + 2 \eta \leq \lambda_2 \dist_{\widetilde X}(x_i, x_j) + \epsilon + 2 \eta \leq \lambda_2 \cdot 2K + \epsilon + 2 \eta.
\]
and similarly $\dist_\Upsilon(y_i'', y_j'') \leq \lambda_2 \cdot 2K + \epsilon + 2 \eta$. At the same time, the first inequality in \Cref{Eqn:LambdaRestate} gives
\begin{equation}\label{Eqn:PieceDistortionUpsilon}
\dist_\Upsilon(x_i'', y_i'') \geq \dist_\Upsilon(x_i', y_i') - 2 \eta \geq \frac{1}{\lambda_1} \dist_{\widetilde X}(x_i, y_i) - \epsilon - 2 \eta \geq \frac{d - 2K}{\lambda_1} - \epsilon - 2 \eta.
\end{equation}
Thus, when $d$ is sufficiently large, we have $\dist_\Upsilon(x_i'', y_i'') > 2 \lambda_2 K + \epsilon + 2 \eta + 2\delta$.

Next, observe that the geodesic quadrilateral with vertices $x_j'', x_i'', y_i'', y_j''$ is $2\delta$--thin. Thus every point of  $[x_i'', y_i''] \subset \widetilde \gamma_i$ must be $2\delta$--close to some other side of the quadrilateral. Let $x_i''' \in [x_i'', y_i'']$ be the point such that $\dist_\Upsilon(x_i'', x_i''') = 2 \lambda_2 K + \epsilon + 2 \eta + 2\delta$. Since we have taken $d$ to be large, this point $x_i'''$ exists and is far from $y_i''$.
Then the side that is $2\delta$--close to $x_i'''$ must be on $\widetilde \gamma_j$. 
The same conclusion holds for the point $y_i''' \in [x_i'', y_i'']$  such that $\dist_\Upsilon(y_i'', y_i''') =  2 \lambda_2  K + \epsilon + 2 \eta + 2\delta$. In particular, the entire segment $[x_i''',y_i''']$ lies in  $\neb_{2\delta}(\widetilde \gamma_j)$. Hence there is a $2\delta$--loose cone-piece in $\Upsilon$, of diameter at least
\[
\dist_\Upsilon(x_i''', y_i''') = \dist_\Upsilon(x_i'', y_i'') - 2(2 \lambda_2  K + \epsilon + 2 \eta + 2\delta).
\]

Since $\langle G : H_1, \ldots, H_m \mid g_1, \ldots, g_k \rangle$ is $C'(\alpha / \overline{\lambda})$ with respect to $(\Upsilon, p, \gamma_1, \ldots, \gamma_k)$, we get
\begin{align*}
\frac{\alpha} { \: \overline{\lambda} \:} \stabletrans{g_i}_{{}_\Upsilon}
& > \dist_\Upsilon(x_i''', y_i''') \\
& = \dist_\Upsilon (x_i'', y_i'') - 2(2 \lambda_2  K + \epsilon + 2 \eta + 2\delta)  \\
& \geq  \dist_\Upsilon(x_i', y_i') -  2 \eta - 2(2 \lambda_2  K + \epsilon + 2 \eta + 2\delta) \\
& =  \dist_{\Upsilon}(x_i', y_i') - (4 \lambda_2  K + 6 \eta  + 4 \delta + 2 \epsilon ) \\
&    \geq \frac{1}{\lambda_1}   \dist_{\widetilde X}(x_i, y_i) - \epsilon - (4 \lambda_2  K + 6 \eta  + 4 \delta + 2 \epsilon )     \\
& \geq \frac{1}{\lambda_1} (\dist_{\widetilde X}(x,y) - 2K)   - (4 \lambda_2  K + 6 \eta  + 4 \delta + 3 \epsilon )  \\
& = \frac{1}{\lambda_1} (d - 2K)   - (4 \lambda_2  K + 6 \eta  + 4 \delta + 3 \epsilon ).
\end{align*}

Combining the last computation with \Cref{Eqn:SystoleComparison}, we obtain
\[
  \frac{ \lambda_1 \lambda_2}{  \overline \lambda } \, \alpha   \systole{Y_i} \geq
  \frac{ \lambda_1}{\:  \overline \lambda \:} \, \alpha  \stabletrans{g_i}_{{}_\Upsilon}
> d \: - 2K   - \lambda_1 (4 \lambda_2  K + 6 \eta  + 4 \delta + 3 \epsilon ).
\]
Since $\lambda_1 \lambda_2 = \lambda < \overline{\lambda} $,  there is a constant $M_1$ such that when $\stabletrans{g_i}_{{}_\Upsilon} \geq M_1$, we have
\[
 \alpha \systole{Y_i} >   \frac{ \lambda_1 \lambda_2}{\:  \overline \lambda \:} \, \alpha   \systole{Y_i}  +  2K   + \lambda_1 (4 \lambda_2  K + 6 \eta  + 4 \delta + 3 \epsilon )
> d.
\]
In other words, the diameter of every cone-piece in $Y_i$ is less than $\alpha  \systole{Y_i}$.
This establishes the cone-piece portion of the claim that $\langle X \mid Y_1, \ldots, Y_k \rangle$ is $C'(\alpha )$.

\medskip

For the wall-piece portion of the desired conclusion, suppose that $P = \widetilde Y_i \cap N(\widetilde U_j)$, where $\widetilde U_j$ is a hyperplane of $\widetilde X$, and suppose that $\diameter(P) = d$ is very large. 
(Note that $P$ may or may not be disjoint from $\widetilde Y_i$, hence might not be a wall-piece, according to \Cref{Def:Piece}. Compare \Cref{Rem:MoreGeneralHyperplane}.)
 Let $x,y \in  P$ be points realizing the diameter. By \Cref{Lem:HullUniformThickness}, there are points $x_i, y_i \in  \widetilde w_i$
  that are $K$--close to $x$ and $y$, respectively.
  Similarly, there are points $x_j, y_j \in \widetilde U_j$
  that are $1$--close to $x$ and $y$, respectively.

Let $\widetilde \varpi_i = f(\widetilde w_i)$ be the image of $\widetilde w_i$ in $\Upsilon$, and observe that $\widetilde \varpi_i$ lies in the $\eta$--neighborhood of the axis $\widetilde \gamma_i$.
Let $x_i'=f(x_i), y_i'=f(y_i)$ be points in $\widetilde \varpi_i$, and let $x_i'' , y_i'' \in \widetilde \gamma_i$ be points within $\eta$ of $x_i', y_i' $.

Let $\Theta_j \subset \Upsilon$ be the image of $\widetilde U_j$ under the quasi-isometry, and let $x_j'=f(x_j), y_j'=f(y_j) \in  \Theta_j$ be the images of $x_j, y_j$. Recall that $\upsilon \in \Upsilon$ is the pre-chosen basepoint, and that each orbit $H_j \upsilon$ is $\kappa_j$--quasiconvex. 
Note that $\Theta_j  \in \neb_\psi(g H_j \upsilon)$ for some $g \in G$, where $\psi$ depends on $X, \lambda_1, \lambda_2, \epsilon, \kappa_j$.
Let $x_j'' , y_j'' \in g H_j \upsilon$ be points that are $\psi$--close to $x_j', y_j'$, respectively.

First, observe that
\begin{align*}
\dist_\Upsilon(x_i'', x_j'')
& \leq \dist_\Upsilon(x_i'', x_i') +  \dist_\Upsilon(x_i', x_j') + \dist_\Upsilon(x_j', x_j'') \\
& \leq \eta + (\lambda_2 \dist_{\widetilde X}(x_i, x_j) + \epsilon) + \psi \\
& \leq \eta + (\lambda_2 (K + 1) + \epsilon)  + \psi.
\end{align*}
The same estimate holds  for $\dist_\Upsilon(y_i'', y_j'')$.

The geodesic quadrilateral with vertices $x_j'', x_i'', y_i'', y_j''$ is $2\delta$--thin. Thus every point of  $[x_i'', y_i''] \subset \widetilde \gamma_i$ must be $2\delta$--close to some other side of the quadrilateral. Let $x_i''' \in [x_i'', y_i'']$ be the point such that $\dist_\Upsilon(x_i'', x_i''') =  \eta + \lambda_2 (K + 1) + \epsilon  + \psi + 2\delta$. (As with cone pieces, such a point $x_i'''$ exists whenever $d$ is sufficiently large.)
Then the side that is $2\delta$--close to $x_i'''$ must be the opposite side, namely $[x_j'', y_j'']$. The same conclusion holds for the point $y_i''' \in [x_i'', y_i'']$ such that $\dist_\Upsilon(y_i'', y_i''') =   \eta + \lambda_2 (K + 1) + \epsilon  + \psi + 2\delta$. Thus
\[
x_i''',y_i''' \:  \in \: \widetilde \gamma_i \cap \neb_{2\delta}([x_j'', y_j''])
\: \subset \: \widetilde \gamma_i \cap \neb_{\kappa}(g H_j \upsilon).
\]
The last containment uses the property $\kappa \geq 2\delta + \kappa_j$, where $g H_j \upsilon$ is $\kappa_j$--quasiconvex.

Since $\langle G : H_1, \ldots, H_m \mid g_1, \ldots, g_k \rangle$ is $C' ( \alpha / \overline{\lambda} )$ with respect to $(\Upsilon, p, \gamma_1, \ldots, \gamma_k)$, we get
\begin{align*}
\frac{\alpha}{\: \overline{\lambda} \:} \stabletrans{g_i}_{{}_\Upsilon}
& > \dist_\Upsilon(x_i''', y_i''') \\
& = \dist_\Upsilon (x_i'', y_i'') - 2( \eta + \lambda_2 (K + 1) + \epsilon  + \psi + 2\delta)  \\
& \geq  \dist_\Upsilon(x_i', y_i')  - 2 ( \eta + \lambda_2 (K + 1) + \epsilon  + \psi  + 2\delta)  -  2 \eta \\
% 2( \eta + \lambda_2 (K + 1) + \epsilon  + \psi + 2\delta) \\
& =  \dist_\Upsilon(x_i', y_i') -  2 ( 2\eta + \lambda_2 (K + 1) + \epsilon  + \psi  + \delta )  \\
&    \geq \frac{1}{\lambda_1}   \dist_{\widetilde X}(x_i, y_i)  - \epsilon - 2 ( 2\eta + \lambda_2 (K + 1) + \epsilon  + \psi  + \delta )    \\
& \geq \frac{1}{\lambda_1} (\dist_{\widetilde X}(x,y) - 2K)   - 2 ( 2\eta + \lambda_2 (K + 1) + 2\epsilon  + \psi  + \delta ) \\
& = \frac{1}{\lambda_1} (d - 2K)   - 2 (2 \eta + \lambda_2 (K + 1) + 2\epsilon  + \psi  + \delta ) .
\end{align*}

Combining the last computation with \Cref{Eqn:SystoleComparison}, we obtain
\[
  \frac{ \lambda_1 \lambda_2}{  \overline \lambda } \, \alpha   \systole{Y_i} \geq
  \frac{ \lambda_1 }{ \: \overline \lambda \: } \, \alpha  \stabletrans{g_i}_{{}_\Upsilon}
> d \: - 2K   - 2 \lambda_1 (2 \eta + \lambda_2 (K + 1) + 2\epsilon  + \psi  + \delta ).
\]

Since $\lambda_1\lambda_2 = \lambda < \overline{\lambda} $,  there is a constant $M_2$ such that when $\stabletrans{g_i}_{{}_\Upsilon} \geq M_2$, we have
\[
 \alpha \systole{Y_i} >   \frac{ \lambda_1 \lambda_2}{  \overline \lambda } \, \alpha   \systole{Y_i}  +  2K   + 2 \lambda_1 ( 2\eta + \lambda_2 (K + 1) + 2\epsilon  + \psi  + \delta )
> d. 
\]
Thus $d = \diameter(P) <  \alpha  \systole{Y_i}$. This bounds the size of hyperplane carriers in $Y_i$, as well as wall-pieces involving $Y_i$. We conclude that when $\stabletrans{g_i}_{{}_\Upsilon} \geq M = \max(M_1, M_2)$
 the cubical presentation is $C'(\alpha)$.
\end{proof}

\subsection{Main result}

We can now restate and prove \Cref{Thm:OtherSample}. After the proof, we discuss a potential strengthening.

\begin{named}{Theorem~\ref{Thm:OtherSample}}
Let $G=\pi_1X$, where $X$ is a compact non-positively curved cube complex, and suppose that $G$ is hyperbolic. Suppose that $G$ also acts properly and cocompactly on a geodesic metric space $\Upsilon$, where every non-trivial element of $G$ stabilizes a geodesic axis. Suppose that there is a $G$--equivariant $\lambda$--quasiisometry   $\widetilde X \to \Upsilon$.

Let $b$ be the growth exponent of $G$ with respect to $\Upsilon$,  
and let $a$ be the maximal growth exponent in $\Upsilon$ of a stabilizer of an essential hyperplane of $\widetilde X$.
Let $k\leq  e^{c \ell}$, where
\[ c<\min \left\{ \frac{(b-a)}{20 \lambda}, \,  \frac{b}{40 \lambda + 1} \right\} .
\]
Then with overwhelming probability as $\ell \to \infty$, for any set of conjugacy classes
$[g_1], \ldots, [g_k]$ with each $|g_i|_{{}_\Upsilon} \leq \ell$,
the group $\overline{G} = G/ \nclose{g_1, \ldots, g_k}$ is hyperbolic and is the fundamental group of a compact, non-positively curved cube complex.
\end{named}

\begin{proof}
Observe that our hypotheses on $c$ can be restated as 
 \[
 c < (b-a) \frac{1/20}{ \lambda} 
 \quad \text{and} \quad
 c < \frac{b  (\frac{1}{ 20}) / \lambda }{ (\frac{1}{ 20}) / \lambda + 2}  .
 \]
By continuity, we may choose constants $\alpha < \frac{1}{20}$ and $\overline{\lambda} > \lambda$ such that
\[
c < (b-a) \frac{\alpha }{\:  \overline{\lambda} \: } 
 \quad \text{and} \quad
c < \frac{b \alpha/ \overline{\lambda}} { \alpha/ \overline{\lambda} + 2}.
\]

As in the proof of \Cref{Thm:main}, we replace $X$ by its essential core, so that all hyperplanes are essential. This does not affect the multiplicative constant $\lambda$ in the quasi-isometry to $\Upsilon$. Let $H_1, \ldots, H_m$ be the hyperplane stabilizers in $G = \pi_1 X$. We also subdivide $X$ while retaining the original metric, so that every non-trivial conjugacy class is represented by a closed--geodesic. Both operations preserve the growth exponents $a$ and $b$. 
\Cref{Cor:SmallCancelUpsilon} implies that with overwhelming probability as $\ell \to \infty$, the presentation  $\langle G : H_1, \ldots, H_m \mid g_1, \ldots, g_k \rangle$ is $C'(\alpha/ \overline{\lambda})$ with respect to $\Upsilon$ and any choice of basepoint and axes. 

For each $i$, let  $\widetilde w_i \subset \widetilde X$ be a geodesic axis stabilized by $g_i \in [g_i]$. For each $\widetilde w_i$, let $\widetilde Y_i = \hull (  \widetilde w_i ) \subset \widetilde{X}$,
whose quotient $Y_i =  \langle g_i \rangle \backslash \widetilde Y_i$ admits a local isometry into $X$.
By \Cref{Lem:HullUniformThickness},  $\widetilde Y_i$ lies in a uniform neighborhood of $\widetilde w_i$, hence its quotient $Y_i$ is compact and a quasi-circle. By construction, every hyperplane of $\widetilde Y_i$ cuts $\widetilde w_i$. This gives us a cubical presentation  $X^* = \langle X \mid Y_1,\ldots, Y_k\rangle$ such that $\pi_1 X^* = \overline{G}$.

Now, \Cref{Prop:QITranslation} implies that  with overwhelming probability as $\ell \to \infty$, the cubical presentation $X^* = \langle X \mid Y_1, \ldots, Y_k \rangle$ is $C'(\alpha)$ with $\alpha < \frac{1}{20}$. Furthermore, every hyperplane $U$ of every $Y_i$ satisfies $\diameter{N(U)} < \alpha \systole{Y_i}$.
Since $\alpha < \frac{1}{14}$, \Cref{lem:1/14 hyperbolic} ensures that $\pi_1X^*$ is hyperbolic. 

Next, we check the hypotheses of \Cref{Thm:C'20Proper}. We have verified that with overwhelming probability, $X^*$ is $C'(\frac{1}{20})$ and that every $Y_i$ is compact and deformation retracts to a closed--geodesic.
By \Cref{Prop:QITranslation}, every hyperplane  $U \subset Y_i$ has a carrier $N(U)$ of diameter strictly less than $\alpha \systole{Y_i}$, which implies that $N(U)$ is embedded. Since $U$ must cut the closed--geodesic $w_i$, we also conclude that $Y_i \setminus U$ is contractible. Finally, the same argument as in \Cref{Lem:Primitive} implies that with overwhelming probability, every $g_i$ is  primitive. 
Therefore,  \Cref{Thm:C'20Proper} ensures that $\pi_1 X^*$ acts freely and cocompactly on the CAT(0) cube complex dual to the wallspace on $\widetilde{X^*}$. 
\end{proof}

\begin{rem}\label{Rem:ProbabilisticQITranslation}
In the above proof of \Cref{Thm:OtherSample}, all of the probabilistic arguments happen inside \Cref{Cor:SmallCancelUpsilon}, which combines Propositions~\ref{Prop:GeneralizedConePiece} and~\ref{Prop:GeneralizedWallPiece}. 
By contrast, \Cref{Prop:QITranslation} involves a global assumption (a $G$--equivariant $\lambda$--quasiisometry), and the proof works entirely in the language of coarse geometry without invoking any counting or probability. 
One can envision a strengthening of \Cref{Prop:QITranslation} that tracks the behavior of a typical conjugacy class and a typical piece. 

To make this precise, choose a basepoint $x \in \widetilde X$ and let $B_n(\widetilde X) = \{ g \in G : \dist_{\widetilde X}(x, gx) \leq n\}$.
Define $B_n(\Upsilon)$ similarly, and
recall from \Cref{Def:Growth} that $f_{G, \Upsilon}(n)$ counts the cardinality of $B_n(\Upsilon)$.
Define the \emph{mean distortion of $\widetilde X$ with respect to $\Upsilon$} to be
\begin{equation*}\label{Eqn:MeanDistortion}
%    \tau(\Upsilon, \widetilde X) = \lim_{n \to \infty} \frac{1}{f_{G,\widetilde X}(n)} \sum_{g \in B_n(\widetilde X)} \!\! \frac{\dist_{\Upsilon}(\upsilon, g\upsilon)}{n},
%    \qquad
\tau(\widetilde X / \Upsilon) \: = \: \lim_{n \to \infty} \frac{1}{f_{G,\Upsilon}(n)} \sum_{g \in B_n(\Upsilon)} \!\! \frac{\dist_{\widetilde X}(x, gx)}{n}
  \: = \: \lim_{n \to \infty} \frac{1}{f_{G,\Upsilon}(n)} \sum_{g \in B_n(\Upsilon)} \!\! \frac{\dist_{\widetilde X}(x, gx)}{\dist_{\Upsilon}(\upsilon, g\upsilon)}.
\end{equation*}
In words, $\tau(\widetilde X / \Upsilon)$ measures the average  factor by which an element of $G$ sampled using $\Upsilon$ gets stretched in $\widetilde X$.
%    , and vice versa for $\tau(\widetilde X / \Upsilon)$. 
See~\cite[Eqn (1.1) and Thm 1.2]{CantrellTanaka:ManhattanCurve} for a proof that the limit exists.

The mean distortion can be bounded as follows. Suppose there is a $G$--equivariant $\lambda$--quasiisometry $f \from \widetilde X \to \Upsilon$, as in
\Cref{Eqn:LambdaRestate}, so that $f(x) = \upsilon$. Then
\[
\frac{1}{\lambda_1} \leq \frac{1}{\tau(\widetilde X / \Upsilon)} \leq \lambda_2.
\]

For a chosen $\epsilon > 0$, an element $g \in G$ or a conjugacy class $[g]$ is called \emph{$\Upsilon$--typical}  if 
\begin{equation}\label{Eqn:UpsilonTypical}
\left(\tau(\widetilde X / \Upsilon) - \epsilon\right) \stabletrans{g}_{{}_\Upsilon} \leq \stabletrans{g}_{{}_X} \leq 
\left(\tau(\widetilde X / \Upsilon) + \epsilon\right) \stabletrans{g}_{{}_\Upsilon}.
\end{equation}
For each $\epsilon > 0$, a conjugacy class sampled uniformly from all those of $\Upsilon$--length at most $\ell$ will be $\Upsilon$--typical with overwhelming probability. Indeed, if $\Upsilon$ is the Cayley graph of $G$ with respect to some generating set, this follows from a large 
deviation result of Cantrell and Tanaka~\cite[Thm 4.23]{CantrellTanaka:ManhattanCurve}. If $\Upsilon$ is itself a CAT(0) cube complex, this follows from a large deviation theorem of Cantrell and Reyes~\cite[Thm 1.4]{CantrellReyes:RigidityCubes}.
For general $\Upsilon$, the referee informs us
 that this can be derived from Cantrell and Reyes~\cite[Eqn (5.2)]{CantrellReyes:MarkedLengthSpectrum}.

We now discuss the prospects for removing $\lambda$ from the statement of \Cref{Prop:QITranslation}. The proof begins by observing that every conjugacy class $[g_i]$ satisfies $ \stabletrans{g_i}_{{}_X} \geq  (\lambda_2)^{-1} \stabletrans{g_i}_{{}_\Upsilon}$; see \Cref{Eqn:SystoleComparison}. Since $[g_i]$ is $\Upsilon$--typical with overwhelming probability, we may replace $(\lambda_2)^{-1}$ by $\big( \tau(\widetilde X / \Upsilon) - \epsilon \big)$.
Next, the proof of \Cref{Prop:QITranslation} uses the constant $\lambda_1$ to pass from the diameter of a piece in $X^*$ to the diameter of a loose piece in $\Upsilon$; see \Cref{Eqn:PieceDistortionUpsilon}. If one knew that a cone-piece or wall-piece in $X^*$, coming from a random presentation as in the statement of \Cref{Thm:OtherSample}, also corresponds to an $\Upsilon$--typical group element of $G$, 
one could replace $\lambda_1$ by $\big(\tau(\widetilde X / \Upsilon) + \epsilon \big)$. The upshot would be that the product $\lambda = \lambda_1 \lambda_2$ would be replaced by the quotient $\big(\tau(\widetilde X / \Upsilon) + \epsilon \big) /\big( \tau(\widetilde X / \Upsilon) - \epsilon \big)$, which approaches $1$ as $\epsilon \to 0$.

Whether the pieces coming from a random presentation (sampled using lengths in $\Upsilon$) can be represented by $\Upsilon$--typical group elements is an interesting problem. See  \Cref{Ques:UpsilonTypicalPieces}.
\end{rem}

\section{Pentagonal surfaces}\label{Sec:Pentagonal}

Throughout this section, we consider the setting where $\Upsilon = \HH^2$ is the hyperbolic plane, equipped with a tiling $T$ by regular right-angled pentagons, and $\widetilde X$ is the square complex dual to this tiling. We work out the optimal constant $\lambda$ in a $\lambda$--quasiisometry from  $\widetilde X$ to $\Upsilon$. This can be considered the first interesting example where \Cref{Thm:OtherSample} applies.

\begin{prop}\label{Prop:PentagonalQI}
Let $T$ be the tiling of $\HH^2$ by regular right-angled pentagons, with hyperbolic metric $\dist_\HH$. The dual tiling $T^*$, with five quadrilaterals at every vertex, can be identified with a CAT(0) square complex $\widetilde X$, with combinatorial metric $\dist_{\widetilde X}$. Then the identity map $\operatorname{id} \from (\widetilde X, \dist_{\widetilde X}) \to (\HH^2, \dist_\HH)$ is a $\lambda$--quasiisometry, where
\[
\lambda = \frac{  \arccosh(2 K^2 + 2K + 1) }{\arccosh(K+1) } \approx 1.5627 
\qquad \text{for} \qquad K = \cos \big( \tfrac{2\pi}{5} \big).
\]
Furthermore, this value of $\lambda$ is optimal.
\end{prop}

As mentioned in the introduction, combining  \Cref{Thm:OtherSample} with \Cref{Prop:PentagonalQI} and the classical work of Huber \cite{Huber} yields a proof of \Cref{Cor:H2Action}.

The proof of \Cref{Prop:PentagonalQI} is entirely elementary and largely pictorial; see Figures~\ref{Fig:PentagonLengths}--\ref{Fig:LengthThreeSegments}. The proof also trades the coarse geometry that has dominated most of this paper for the fine geometry of $\HH^2$. We begin with \Cref{Lem:PentagonLengths}, which computes a number of lengths in a single right-angled pentagon using the hyperbolic laws of sines and cosines \cite{Fenchel:HyperbolicGeometry, Ratcliffe:Foundations}. We then continue with \Cref{Prop:PentagonalRestate}, which considers a number of combinatorial possibilities for how a hyperbolic geodesic segment can cross a sequence of several adjacent pentagons.

\begin{lem}\label{Lem:PentagonLengths}
Let $P \subset \HH^2$ be a regular right-angled pentagon. Label lengths as in \Cref{Fig:PentagonLengths}. Set $K = \cos \big( \frac{2\pi}{5} \big)$. Then the labeled lengths can be expressed as follows:
\begin{align}
a &= \arccosh \big( \sqrt{K+1} \big) 
%    = \arcsinh \big( \sqrt{K} \big)
 \approx 0.5306. \label{Eqn:ALength} \\
b &= \arccosh \left( \tfrac{1}{\sqrt{1-K}} \right) \approx 0.6269. \label{Eqn:BLength} \\
c &= \arccosh(K+1) \approx  0.7672. \label{Eqn:CLength} \\
d &= \arccosh(2 K^2 + 2K + 1) \approx 1.1989. \label{Eqn:DLength} \\
e &= \arccosh \big( (2K+1)\sqrt{K+1} \big) \approx 1.2265. \label{Eqn:ELength} \\
f &= \arccosh(4K^2 + 4K+1) \approx 1.6169. \label{Eqn:FLength}
\end{align}

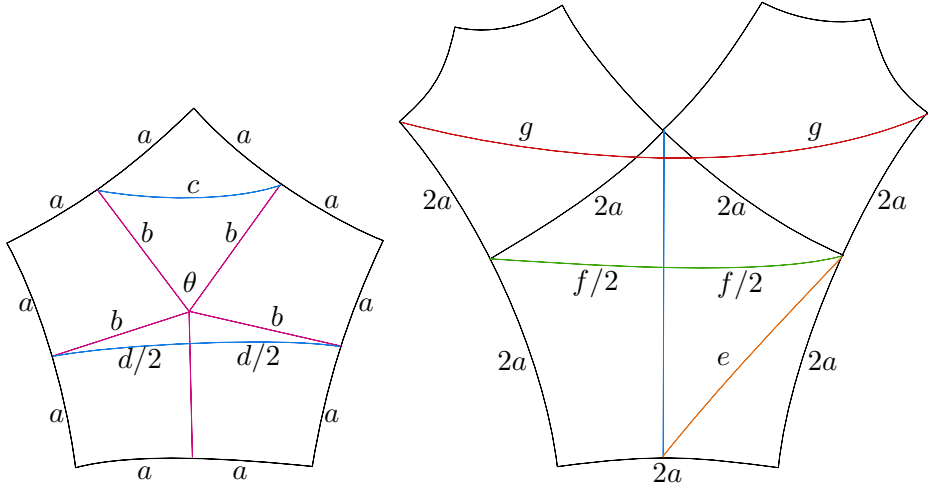
\begin{figure}
\begin{overpic}[width=2in]{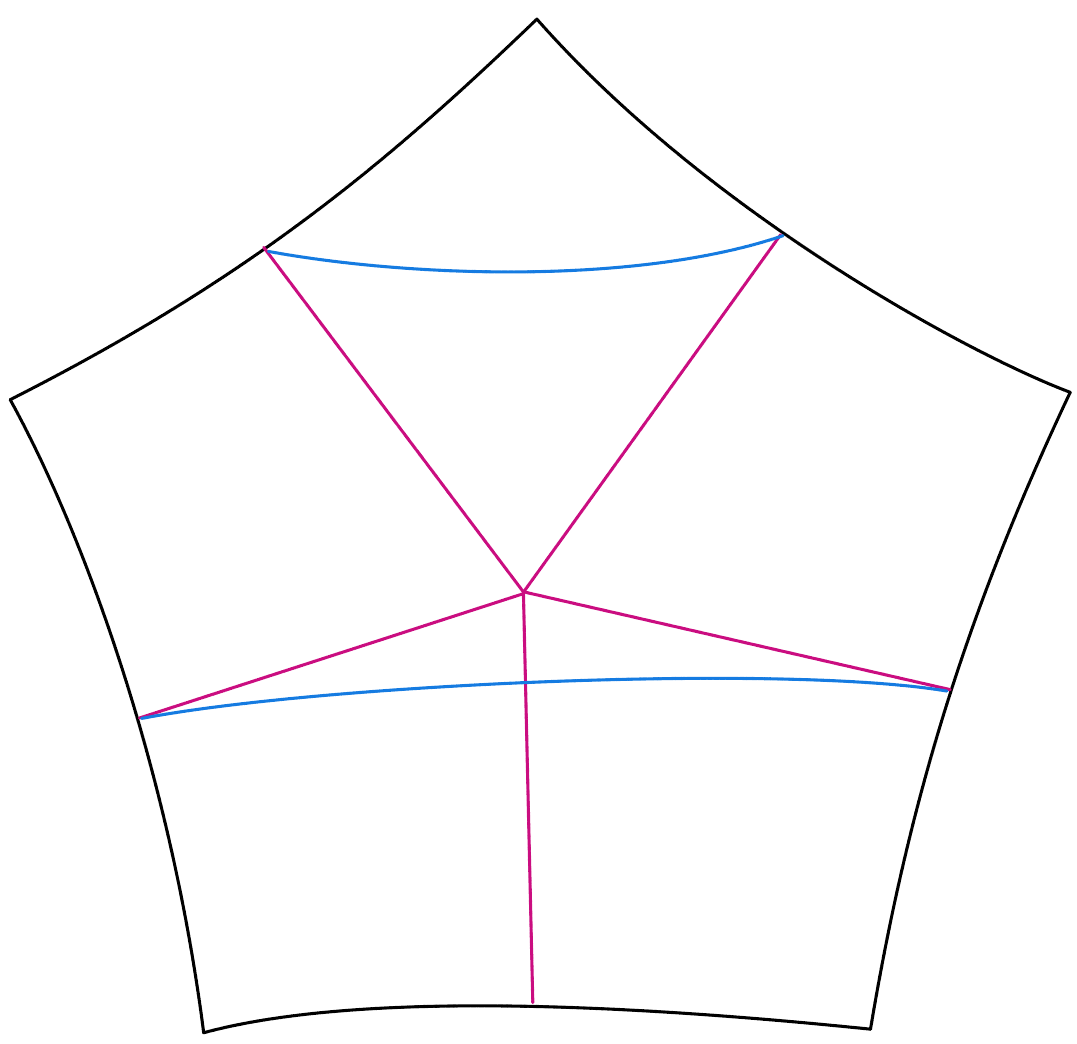}
\put(35,-2){$a$}
\put(60,-2){$a$}
\put(12,13){$a$}
\put(84,13){$a$}
\put(4,42){$a$}
\put(93,42){$a$}
\put(12,69){$a$}
\put(84,69){$a$}
\put(35,86){$a$}
\put(61,86){$a$}
\put(36,60){$b$}
\put(58,60){$b$}
\put(28,37){$b$}
\put(70,38){$b$}
\put(48,73){$c$}
\put(30,27){$d/2$}
\put(61,28){$d/2$}
\put(47,47){$\theta$}
\end{overpic}
\begin{overpic}[width=2.8in]{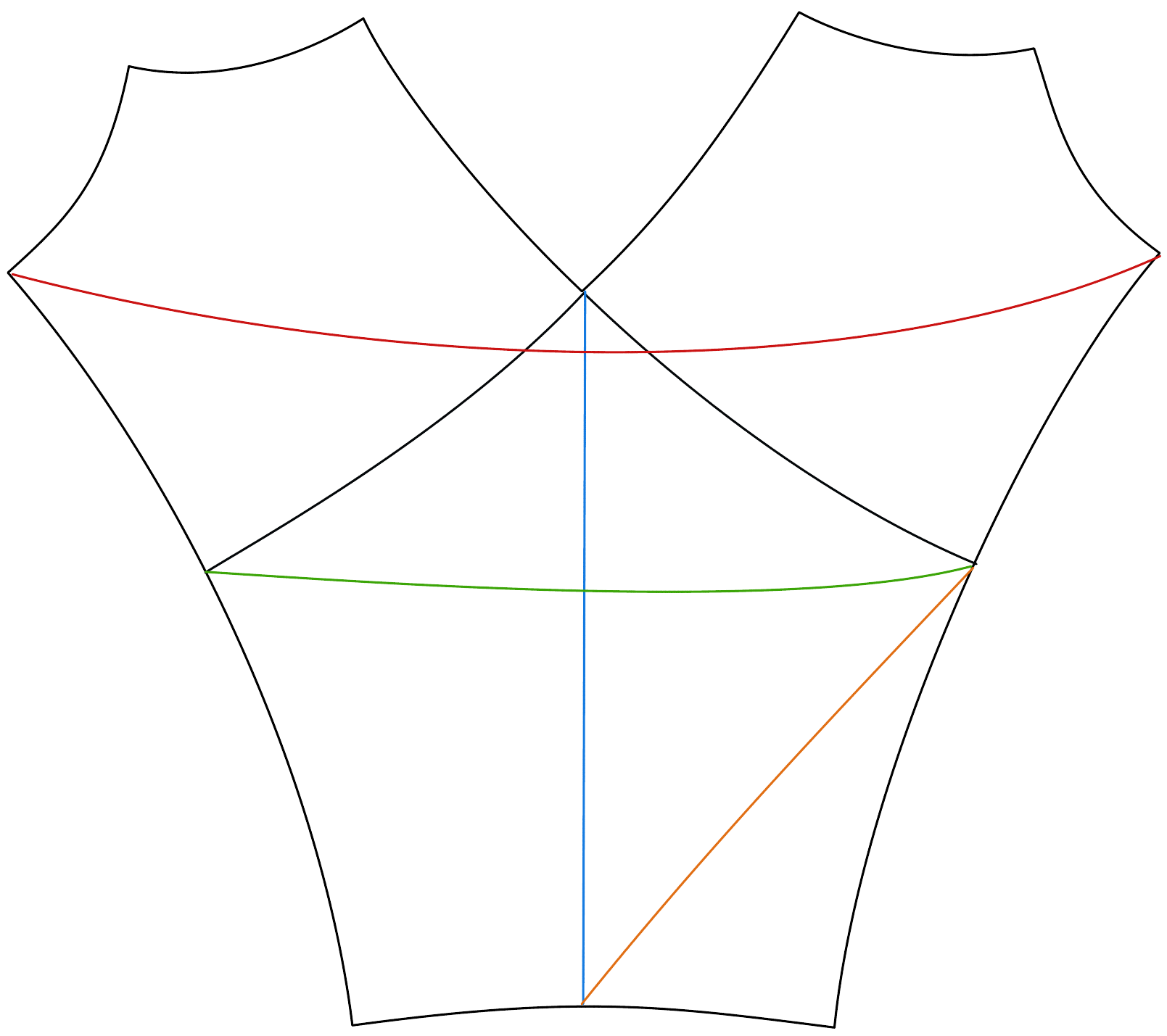}
\put(48,-2){$2a$}
\put(19,19){$2a$}
\put(77,19){$2a$}
\put(5,49){$2a$}
\put(90,49){$2a$}
\put(37,48){$2a$}
\put(60,48){$2a$}
\put(60,20){$e$}
\put(33,34){$f/2$}
\put(60,33.5){$f/2$}
\put(23,63){$g$}
\put(77,63){$g$}
\end{overpic}
\caption{Lengths in a right-angled pentagon, as computed in \Cref{Lem:PentagonLengths}. }
\label{Fig:PentagonLengths}
\end{figure}

Furthermore, the length $g$ of a geodesic segment contained in two adjacent right-angled pentagons, as in \Cref{Fig:PentagonLengths}, satisfies
\begin{align}
2g = \arccosh(1+2K (8 K^2+8K+1)^2) \approx 3.1838. \label{Eqn:GLength}
% alternate formula: g = \arcsinh( \sqrt{K} (8 K^2+8K+1) ).
\end{align}
\end{lem}

\begin{proof}
In \Cref{Fig:PentagonLengths}, left, the right-angled pentagon $P$ is subdivided into five isometric quadrilaterals, arranged symmetrically about the center point of $P$. Each quadrilateral $Q$ has two sides of length $a$, two sides of length $b$, three right angles, and one angle of $\theta = \frac{2\pi}{5}$. The presence of three right angles makes $Q$ a \emph{Lambert quadrilateral} or  \emph{almost rectangular quadrilateral} in the terminology of Ratcliffe \cite[Sec 3.5]{Ratcliffe:Foundations}.
By \cite[Thm 3.5.9]{Ratcliffe:Foundations}, we have
\[
K = \cos \theta  = \sinh^2 a,
\]
hence $K+1 = \cosh^2 a$, implying \Cref{Eqn:ALength}.
By \cite[Thm 3.5.8]{Ratcliffe:Foundations}, we have
\[
\cosh^2 b = \left( \frac{\cos \theta \cdot \cos (\tfrac{\pi}{2}) + \cosh a}{\sin \theta \cdot \sin (\tfrac{\pi}{2})} \right)^2
 = \frac{\cosh^2 a}{\sin^2 \theta} = \frac{1+K}{1-K^2} = \frac{1}{1-K},
\]
implying \Cref{Eqn:BLength}.

The hyperbolic law of cosines \cite[Thm 3.5.3]{Ratcliffe:Foundations} implies the following version of the Pythagorean theorem as a special case. In a hyperbolic right triangle, with legs of length $x,y$ and hypotenuse of length $z$, the lengths satisfy
\begin{equation}\label{Eqn:HyperbolicPythagorean}
\cosh x \cdot \cosh y = \cosh z.
\end{equation}

Now, consider the distances between the midpoints of edges in the right-angled pentagon $P$. 
By \eqref{Eqn:HyperbolicPythagorean}, the midpoints of two adjacent edges are separated by distance $c$, where
\[
\cosh c = \cosh^2 a = \sinh^2 a + 1 = K + 1,
\]
implying \Cref{Eqn:CLength}.
By \cite[Thm 3.5.3]{Ratcliffe:Foundations}, the midpoints of two non-adjacent edges are separated by distance $d$, where
\[
\cosh d = \cosh^2 b - \sinh^2 b \cdot \cos (2\theta).
\]
Substituting $\cosh^2 b = \frac{1}{1-K}$ and $\sinh^2 b = \frac{K}{1-K}$, as well as $\cos(2 \theta) = 2K^2 -1$ gives
\[
\cosh d = \frac{1}{1-K} - \frac{K (2K^2 -1)}{1-K}  = 2 K^2 + 2K + 1,
\]
implying \Cref{Eqn:DLength}.

To compute the lengths $e$     and $f$, 
 observe that 
\[
\cosh a = \sqrt{\sinh^2 a +1} = \sqrt{K+1}
\qquad \text{and} \qquad
\cosh(2a) = 2 \sinh^2 a + 1 = 2K +1.
\]
Now, \eqref{Eqn:HyperbolicPythagorean} gives
\[
\cosh(e) = \cosh(2a) \cdot \cosh a =  (2K+1) \sqrt{K+1} ,
\]
implying \Cref{Eqn:ELength}. 
Similarly, \eqref{Eqn:HyperbolicPythagorean} gives
\begin{equation*}%\label{Eqn:Cosh2A}
\cosh(f) = \cosh(2a)^2 =  (2K+1)^2 = 4K^2 + 4K + 1,
\end{equation*}
implying \Cref{Eqn:FLength}.

Finally, we use \Cref{Fig:PentagonLengths} to compute the length $g$. In that figure, we have a hyperbolic quadrilateral with two right angles at the ends of a side of length $2a$, adjacent sides of length $4a$, and the fourth side of length $2g$. According to first formula in the last block of displayed equations on \cite[page 88]{Fenchel:HyperbolicGeometry}, we have
\begin{equation}\label{Eqn:Fenchel}
\cosh(2g) = - \sinh^2(4a) + \cosh^2(4a) \cosh(2a).
\end{equation}
Above, we have already computed that $\cosh(2a)=2K+1$, hence 
\[
\cosh(4a) = 2 \cosh^2(2a) - 1 = 8 K^2 + 8K + 1.
\]
Substituting all this into \Cref{Eqn:Fenchel} gives
\begin{align*}
\cosh(2g) & = - \sinh^2(4a) + \cosh^2(4a) \cosh(2a) \\
& = 1 - \cosh^2(4a) + \cosh(2a) \cosh^2(4a) \\
& = 1+ (\cosh(2a) -1) \cosh^2(4a) \\
& = 1 + 2K (8 K^2 + 8K + 1)^2,
\end{align*}
implying \Cref{Eqn:GLength}.
\end{proof}

For the rest of this section, the symbols $a, \ldots, g$ will always denote the constants computed in the above lemma.
Now, we can prove the following more global comparison between the cubical and hyperbolic metrics on $\HH^2$.

\begin{prop}\label{Prop:PentagonalRestate}
Let $T$ be the tiling of $\HH^2$ by regular right-angled pentagons, with hyperbolic metric $\dist_\HH$. The dual tiling $T^*$, with five quadrilaterals at every vertex, can be identified with a CAT(0) square complex $\widetilde X$, with combinatorial metric $\dist_{\widetilde X}$. Then there is a constant $\epsilon > 0$ such that for all $x,y \in \widetilde X^{0}$, the distances $\dist_\HH(x,y)$ and $\dist_{\widetilde X}(x,y)$ can be compared as follows:
\begin{align}
c \cdot \dist_{\widetilde X}(x,y) - \epsilon \: \leq \: \dist_\HH(x,y) \: \leq \: d \cdot \dist_{\widetilde X}(x,y) + \epsilon, \label{Eqn:HypLessCube}
\end{align}
where the constants $c,d$ are as in \Cref{Lem:PentagonLengths}. Furthermore, the multiplicative constants $c,d$ in \eqref{Eqn:HypLessCube} are sharp.
\end{prop}

Observe that the constant $\lambda$ in the statement of \Cref{Prop:PentagonalQI} is exactly $d/c$. Thus \Cref{Prop:PentagonalRestate} implies \Cref{Prop:PentagonalQI}.

\begin{figure}
 \begin{overpic}[height=3.5in]{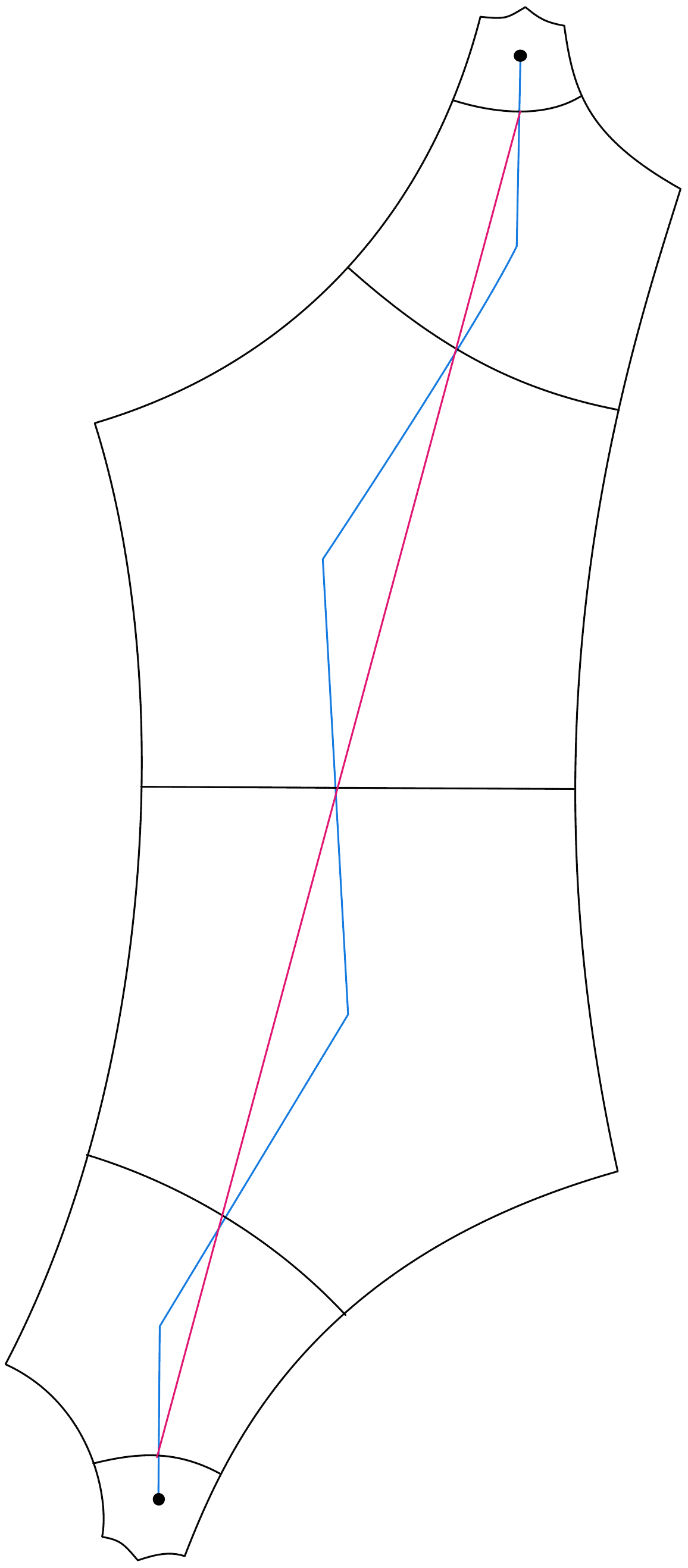}
 \put(7,3){$x$}
 \put(34,95){$y$}
 \put(16,60){$w$}
 \put(26,60){$\gamma$}
 \end{overpic}
\hspace{0.25in}
\begin{overpic}[height=3.5in]{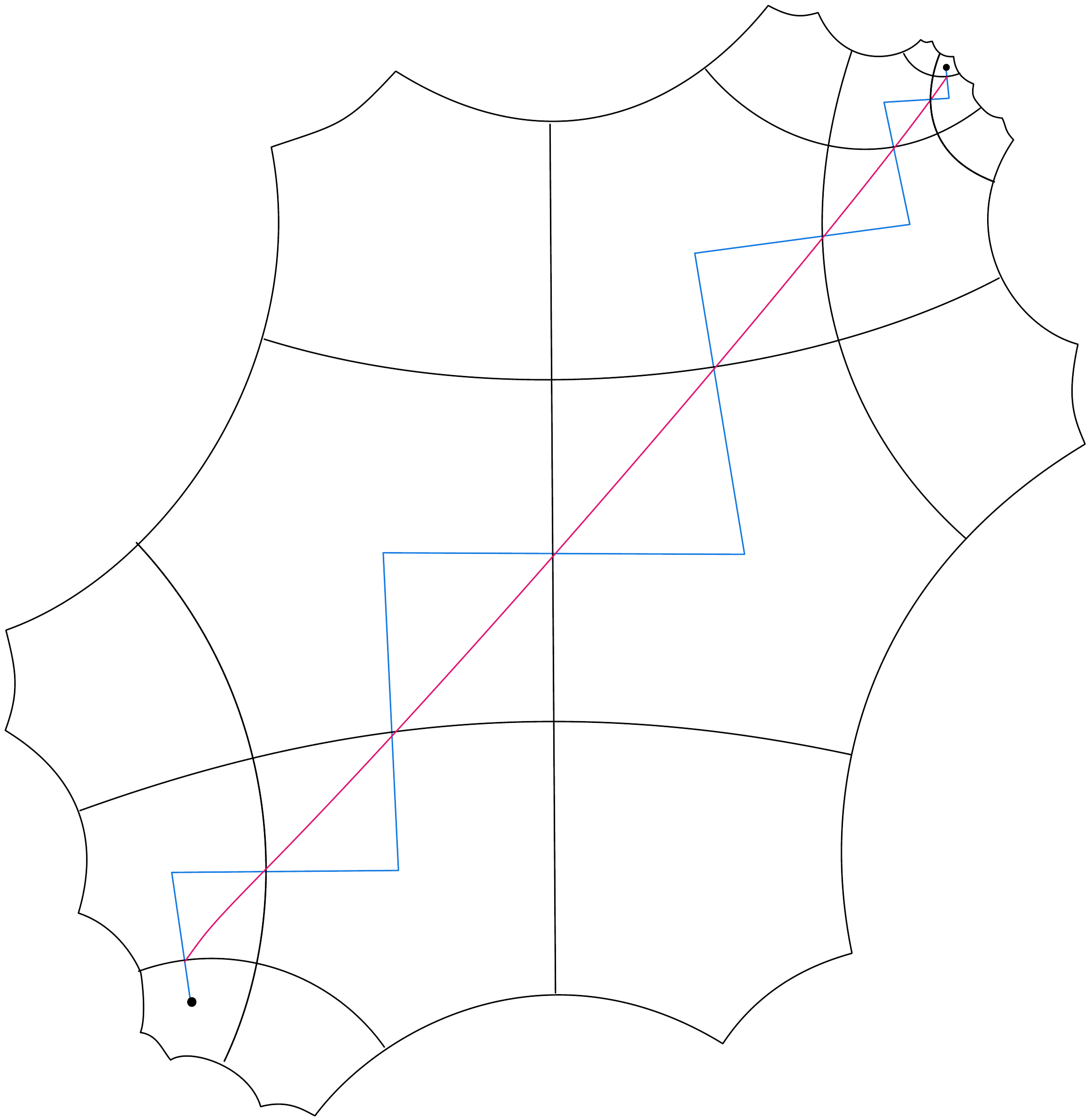}
\put(15,7){$x$}
\put(86,95){$y$}
\put(42,52){$w$}
\put(43,40){$\gamma$}
\end{overpic}
\caption{The blue path $w \to X$ is a combinatorial geodesic from $x$ to $y$. The pink path $\gamma$ is constructed by taking hyperbolic shortcuts between midpoints of consecutive edges of $w$. Each segment of $\gamma$ in a pentagon of $T$ has length $c$ or $d$. The sharpness of the constants $c$  and $d$ in \eqref{Eqn:HypLessCube} is demonstrated by the right and left panels, respectively.}
\label{Fig:Shortcut}
\end{figure}
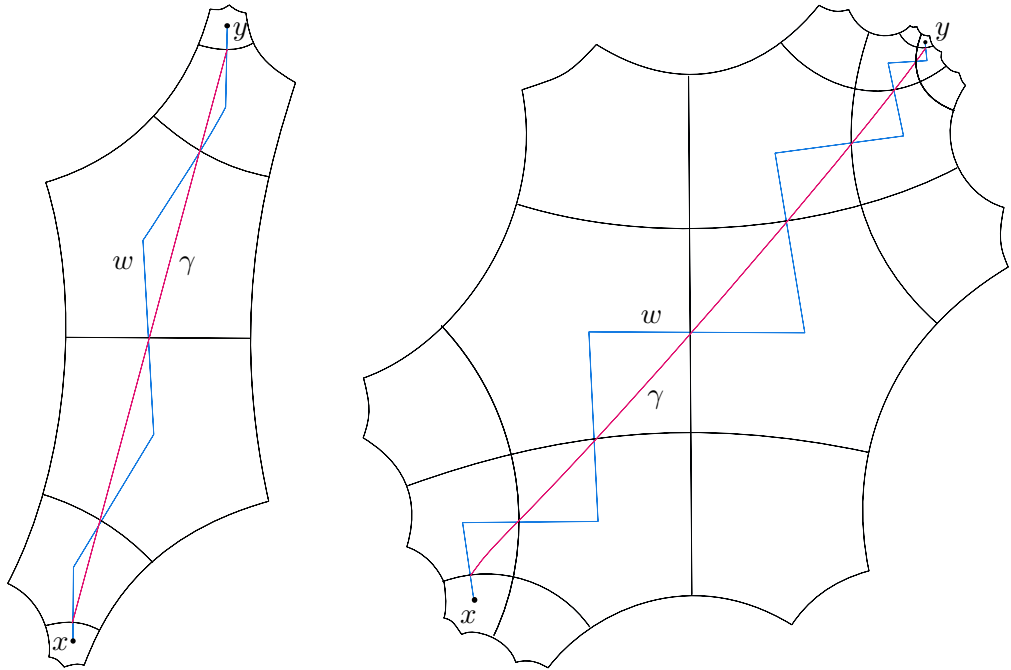

\begin{proof}
We begin by proving the second inequality of \eqref{Eqn:HypLessCube}.
Let $w \to \widetilde X$ be a combinatorial geodesic in $\widetilde X$ with endpoints $x, y \in \widetilde X^{0}$. By choosing a sufficiently large additive constant $\epsilon$, we may assume without loss of generality that $|w| = \dist_{\widetilde X}(x,y) \geq 2$. In the hyperbolic metric on $\HH^2$, the combinatorial geodesic $w$ is a concatenation of two or more edges of the dual tiling $T^*$. By \Cref{Lem:PentagonLengths}, every edge of $w$ has hyperbolic length $2b$.

Since edges of $T^*$ meet at angles of $\theta = \frac{2\pi}{5}$ or $2\theta = \frac{4\pi}{5}$, we may homotope $w$ to a shorter piecewise-geodesic path $\gamma$
by constructing hyperbolic shortcuts between midpoints of consecutive edges. See \Cref{Fig:Shortcut}. Each such shortcut replaces two cubical half-edges (of combined cubical length $1$) by a hyperbolic segment of length either $c$ or $d$.
The first and last half-edges of $w$ remain as they are, and have hyperbolic length $b$. Since $c < d$, it follows that
\[
\dist_{\HH^2}(x, y) \: \leq \:  b + d (|w|-1) + b \: \leq \:   d |w| + \epsilon \: = \:  d \cdot \dist_{\widetilde X}(x,y) + \epsilon,
\]
for an appropriate value of $\epsilon$. 
Sharpness of the multiplicative constant $d$ holds because one may concatenate arbitrarily many segments of length $d$ in adjacent pentagons to form a hyperbolic geodesic. See the left panel of \Cref{Fig:Shortcut}.

By the same token, sharpness of the multiplicative constant $c$ in the first inequality of \eqref{Eqn:HypLessCube} holds because 
one may concatenate arbitrarily many segments of length $c$ in adjacent pentagons to form a hyperbolic geodesic, as in the right panel of \Cref{Fig:Shortcut}.

% \smallskip

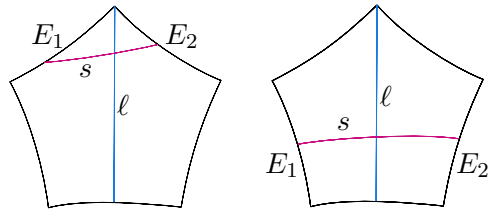
\begin{figure}[ht]
\begin{overpic}[width=2.5in]{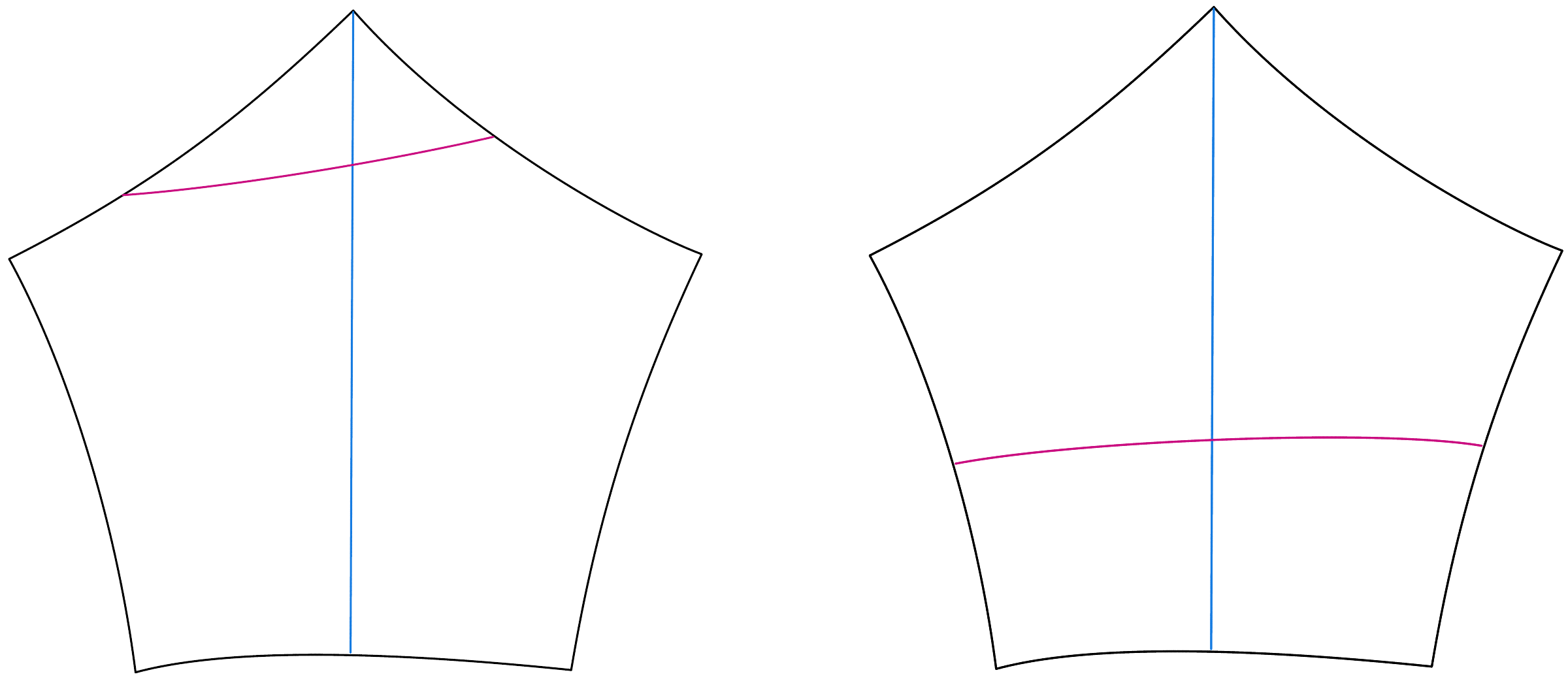}
\put(15,28){$s$}
\put(69,17){$s$}
\put(23,20){$\ell$}
\put(78,22){$\ell$}
\put(5,35){$E_1$}
\put(33,35){$E_2$}
\put(54,8){$E_1$}
\put(94,8){$E_2$}
\end{overpic}
\caption{The two combinatorial types of segment $s$ cutting through a right-angled pentagon, and the altitude $\ell$
altitude associated to each $s$.}
\label{Fig:Altitudes}
\end{figure}

To prove the first inequality of \eqref{Eqn:HypLessCube}, with its optimal multiplicative constant, we make the following definitions. An \emph{altitude} of a right-angled pentagon $P$ is a geodesic segment $\ell$ from a vertex to the midpoint of the opposite edge. If $s \subset P$ is a hyperbolic geodesic connecting interior points of sides $E_1, E_2$, the \emph{altitude associated to $s$} is the unique altitude $\ell$ with the property that reflection in $\ell$ interchanges $E_1$ with $E_2$. See \Cref{Fig:Altitudes}.

Let $\gamma \to \HH^2$ be a hyperbolic geodesic with endpoints $x, y \in \widetilde X^{0}$. Then $ \dist_{\widetilde X}(x,y)$ equals the number of hyperplanes that separate $x$ from $y$. Since the hyperplanes of $X$ are identified with the bi-infinite geodesics containing edges of $T$, it follows that $ \dist_{\widetilde X}(x,y)$ is the number $n$ of edges of $T$ crossed by  $\gamma$. After an arbitrarily small perturbation, affecting $\len_{\HH}(\gamma)$ by an additive error, we may assume that $\gamma$ is disjoint from $T^{0}$. 

In every pentagon $P$ that intersects $\gamma$ but does not contain the endpoints $x,y$, draw the altitude associated to $\gamma \cap P$. These $(n-1)$ altitudes partition $\gamma$ into a concatenation $\gamma_1 \gamma_2 \cdots \gamma_n$, where $\gamma_1$ is the segment from $x$ to the first altitude; $\gamma_n$ is the segment from the last altitude to $y$; and every remaining $\gamma_i$ connects the altitudes in adjacent right-angled pentagons. 
To prove  \eqref{Eqn:HypLessCube}, we will show that the \emph{average} length of a segment $\gamma_i$ for $0<i<n$ is at least $c \approx 0.7672$. 

\begin{figure}[h]
\begin{overpic}[width=2.5in]{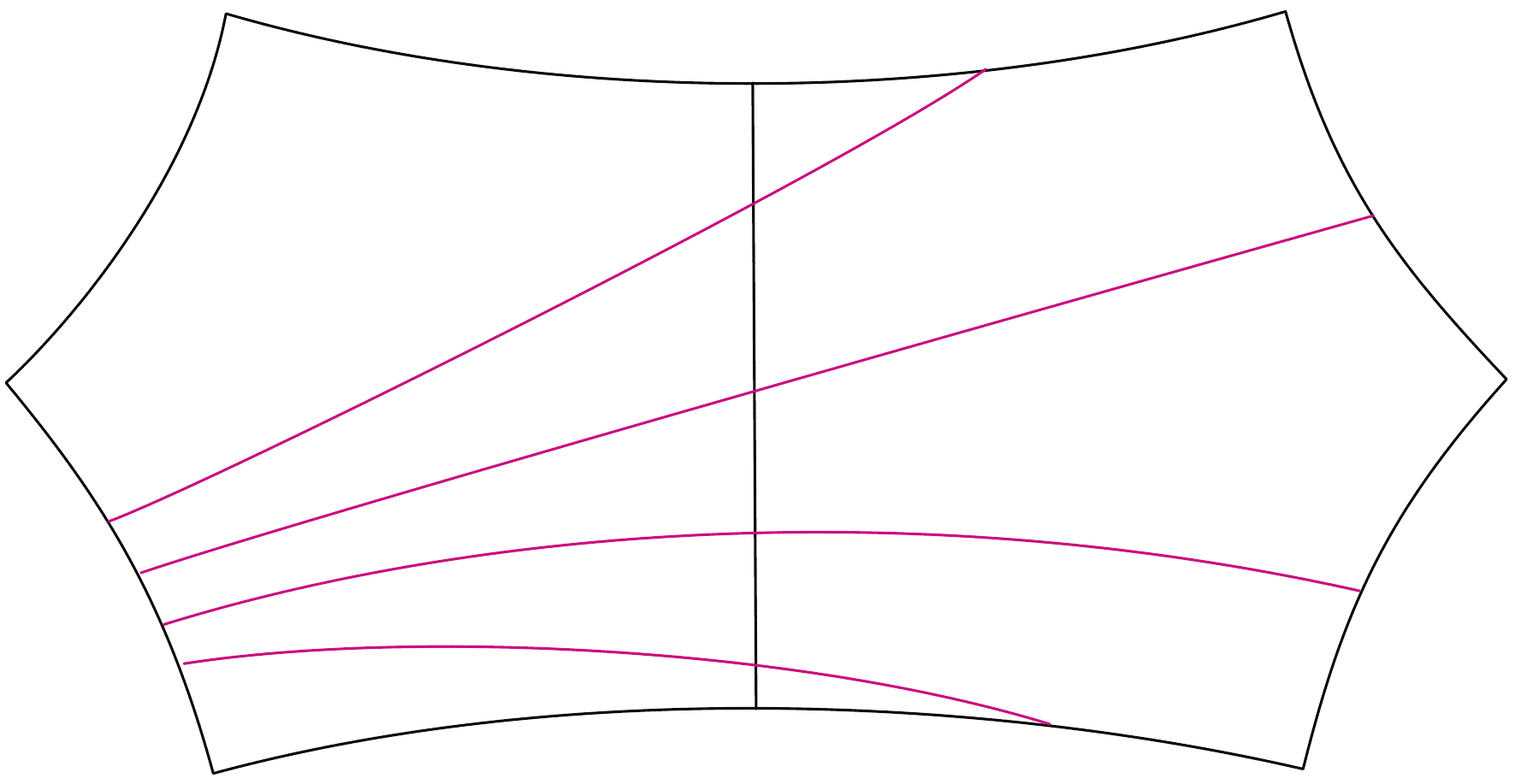}
\put(20,33){$P$}
\put(79,22){$P'$}
\put(68,-2){I}
\put(92,10){II}
\put(93,37){III}
\put(62,49){IV}
\end{overpic}
\hspace{0.25in}
\begin{overpic}[width=2.5in]{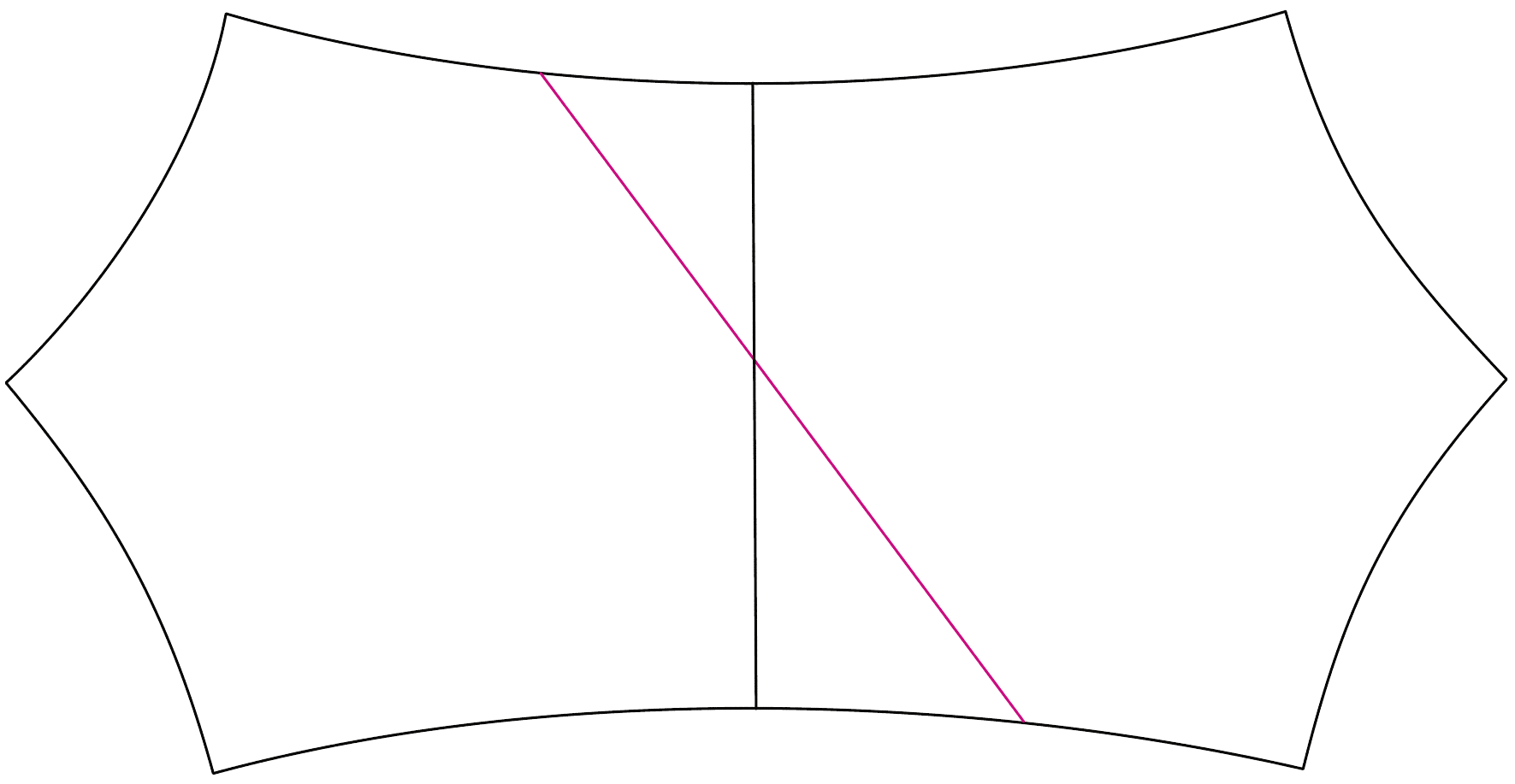}
\put(20,24){$P$}
\put(75,24){$P'$}
\put(65,-2){V}
\end{overpic}
\caption{The five combinatorial types of (unoriented) intersection between a hyperbolic geodesic and two adjacent pentagons $P,P'$. In types I--IV, the geodesic $\gamma$ enters through a side of $P$ that is \emph{not} adjacent to the shared side $P \cap P'$.
%    , and exits through one of four sides of $P'$. 
In type V, the geodesic $\gamma$ both enters and exits $P \cup P'$ through sides adjacent to the shared side $P \cap P'$.}
\label{Fig:Types}
\end{figure}
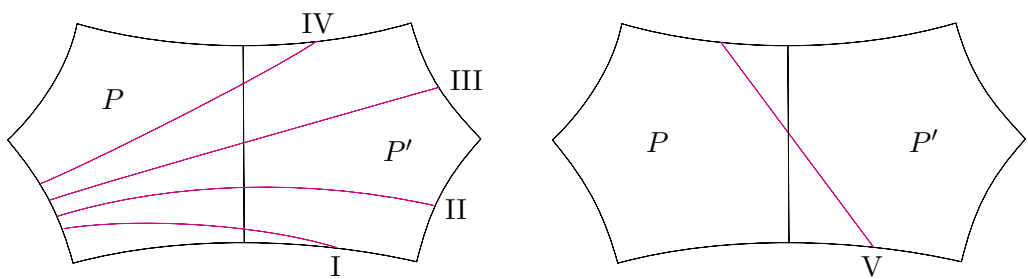

There are five combinatorial possibilities for a segment $\gamma_i$ between two altitudes, corresponding to five types of intersections between $\gamma$ and two adjacent pentagons. See \Cref{Fig:Types}. If $\ell$ and $\ell'$ are the altitudes in adjacent pentagons $P$ and $P'$, respectively, the convexity of distance functions implies that the shortest geodesic segment from $\ell$ to $\ell'$ meets each of them orthogonally or at an endpoint. Consequently, it is not hard to determine the shortest possible length from $\ell$ to $\ell'$; see 
 \Cref{Fig:TypesWithLengths}.

\begin{figure}[h]
\begin{overpic}[width=1in]{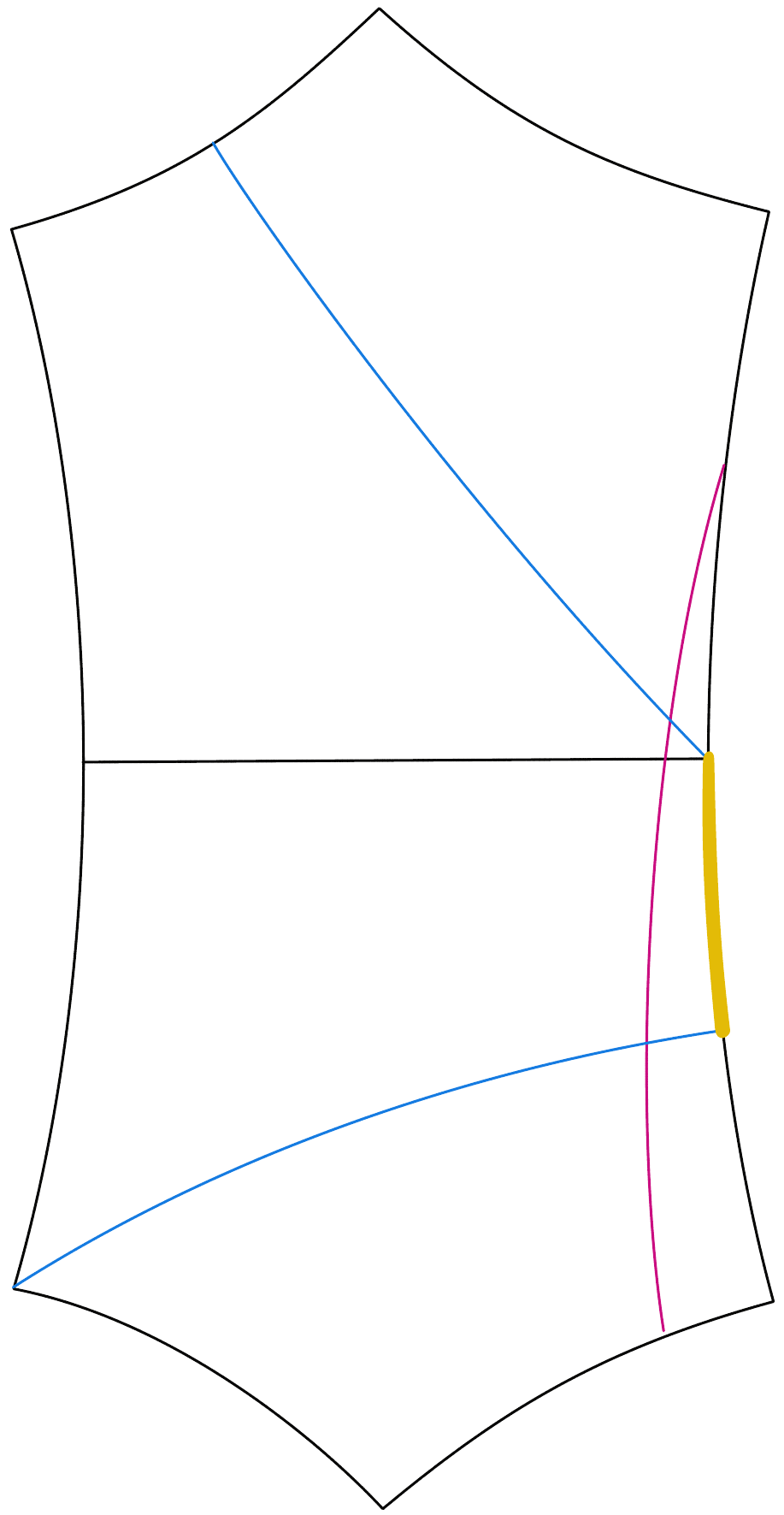}
\put(35,95){I}
\put(48,40){$a$}
\end{overpic}
   \hspace{0.17in}
\begin{overpic}[width=1in]{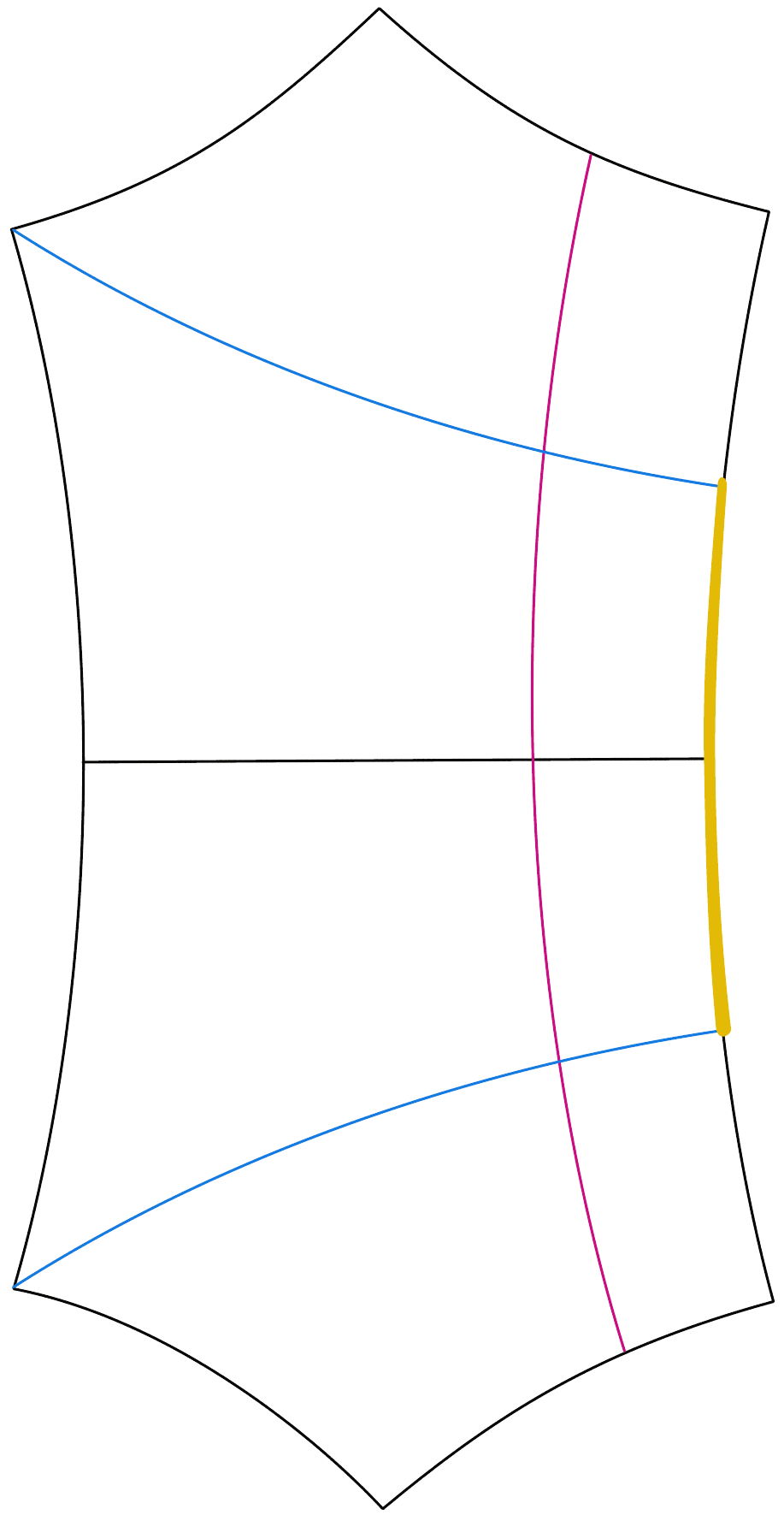}
\put(35,95){II}
\put(48,40){$a$}
\put(48,58){$a$}
\end{overpic}
 \hspace{0.17in}
\begin{overpic}[width=1in]{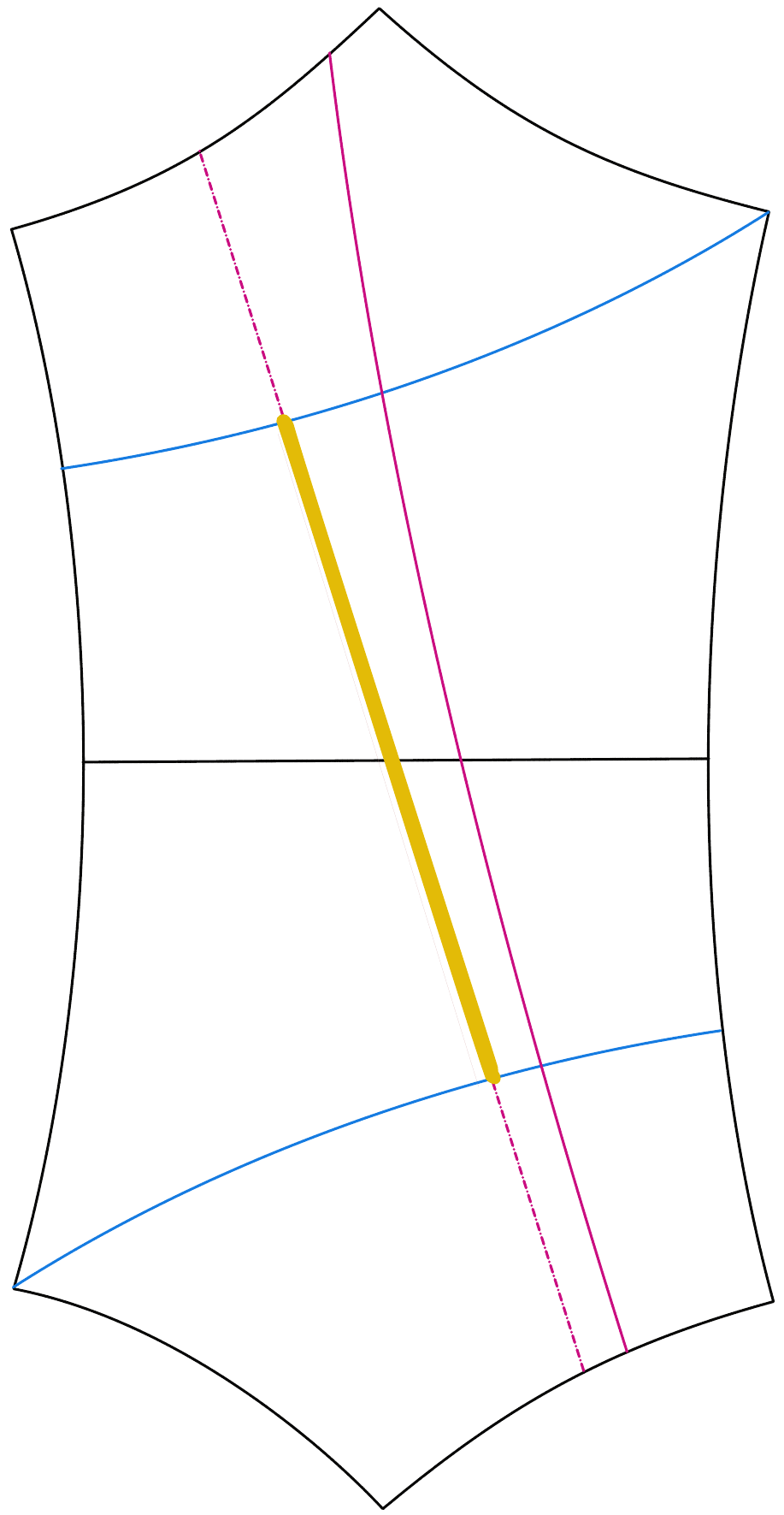}
\put(35,95){III}
\put(15,35){$d/2$}
\put(8,58){$d/2$}
\end{overpic}
\hspace{0.17in}
\begin{overpic}[width=1in]{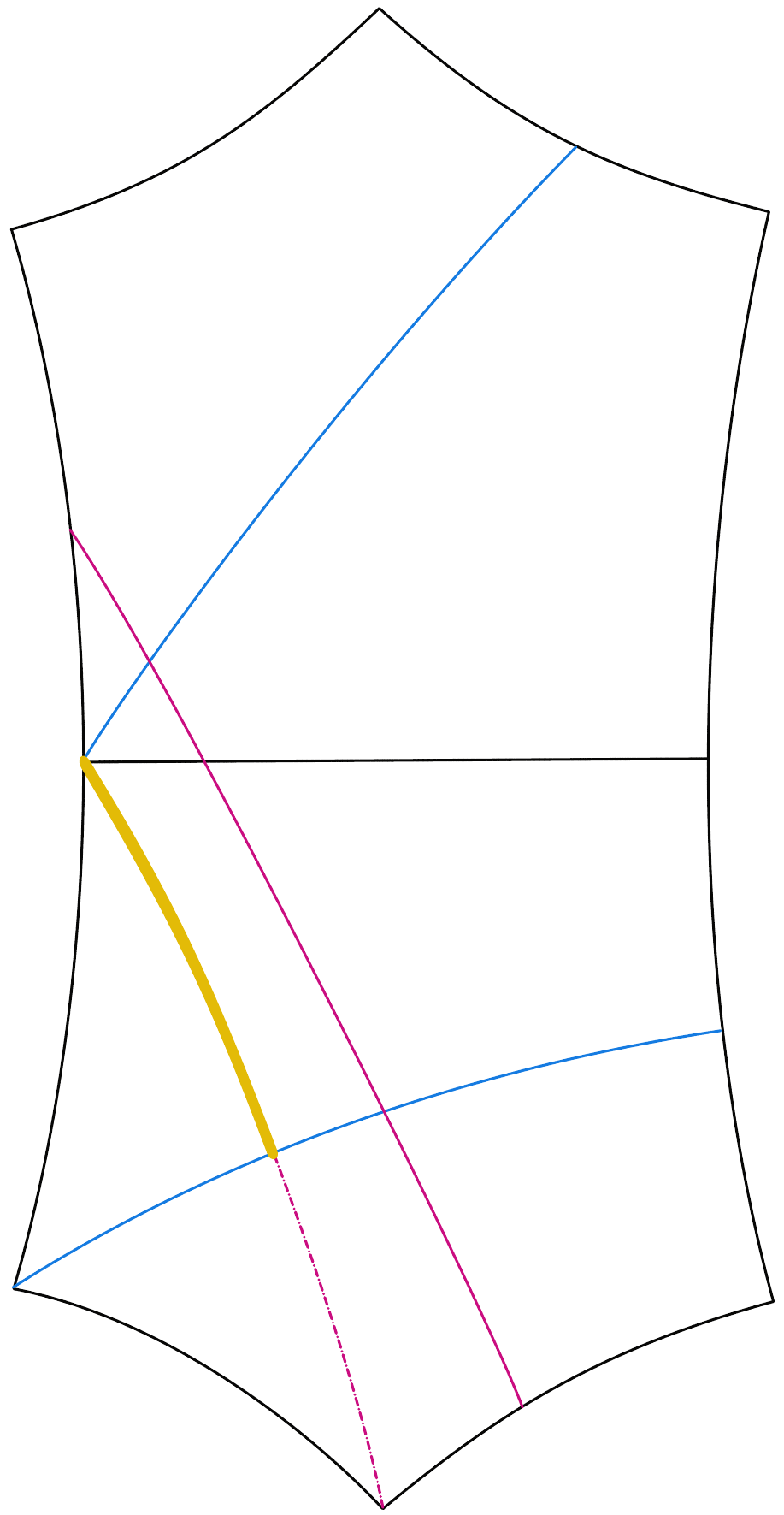}
\put(35,95){IV}
\put(15,35){$f/2$}
\end{overpic}
\hspace{0.17in}
\begin{overpic}[width=1in]{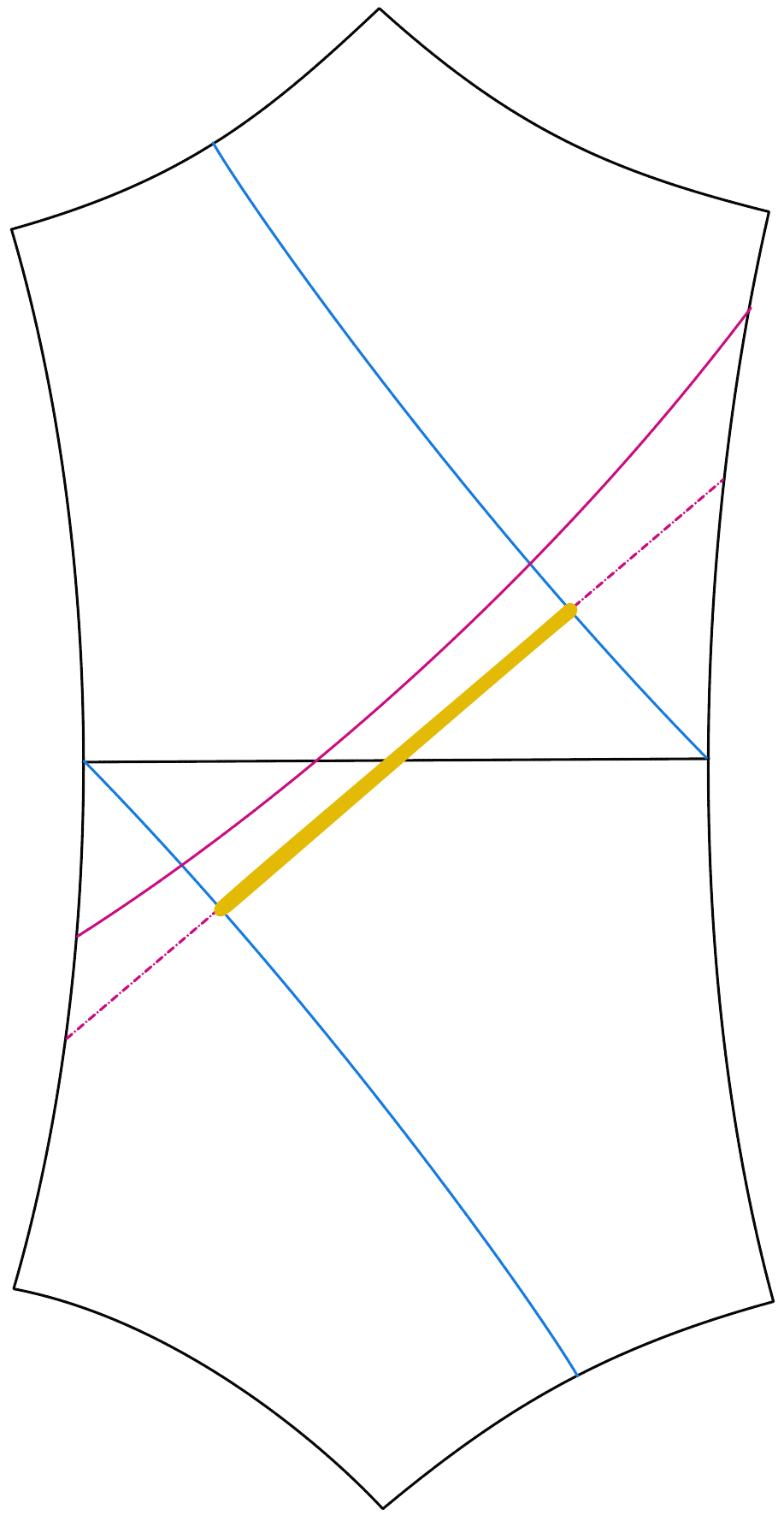}
\put(35,95){V}
\put(21,40){$c/2$}
\put(17,54){$c/2$}
\end{overpic}
\caption{For each combinatorial type of intersection between a hyperbolic geodesic and two adjacent pentagons, the highlighted segment shows the infimal length between the altitudes.}
\label{Fig:TypesWithLengths}
\end{figure}
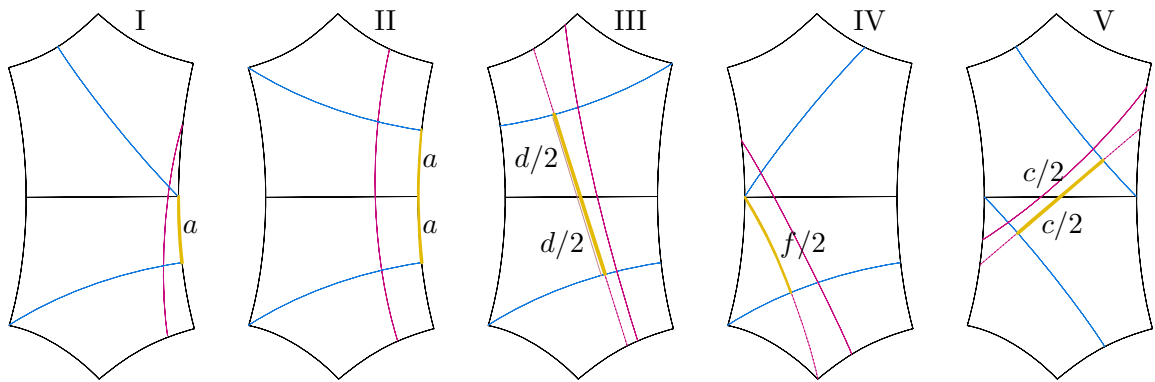

In type I, \Cref{Fig:TypesWithLengths} shows that the worst-case scenario for $\len(\gamma_i)$ is $a \approx 0.5306$; since $a < c$, this complicated case is analyzed below. In type II, \Cref{Fig:TypesWithLengths} shows that we always have $\len(\gamma_i) \geq 2a > c$. In type III, we always have $\len(\gamma_i) \geq d > c$. In type IV, we always have $\len(\gamma_i) \geq f/2 > c$. Finally, in type V, we always have $\len(\gamma_i) \geq c$. The comparisons to $c$ come from the numerical values computed in  \Cref{Lem:PentagonLengths}. Thus, in every case except type I, we have $\len(\gamma_i) \geq c$.

It remains to analyze segments $\gamma_i$ of type I. We claim the following:
\begin{enumerate}[\:\: $(1)$]
\item\label{Claim:AdjacentType} If $\gamma_i$ is of type I, and $2 < i < n-1$, then there is an adjacent index $j = i \pm 1$ such that the corresponding
segment $\gamma_j$ is of type II, III, or IV. 
\item\label{Claim:AdjacentII} If $\gamma_j$ is of type II or IV, then $\gamma_i$ is the \emph{only} type--I segment adjacent to $\gamma_j$. 
\item\label{Claim:TotalLength} For all types of $\gamma_j$, we have $\len(\gamma_i) + \len(\gamma_j) \geq \min(3a, g) > 2c$.
\item\label{Claim:AdjacentIII} If $\gamma_j$ is of type III, and furthermore $\gamma_j$ also adjacent to a type--I segment $\gamma_{k}$ where $k = i \pm 2$, then $\len(\gamma_i) + \len(\gamma_j) + \len(\gamma_k) \geq 2e > 3c$. 
\end{enumerate}
Assuming these claims, we can complete the proof as follows. By Claim~\ref{Claim:AdjacentType}, every segment $\gamma_i$ of type I (where $3 \leq i \leq n-2$) borrows some of the length from an adjacent segment $\gamma_j$ of type II, III, or IV. By Claim~\ref{Claim:AdjacentII}, every segment $\gamma_j$ of type II or IV acts as a ``lender'' to at most one segment of type I.  Every segment $\gamma_j$ of type III acts as a ``lender'' to at most two segments of type I. In all cases, Claims~\ref{Claim:TotalLength} and~\ref{Claim:AdjacentIII} say that the average length of $\gamma_j$ and its adjacent type--I segments is more than $c$. 
Putting it all together, we have
\[
\dist_{\HH^2}(x,y) = \len(\gamma) = \sum_{i=1}^n \len(\gamma_i) > \sum_{i=3}^{n-2} c = (n-4) c = c \cdot \dist_{\widetilde X}(x,y) - 4 c,
\]
which completes the proof of \Cref{Eqn:HypLessCube}.

\begin{figure}[h]
\begin{overpic}[width=4in]{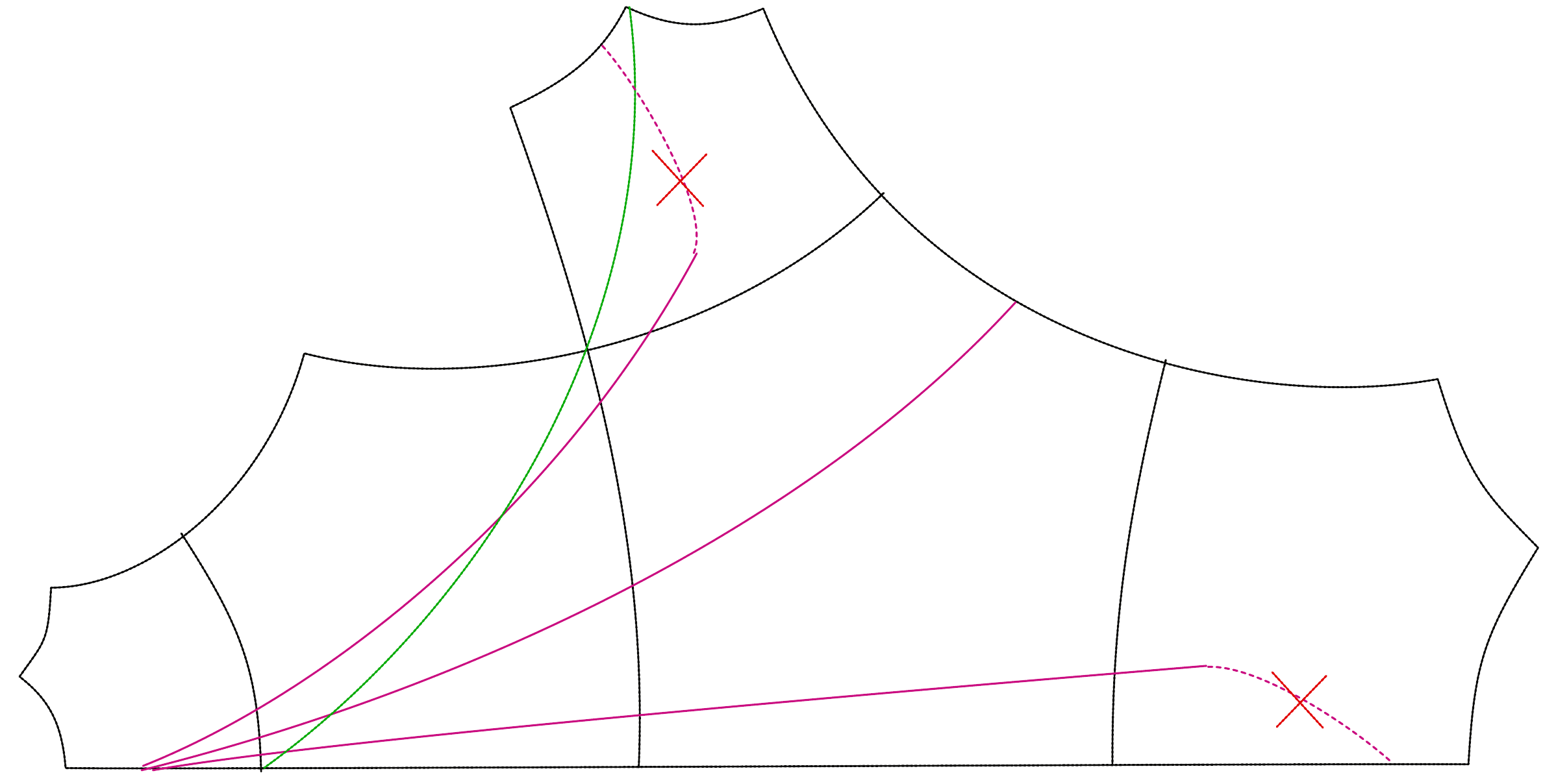}
\put(8,6){$P$}
\put(24,20){$Q$}
\put(58,15){$R$}
\put(73,9){II}
\put(64,33){III}
\put(46,34){IV}
\put(60,-2){$\delta$}
\put(38,38){$\xi$}
\put(32,5){$\gamma$}
\put(32,11){$\gamma$}
\put(25,14){$\gamma$}
\end{overpic}
\caption{Three possibilities for  a geodesic $\gamma$ that intersects $\{P, Q\}$ in a segment of type I. The next segment of $\gamma$ in pentagons $Q \cup R$ must be of type II, III, or IV.}
\label{Fig:TypeINeighbors}
\end{figure}
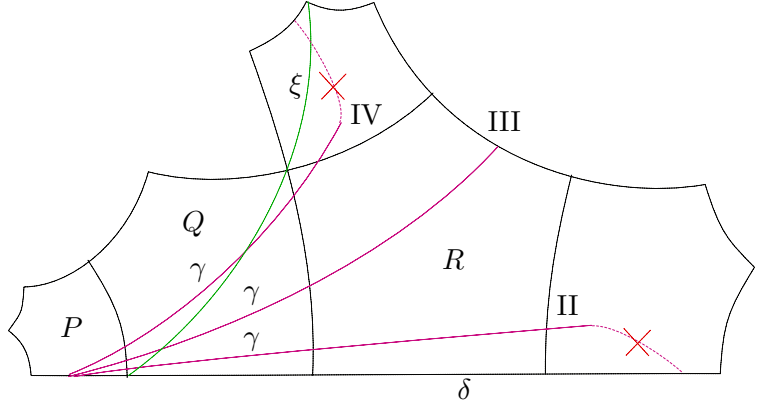

\Cref{Fig:TypeINeighbors} illustrates the proof of Claims~\ref{Claim:AdjacentType} and~\ref{Claim:AdjacentII}. The geodesic $\gamma$ intersects the left-most pair of pentagons $\{P, Q\}$ in a segment $\gamma_i \subset \gamma$ of type I. We assume without loss of generality that the indices are increasing as $\gamma$ traverses the figure from left to right. The continuation of $\gamma$ must exit the central pentagon  $R$ through one of the sides marked II, III, or IV, because  $\gamma$ cannot intersect  $\delta$ twice. This proves Claim~\ref{Claim:AdjacentType}. If the next segment $\gamma_j = \gamma_{i+1}$ is of type II,  then the following segment $\gamma_k = \gamma_{i+2}$ cannot be of type I, because otherwise $\gamma$ would again intersect $\delta$ twice. Similarly, if $\gamma_j = \gamma_{i+1}$ is of type IV, then the following segment $\gamma_k = \gamma_{i+2}$ cannot be of type I, because $\gamma$ cannot intersect the geodesic $\xi$ twice. This proves Claim~\ref{Claim:AdjacentII}.

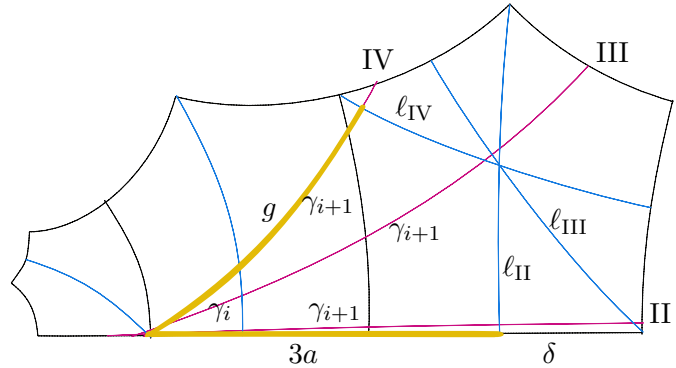
\begin{figure}[h]
\begin{overpic}[width=3.5in]{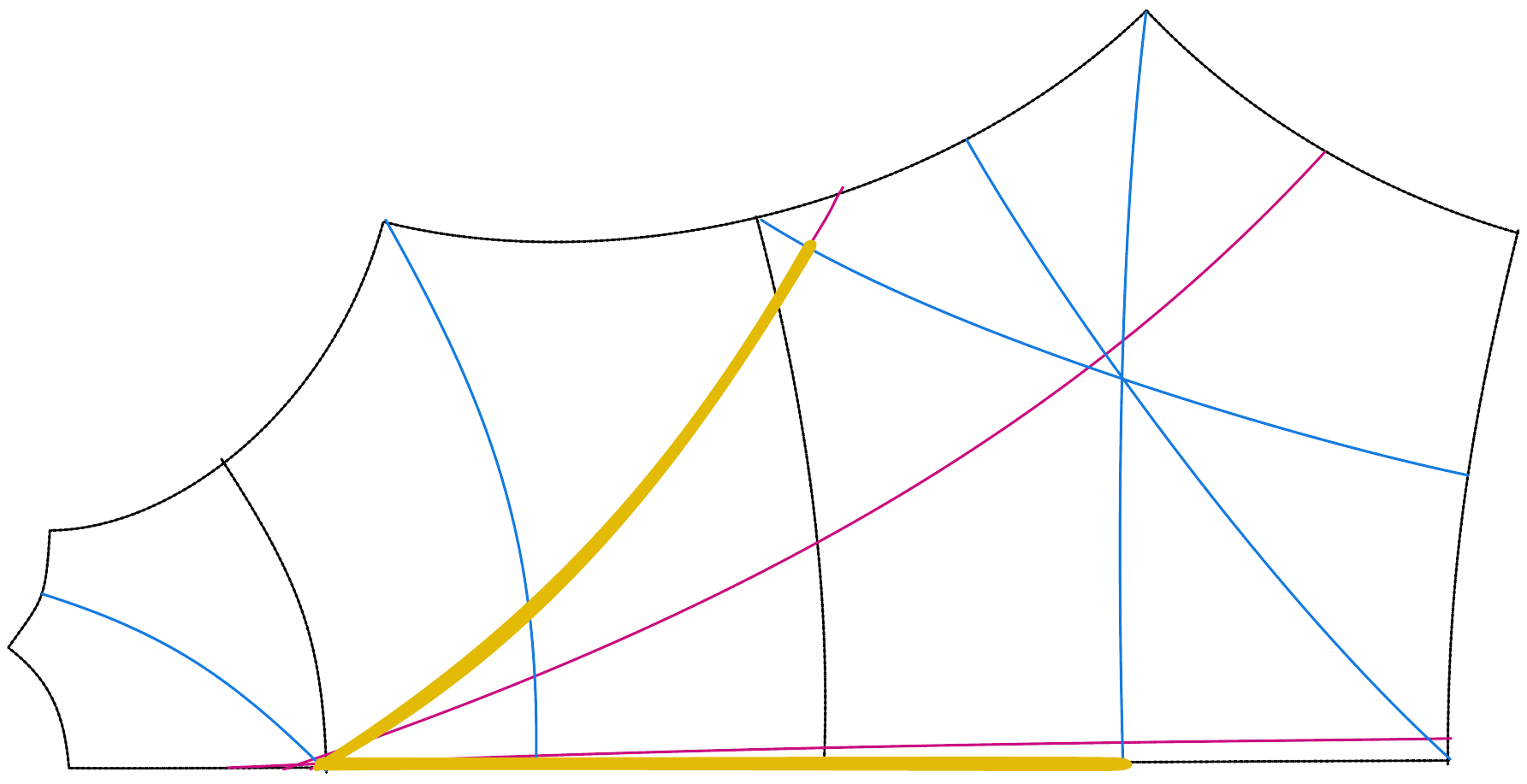}
\put(96,3){II}
\put(88,42){III}
\put(53,41){IV}
\put(30,4){$\gamma_i$}
\put(45,4){$\gamma_{i+1}$}
\put(57,16){$\gamma_{i+1}$}
\put(44,20){$\gamma_{i+1}$}
\put(74,10){$\ell_{\rm II}$}
\put(81,17){$\ell_{\rm III}$}
\put(58,34){$\ell_{\rm IV}$}
\put(42,-3){$3a$}
\put(38,19){$g$}
\put(80,-3){$\delta$}
\end{overpic}
\caption{The left-most segment $\gamma_i$ is of type I. The next segment $\gamma_j = \gamma_{i+1}$ is of type II, III, or IV, and terminates at the altitude  $\ell_{\rm II}$, $\ell_{\rm III}$, or $\ell_{\rm IV}$, respectively. In type II and type IV, the shortest possible lengths of $\gamma_i \cup \gamma_{i+1}$ are highlighted. In type III, $\gamma_{i+1}$ must intersect $\ell_{\rm II}$ or $\ell_{\rm IV}$.}
\label{Fig:LengthTwoSegments}
\end{figure}

For the three combinatorial types of $\gamma_j = \gamma_{i+1}$, \Cref{Fig:LengthTwoSegments} illustrates the infimal lengths of $(\gamma_i \cup \gamma_j)$. Let $\ell_{\rm II}, \ell_{\rm III}, \ell_{\rm IV}$ be the terminal altitudes for the three possible types of $\gamma_j$.
If $\gamma_j$ is of type II, then the worst-case scenario is when $\gamma$ fellow-travels $\delta$, hence we obtain $\len(\gamma_i \cup \gamma_j) >  3a > 2c$, where the final inequality uses \Cref{Lem:PentagonLengths}. If $\gamma_j$ is of type IV, then the worst-case scenario is when $\gamma_i \cup \gamma_j$ starts at the endpoint of an altitude and ends perpendicular to $\ell_{\rm IV}$. In this case, we have $\len(\gamma_i \cup \gamma_j) \geq g > 2c$, where the final inequality is by \Cref{Lem:PentagonLengths}. Finally, observe that if $\gamma_j$ is of type III, then $\gamma_j$ must intersect either $\ell_{\rm II}$ or $\ell_{\rm IV}$. Thus, by the cases already discussed, we have $\len(\gamma_i \cup \gamma_j) > 2c$. This proves Claim~\ref{Claim:TotalLength}.

\begin{figure}[h]
\begin{overpic}[width=3.5in]{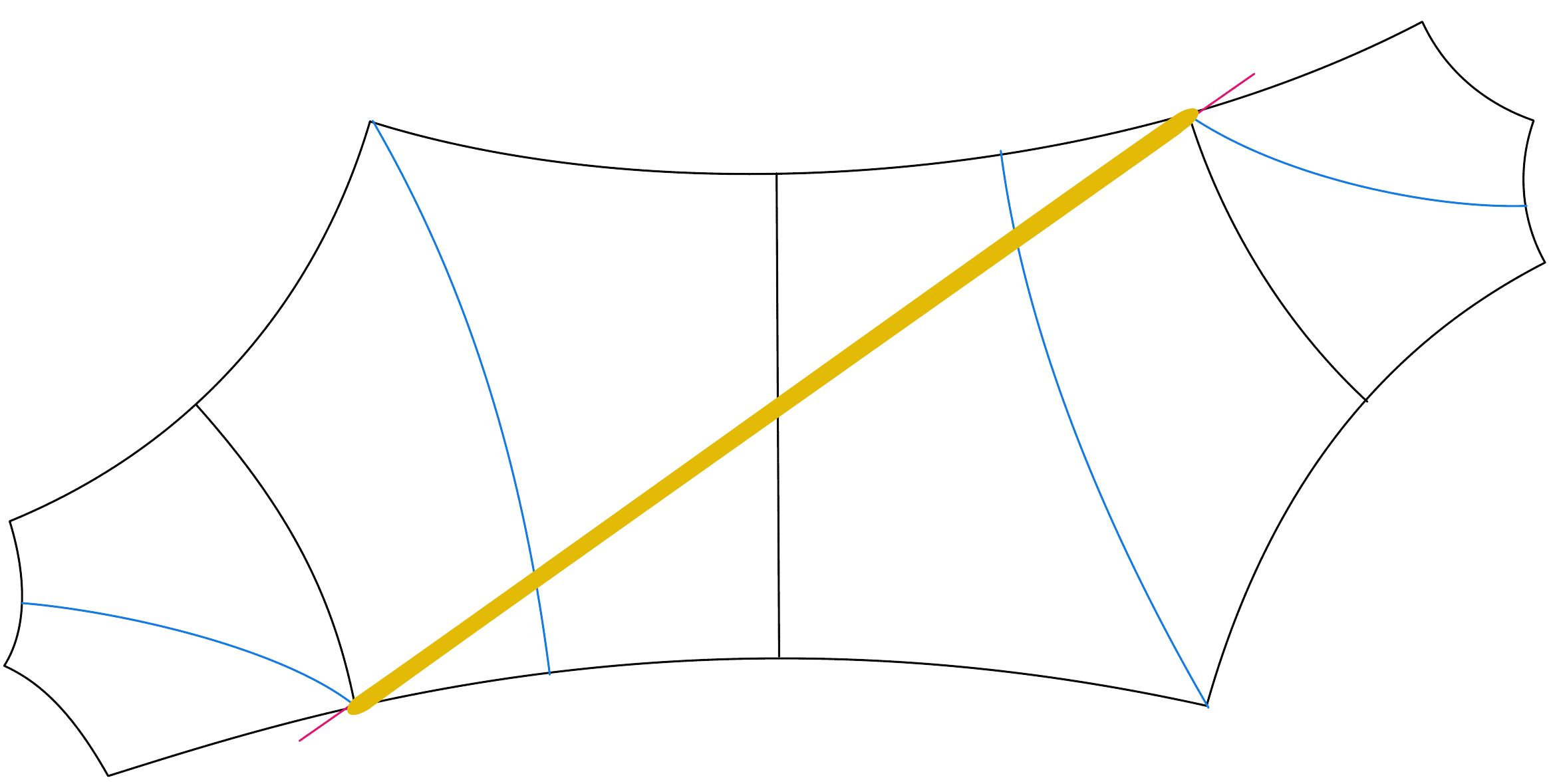}
\put(25,10){$\gamma_i$}
\put(40,24){$\gamma_{i+1}$}
\put(68,34){$\gamma_{i+2}$}
\end{overpic}
\caption{If $\gamma_j = \gamma_{i+1}$ is of type III and $\gamma_i, \gamma_{i+2}$ are both of type I, the shortest possible configuration is symmetric about the point in the center. In this configuration, each half of $(\gamma_i \cup \gamma_{i+1} \cup \gamma_{i+2})$ is a segment of length $e$.}
\label{Fig:LengthThreeSegments}
\end{figure}

Finally, \Cref{Fig:LengthThreeSegments} illustrates the proof of Claim~\ref{Claim:AdjacentIII}. If $\gamma_j = \gamma_{i+1}$ is of type III and $\gamma_i, \gamma_{i+2}$ are both of type I, then the shortest possible length of  $(\gamma_i \cup \gamma_j \cup \gamma_k)$ is $2e > 3c$, where the inequality follows from \Cref{Lem:PentagonLengths}. This completes the proof of the Claims, implying   \Cref{Eqn:HypLessCube} and the Proposition.
\end{proof}

\section{Problems}\label{Sec:Problems}

This section collects several problems and suggested directions for future research.

\begin{prob}\label{Prob:SurfaceOptimalDensity}
Determine the optimal density at which random quotients of a surface group $G$ are cubulated. How does the answer depend on the choice of proper metric on $G$?
\end{prob}

\begin{prob}\label{Prob:GeneralOptimalDensity}
Given a non-positively curved cube complex $X$ such that $G = \pi_1 X$ is hyperbolic, find the optimal density such that quotients of $G$ are cubulated. How does the answer depend on $X$ and the growth rate of its hyperplanes? 
\end{prob}

The following three problems are stated in order of increasing relatively hyperbolic ambition.

\begin{prob}\label{Prob:FreeProdOptimalDensity}
Let $G = G_1 * G_2$ be a free product of cubulated groups. Find the optimal density at which random quotients of $G$ are cubulated. Letting $G_i = \pi_1 X_i$, where $X_i$ is a non-positively curved cube complex, we think of $G = \pi_1 X$, where $X = X_1 \cup A \cup X_2$, where $A \cong[0,n]$ is an arc glued to $X_1$  and $X_2$ at its endpoints. Does the optimal density depend on the length of $A$? 
See Martin--Steenbock~\cite{MartinSteenbock} and Jankiewicz--Wise~\cite{JankiewiczWise}.
\end{prob}

\begin{prob}
Let $G$ be a relatively hyperbolic group that is the fundamental group of a compact non-positively curved cube complex. Show that there is a density at which random quotients of $G$ are cubulated (and relatively hyperbolic).
\end{prob}

\begin{prob}
Let $G$ be an acylindrically hyperbolic group that is the fundamental group of a compact non-positively curved cube complex. Show that there is a density at which random quotients of $G$ are cubulated.
\end{prob}

The following problem has not even been studied for virtually free groups:

\begin{prob}
Generalize \Cref{Thm:main} to cubulated hyperbolic groups with torsion.
\end{prob}

The main challenge is that cubical small-cancellation theory
was not described in the case where there is torsion. One could either develop that theory, or use another hyperbolic Dehn filling theory with walls (e.g.~Osin~\cite{Osin2007} or Groves--Manning~\cite{GrovesManning2008} or Dahmani--Guirardel--Osin~\cite{DGO}) to ensure that the relators embed in the universal cover of the quotient, and subsequently apply the criterion of \Cref{Thm:C'20Proper} which
can instead be applied to $\widetilde{X^*}$.)

\begin{prob}
Let $G$ be a non-elementary relatively hyperbolic group. Prove that there is a density such that all low-density quotients of $G$ are again relatively hyperbolic. Note that the (torsion-free) hyperbolic case was already handled by Ollivier \cite{Ollivier:PhaseTransition}.
\end{prob}

 An unlikely but more fanciful possibility is to use \Cref{Thm:main} to find new examples of hyperbolic groups that are not cubulated.

\begin{conj}\label{Conj:OptimalCubulateManifold}
Let $M = \HH^n / \Gamma$ be a closed hyperbolic $n$--manifold. Assume that either $n \leq 3$ or that $M$ is arithmetic of simplest type. Then, for every $\lambda > 1$, $M$ is homotopy equivalent to a compact non-positively curved cube complex $X$ such that there is a  $\Gamma$--equivariant $\lambda$--quasiisometry from $\widetilde X$ to $\widetilde M$.
\end{conj}

Recall that the definition of a $\lambda$--quasiisometry appears in \Cref{Eqn:Lambda12}.

\Cref{Conj:OptimalCubulateManifold} is supported by the following intuition. If $M$ is a surface, then there is a plethora of cubulations of $M$ via closed--geodesics. As $\ell \to \infty$, a randomly chosen closed--geodesic $\gamma$ of length approximately $\ell$ is nearly equidistributed in $M$ (as well as in $T^1 M$). Therefore, the number of lifts of $\gamma$ separating a distant pair of points $p,q \in \widetilde M = \HH^2$ should be nearly proportional to $\dist_{\HH^2}(p,q)$. The same intuition applies for $n=3$, where a closed hyperbolic $3$--manifold $M$ has a plethora of cubulations via nearly geodesic surfaces that are similarly nearly equidistributed in $M$. See Kahn and Markovic \cite{KahnMarkovic09}.

Similarly, an arithmetic hyperbolic manifold $M = \HH^n / \Gamma$ of simplest type always contains totally geodesic (codimension-1) hypersurfaces. Since the commensurator  $\operatorname{Comm}(\Gamma)$ is dense in $\Isom(\HH^n)$, one may move a single hypersurface by many elements of $\operatorname{Comm}(\Gamma)$, as in \cite{BergeronHaglundWiseSimple}, to achieve  a cubulation of $\Gamma$  such that, again, $\dist_{\HH^n}(p,q)$ is nearly proportional to the number of hypersurfaces separating $p$ from $q$.

Brody and Reyes~\cite{BrodyReyes:HyperbolicCubulations} have very recently proved \Cref{Conj:OptimalCubulateManifold}, formalizing the intuition in the above discussion.

An ambitious generalization of \Cref{Conj:OptimalCubulateManifold} asks for arbitrarily homogeneous cubulations of more general groups:

\begin{question}\label{Ques:OptimalCubulation}
Let $G$ be a cubulated, hyperbolic group that acts geometrically on a metric space $\Upsilon$. Is it true that for every $\lambda > 1$, there exists a (proper, cocompact) $G$--action on a non-positively curved cube complex $\widetilde X$, admitting a $G$--equivariant $\lambda$--quasiisometry $f \from \widetilde X \to \Upsilon$?
\end{question}

Finally, we pose a probabilistic question about pieces that is prompted by \Cref{Rem:ProbabilisticQITranslation}.

\begin{question}\label{Ques:UpsilonTypicalPieces}
Suppose, as in \Cref{Thm:OtherSample}, that $G = \pi_1 X$, where $X$ is a compact non-positively curved cube complex. Suppose $G$ is hyperbolic, and acts properly and cocompactly on a geodesic metric space $\Upsilon$. Consider a cubical presentation $X^* = \langle X \mid Y_1, \ldots, Y_k \rangle$. 

Let $D = \diameter(X)$. For $\epsilon > 0$, a piece $P$ of $X^*$ is called \emph{$\Upsilon$--typical} if there are points $x,y \in P$ realizing $\dist_{\widetilde X}(x,y) = \diameter(P)$, and an $\Upsilon$--typical group element $g \in G$ such that $\dist_{\widetilde X}(gx, y) \leq D$. Recall that $\Upsilon$--typical group elements are defined in \Cref{Eqn:UpsilonTypical}, using an additive constant $\epsilon$.

Suppose that $[g_1], \ldots, [g_k]$ have been sampled uniformly from among all the conjugacy classes satisfying $\stabletrans{g}_{{}_\Upsilon} \leq \ell$, and  $X^* = \langle X \mid Y_1, \ldots, Y_k \rangle$ is the cubical presentation associated to $G / \nclose{g_1, \ldots, g_k}$. 
What is the distribution of pieces (above a certain threshold of diameter)? 
Is it true that with overwhelming probability, all sufficiently large pieces of $X^*$ are $\Upsilon$--typical?
\end{question}

A positive answer to \Cref{Ques:UpsilonTypicalPieces} would enable one to complete the line of argument outlined in \Cref{Rem:ProbabilisticQITranslation}. We suspect that the answer is ``yes'' for cone-pieces. We are less confident about wall-pieces, because the distribution of wall-pieces may depend on the $\Upsilon$--action of hyperplane stabilizers in $G$.

\bibliographystyle{alpha}
\bibliography{biblio.bib}

\end{document}